\documentclass[12pt]{amsart}

\usepackage{layout} 
\usepackage[top=2.5cm, bottom=1.5cm, left=2cm, right=2cm]{geometry} 
\usepackage[english]{babel}
\usepackage[utf8]{inputenc}
\usepackage[T1]{fontenc}
\usepackage{amsmath}
\usepackage{amssymb}
\usepackage{mathrsfs}
\usepackage{amsthm}
\usepackage{enumerate} 
\usepackage{hyperref}
\usepackage{enumitem}
\usepackage{comment}
\usepackage{mathtools}
\usepackage{esint}
\usepackage{tikz}
\usepackage{faktor}
\usepackage{appendix}
\usepackage{color}
\usepackage[capitalize]{cleveref}

\newtheorem{theorem}{Theorem}[section]
\newtheorem{proposition}[theorem]{Proposition}
\newtheorem{corollary}[theorem]{Corollary}
\newtheorem{lemma}[theorem]{Lemma}
\newtheorem{remark}[theorem]{Remark}
\newtheorem{defi}[theorem]{Definition}
\newtheorem{question}[theorem]{Open Question}
\newtheorem{claim}[theorem]{Claim}

\newcommand{\scal}[2]{\left\langle #1,#2 \right\rangle}
\newcommand{\g}{\nabla}

\newcommand{\lap}{\Delta}
\newcommand{\dr}{\partial}
\newcommand{\vol}{\mathrm{vol}}

\newcommand{\dist}{\mathrm{dist}}

\newcommand{\tr}{\mathrm{tr}}

\newcommand{\Riem}{\mathrm{Riem}}

\newcommand{\proj}{\mathrm{proj}}
\newcommand{\KN}{\mathbin{\bigcirc\mspace{-15mu}\wedge\mspace{3mu}}}

\newcommand{\D}{\mathbb{D}}
\newcommand{\R}{\mathbb{R}}
\newcommand{\C}{\mathbb{C}}
\newcommand{\N}{\mathbb{N}}

\newcommand{\s}{\mathbb{S}}
\newcommand{\T}{\mathbb{T}}

\newcommand{\Hd}{\mathbb{H}}

\newcommand{\Qr}{\mathcal{Q}}
\newcommand{\Yr}{\mathcal{Y}}

\newcommand{\Hr}{\mathcal{H}}
\newcommand{\Fr}{\mathcal{F}}
\newcommand{\Lr}{\mathcal{L}}
\newcommand{\Dr}{\mathcal{D}}

\newcommand{\Er}{\mathcal{E}}
\newcommand{\Cr}{\mathcal{C}}
\newcommand{\Ur}{\mathcal{U}}

\newcommand{\pr}{\mathcal{P}}
\newcommand{\Arond}{\mathring{A}}

\newcommand{\inter}[2]{[\![#1,#2]\!]}
\newcommand{\ust}{\underset}
\newcommand{\para}[1]{\left( #1 \right)}

\newcommand{\bz}{\bar{z}}
\newcommand{\vp}{\varphi}
\newcommand{\ve}{\varepsilon}

\newcommand{\gf}{\mathfrak{g}}

\newcommand{\vn}{\vec{n}}
\newcommand{\deltar}{\mathring{\delta}}

\title[Some global properties of umbilic points of Willmore immersions]{Some global properties of umbilic points of Willmore immersions in the $3$-sphere}
\date{\today}

\author{Nicolas Marque}
\address[Nicolas Marque]{Institut \'Elie Cartan de Lorraine, Université de Lorraine}
\email{nicolas.marque@univ-lorraine.fr}

\author{Dorian Martino}
\address[Dorian Martino]{ETH Zürich, Rämistrasse 101, 8092 Zürich, Switzerland}
\email{dorian.martino@math.ethz.ch}

\begin{document}
	
	\maketitle
	
	\begin{abstract}
		We study the umbilic points of Willmore surfaces in codimension 1 from the viewpoint of the conformal Gauss map. We first study the local behaviour of the conformal Gauss map near umbilic curves and prove that they are geodesics up to a conformal transformation if and only if the Willmore immersion is, up to a conformal transformation, the gluing of minimal surfaces in the 3-dimensional hyperbolic space. Then, we prove a Gauss--Bonnet formula for the conformal Gauss map of Willmore surfaces which turns out to be an asymptotic expansion involving the length of the umbilic curves in the spirit of renormalized volume expansions. We interpret this formula as a unified version for the different expressions of the value of the Willmore energy for conformally minimal surfaces in each space-form.
	\end{abstract}
	
	
	\section{Introduction}
	
	Let $\Sigma$ be a closed Riemann surface. Given an immersion $\Phi \colon \Sigma\to \R^3$, its Willmore energy is defined by
	\begin{align*}
		\Er(\Phi) \coloneqq \int_\Sigma \big|\Arond_\Phi\big|^2_{g_\Phi}\, d\vol_{g_\Phi},
	\end{align*}
	where $\Arond_\Phi$ is the traceless part of the second fundamental form $A_{\Phi}$ of $\Phi(\Sigma)$ and $g_\Phi \coloneqq \Phi^*\delta$ is the metric induced by $\Phi(\Sigma)\subset (\R^3,\delta)$, with $\delta$ the flat metric on $\R^3$. An immersion $\Phi \colon \Sigma\to \R^3$ is said to be Willmore if it is a critical point of $\Er$. Thanks to the Gauss\textendash Bonnet formula, one could equivalently define Willmore surfaces as critical points of the following functional:
	\begin{align}\label{def:W}
		W(\Phi) \coloneqq \int_{\Sigma} H^2_{\Phi}\, d\vol_{g_\Phi},
	\end{align}
	where $H_\Phi = \frac{1}{2}\, \tr_{g_\Phi} A_{\Phi}$ is the mean curvature. Indeed, if $\chi(\Sigma)$ is the Euler characteristic of $\Sigma$, the two functionals $W$ and $\Er$ are linked by the relation 
	\begin{align}\label{eq:EW}
		\Er(\Phi) = 2\, W(\Phi) - 4\pi\, \chi(\Sigma).
	\end{align}
	From an analytic viewpoint, one of the main difficulties to study Willmore surfaces is the fact that the Euler--Lagrange equation of $W$ is a non-linear critical elliptic equation of order four. Nevertheless, it was proven that Willmore surfaces satisfy an $\ve$-regularity property, either extrinsically by Kuwert--Schätzle \cite{kuwert2001}  or intrinsically by Rivière \cite{riviere2008}, showing that Willmore surfaces are real-analytic surfaces.\\
	
	In order to understand Willmore surfaces, it is important to study classification questions. A first observation following from \eqref{def:W} is that minimal surfaces are Willmore surfaces. Moreover, Blaschke \cite{blaschke1955} proved that the quantity $\big|\Arond_\Phi\big|^2_{g_\Phi} d\vol_{g_\Phi}$ is pointwise conformal invariant. Hence, the functional $\Er$ is conformally invariant. By \eqref{eq:EW}, the functional $W$ is also invariant under all the conformal transformations of $\R^3$ preserving the topology of the immersion. It was actually proved in \cite{mondino2018} that $\Er$ is the only conformally invariant functional on surfaces of $\R^3$, up to the addition of a topological constant. Hence, all the conformal transformations of minimal surfaces are also Willmore surfaces. Since the round sphere $\s^3$ and the hyperbolic space $\Hd^3$ are globally conformal to $\R^3$, one can also study surfaces of $\s^3$ and $\Hd^3$ by the means of $\Er$ and show that surfaces of $\R^3$ that are conformal transformations of some minimal surface in $\s^3$ or $\Hd^3$ are also Willmore surfaces. A natural question is to know whether other Willmore surfaces exist or not.\\

	When $\Sigma = \s^2$, Bryant \cite{bryant1984} proved that any smooth critical point of $\Er$ must be a conformal transformation of a minimal surface in $\R^3$ with planar ends. To do so, he considered the conformal Gauss map $Y\colon \Sigma\to (\s^{3,1},\eta)$ of any immersion $\Phi\colon \Sigma\to \R^3$, see \Cref{sec:Preliminaries} for a definition of the conformal Gauss map, where 
	\begin{align*}
		\eta \coloneqq (dx^1)^2 + \cdots +(dx^4)^2 - (dx^5)^2,&  & 	\s^{3,1} \coloneqq \{x\in \R^5 : |x|^2_\eta = 1\}.
	\end{align*}
	He proved that $\Phi$ is a smooth Willmore immersion if and only if $Y$ is a smooth harmonic and conformal map, \textit{i.e.} a branched minimal surface. As a consequence, he showed that the quartic $\Qr \coloneqq \scal{\dr^2_{zz} Y}{\dr^2_{zz} Y}_\eta (dz)^4$ is holomorphic and that it vanishes everywhere if and only if $\Phi$ is a conformal transformation of a minimal surface in $\R^3$. If $\Sigma = \s^2$, the only holomorphic quartic is $0$. Therefore, any smooth Willmore immersion of $\s^2$ must be a conformal transformation of a minimal surface. After some partial extensions to the branched case \cite{lamm2013,michelat2022}, this classification has been extended to all branched Willmore spheres by the second author \cite{martino2024}.\\
	
	When $\Sigma=\T^2$, the situation is much more complicated and only few examples are known to the best of the authors' knowledge. Pinkall \cite{pinkall1985} constructed a family of Willmore tori which are not conformal transformations of minimal surfaces in any of the three space-forms $\R^3$, $\s^3$ or $\Hd^3$. Ferus--Pedit \cite{ferus1990} classified all the $\s^1$-invariant Willmore tori in $\s^3$ by proving that under this symmetry, one obtains a completely integrable system when studying the geodesic flow. Babich--Bobenko \cite{babisch1993} constructed examples of Willmore tori having closed curves of umbilic points by gluing minimal surfaces of $\Hd^3$ along their asymptotic boundaries on $\R^2\times \{0\}$. Rotationally symmetric examples of Willmore surfaces with curves of umbilic points have also been constructed by Dall'Acqua--Schätzle \cite{dallacquaschatzle}, where they prove that all such examples are also gluing of minimal surfaces in $\Hd^3$.\\
	
	In order to develop the intuition on Willmore surfaces in a generic manner, it is important to understand the role of Bryant's quartic when it does not vanish. Palmer \cite{palmer1991} proved that the Bryant's quartic $\Qr$ together with the traceless part $\Arond$ of the second fundamental form of a Willmore immersion $\Phi\colon \Sigma\to \R^3$ provides a full description of the curvature of its conformal Gauss map $Y$ away from the umbilic points of $\Phi$, see \Cref{Gaussbonnetpasdansunecarte} below (see also \cite{eschenburg} for an interpretation of $\Qr$ as a part of the second fundamental form of $Y$). The first author \cite{marque2021} proved that $\Phi$ is a conformal transformation of a minimal surface in one of the three space-form (away from the umbilic set) if and only if the curvature of the normal bundle of $Y$ vanishes. In order to understand global properties of $\Phi$, it seems necessary to study the behaviour of $Y$ near umbilic points, in the same manner that curvature estimates for minimal surfaces in Riemannian spaces are crucial for their understanding. \\
	
	Let $\Phi\colon \Sigma^2 \to \R^3$ be a Willmore immersion from a closed Riemann surface and $\Ur$ be its umbilic set. Bryant \cite{bryant1984} proved that $\Sigma\setminus \Ur$ is dense in $\Sigma$. Later, Schätzle \cite{schatzle2017} proved that $\Ur$ consists in isolated points and real-analytic curves by a direct analysis of the Willmore equation. Since we will need a detailed description of umbilic point we will give an interpretation of his proof using the conformal Gauss map and obtain the following.

	\begin{theorem}\label{th:Umbilic}
		
		Let $\Sigma$ be a closed Riemann surface and $\Phi\colon \Sigma\to \R^3$ be a Willmore immersion, not totally umbilic. Let $\Ur\subset \Sigma$ be the umbilic set of $\Phi$. We can identify four types of umbilic points, labelled type \ref{typeI} to \ref{typeIV} . Then $\Ur$ consists in the disjoint union of two sets $\pr$ and $\Lr$, where:
		
		\begin{itemize}
			
			\item $\pr$ is the union of a finite number of isolated points. They are all of type \ref{typeI} to \ref{typeIII}, and are called singular.
			
			\item $\Lr$ is the union of a finite number of disjoint closed curves with finite lengths and without self-intersections. This is the union of all the points of type \ref{typeIV}. Among those points, there is a finite number of singular points characterized by $\nabla \Arond (p) =0$.
			
		\end{itemize}
		
		The sets of singular umbilic points is thus composed of the points of types \ref{typeI}-\ref{typeIII}, and of the singular umbilic points of type \ref{typeIV}.
		
	\end{theorem}
	
	We discuss  in \Cref{descriptionumbilicpoints} the different types of umbilic points, and explain how our classification depends on their local profile and how it impacts their contribution in the Gauss-Bonnet formula below.  A geometric interpretation of the umbilic curves is provided in \Cref{sec:Umbilic_curves} where we identify two different classes of umbilic curves. Moreover, we study in section \ref{sec:Application} conformally minimal surfaces in the three space-forms and see they have umbilic points of type \ref{typeI}, \ref{typeII} or \ref{typeIV} (the latter only for conformally minimal surfaces in $\Hd^3$).
	A relevant quantity for singular umbilic points will be called their multiplicity as it will play the role of the multiplicity of a branched point in the Gauss--Bonnet formula \cite{lamm2013}. However the zoology of umbilic points is much more varied than that of branch points and the definition will vary according to type.
	
	\begin{defi}
		Let $\Sigma$ be a closed Riemann surface and $\Phi\colon \Sigma\to \R^3$ be a Willmore immersion. 
\begin{itemize}
\item We define the multiplicity of a singular umbilic point $p\in\pr$ as 
$$
	n(p)= \sup\left\{ k\in \N : \limsup_{x\to p} \frac{\big| \Arond\big|_{g_{\Phi}}(x)}{d(x,p)^k} <+\infty \right\}.
$$ 
\item The multiplicity  of a singular umbilic point of type \ref{typeIV} is defined as $$n(p) = \sup\left\{ k \in \N : \partial_z^k \left( \Arond(\partial_z, \partial_z )\right)(0) =0 \right\},$$ where $z$ is a local complex coordinate centered at $p$.
\end{itemize}
	\end{defi}
Once more details are given in subsection  \ref{descriptionumbilicpoints}, with how to compute the multiplicity from the expression of the tracefree second fundamental form in a  local complex chart. In addition, we show in \Cref{lm:KY_pos_circ} that umbilic curves are different from the isolated umbilic points in that the Gaussian curvature of the conformal Gauss map goes to $+\infty$ near umbilic curves, but remains bounded near isolated umbilic points of type \ref{typeI} and \ref{typeII}. One can also wonder whether it is possible to describe the umbilic curves by an equation. Schätzle \cite{schatzle2017} proved that any real-analytic curve in $\R^3$ could be thickened into a piece of Willmore surface. It would be interesting to know whether one can close this surface or not. On the other hand, one can assume that a given umbilic curve satisfies an equation. We prove in a similar manner as in \cite{dallacquaschatzle} that the case of geodesic umbilic curves is very restricting, see \Cref{theogeodbb} for a precise statement and \Cref{th:Conf_BB} for a conformally invariant version. 
	
	\begin{theorem}\label{th:BB_curve}
		Let $\Sigma$ be a closed Riemann surface and $\Phi\colon \Sigma\to \R^3$ be a Willmore immersion not totally umbilic. Assume that there exists a curve $\Cr\subset(\Sigma,g_{\Phi})$ umbilic and geodesic. Then there exists a isometry $f$ of $\R^3$ such that the sets $(f\circ \Phi)(\Sigma)\cap \{x^3>0\}$ and $\left[ - (f\circ \Phi)(\Sigma)\cap \{x^3<0\}\right]$ are minimal surfaces in the 3-dimensional hyperbolic space.
	\end{theorem}
	This results underscores the existence  of global effects on umbilic curves: if one umbilic curve on a Willmore surface is geodesic, it is made of two hyperbolic minimal surfaces halves glued together at the $\{x^3= 0 \}$ plane. As proved in \cite{babisch1993} then the set of umbilic curves is exactly the intersection of the surface and the plane, i.e. the conformal infinity of the minimal halves, that is a union of geodesic curves. Then if one umbilic curve on a Willmore surface is geodesic, they are all geodesics.
	
	Another fundamental and global geometric property is the Gauss--Bonnet formula. Even though the geometry of the conformal Gauss map degenerates around umbilic points, one can wonder whether the integral of its Gauss curvature $K_Y$ is a well-defined quantity. By \Cref{lm:KY_pos_circ}, we know that $K_Y$ is not integrable. However, we compute an asymptotic expansion and obtain a result in the spirit of the renormalized area formula obtained by Alexakis--Mazzeo \cite{alexakis2010}. 
	
	\begin{theorem}\label{th:main}
		Let $\Sigma$ be a closed Riemann surface. Let $\Psi\colon \Sigma\to\s^3$ be a smooth Willmore immersion not totally umbilic, $\Ur\subset \Sigma$ be the umbilic set of $\Psi$. Let $Y\colon \Sigma\to \s^{3,1}$ be its conformal Gauss map and $K_Y$ the Gauss curvature of $Y$. 
		
		Let $h$ be a smooth metric on $\Sigma$ conformal to $g_\Psi$. Let $L^1_h,\ldots,L^J_h$ the length of the closed umblic curves contained in $\mathcal{L} \subset \Ur$ computed with respect to the metric $h$. We denote $p_1, \dots, p_m$ the singular umbilic points away from the umbilic curves and $p_{m+1}, \dots, p_q$ those on umbilic curves. Let $n_1,\ldots,n_q$ denote the multiplicities of the singular umbilic points. For each $p_i$, $i\le m$ we consider local centred complex coordinates and define $\D_{r,i} = \{ |z|(p) < r\}$. Given $\ve>0$, we denote the $\varepsilon$-neighbourhood of $\Ur$ for the metric $h$ by
		$$
		\mathcal{U}_\varepsilon \coloneqq \left( \bigcup_{i=1}^m \D_{\ve,i} \right)\cup \{ p\in \mathcal{L} : d_h(p,\Ur) <\ve \}.
		$$
		Then it holds: 
		$$ 
		\int_{\Sigma\setminus \mathcal{U}_\varepsilon} K_Y\, d\vol_{g_Y} \ust{\ve\to 0}{=} \sum_{k=1}^J \frac{ 2L_h^k }{\ve} + 2\pi\chi(\Sigma) +2\pi \sum_{i=1}^q n_i+ O\big(\varepsilon |\log \varepsilon| \big).
		$$
		This can be reformulated as:
		$$
		\Er(\Psi) = 2\, \lim_{\ve\to 0} \left( \int_{\Sigma\setminus \Ur_{\ve}} \Re\left( 4\Qr\otimes h_0^{-2}\right)\, d\vol_{g_\Psi}  + \sum_{k=1}^J \frac{ 2L_h^k }{\ve} \right)+ 4\pi\chi(\Sigma) +4\pi \sum_{i=1}^q n_i.
		$$
	\end{theorem}
	In the above theorem, the term $O(\varepsilon \left|\log \varepsilon\right|)$ can be replaced by a $O(\varepsilon)$ if none of the singular umbilic points is of type \ref{typeIII}.\\

	As straight-forward consequences of \Cref{th:main}, we obtain that Willmore surfaces for which the Gaussian curvature of the conformal Gauss map is integrable have no umbilic circles. In particular this yields another proof that the set of umbilic points of minimal surfaces in $\R^3$ consists in isolated points. Indeed, since in these cases $\Qr=0$ we have $K_Y=1$ which is integrable.\\
	
	Since we have explicit formulas of the Bryant's quartic for conformally minimal surfaces in each space-form, we compute explicitly the last formula of \Cref{th:main}. It turns out to be a unified version of the formula from Li--Yau \cite{li1982} for the energy of minimal surfaces in $\s^3$ in terms of conformal volume and the one from Alexakis--Mazzeo \cite{alexakis2010} giving the energy of minimal surfaces in $\Hd^3$ using the renormalized area.
	\begin{proposition}\label{pr:Value_Conf_Min}
		Let $\Sigma$ be a closed Riemann surface and $\Phi\colon \Sigma\to \R^3$ be a Willmore immersion. Let $n_1,\ldots,n_q$ be the multiplicities of its singular umbilic points.
		\begin{enumerate}
			\item  If $\Phi$ is a conformal transformation of a minimal surface in $\R^3$, then it holds
			\begin{align*}
				\Er(\Phi) = 4\pi \left( \chi( \Sigma) + \sum_{i=1}^q n_i\right).
			\end{align*}
			
			\item  If $\Phi$ is a conformal transformation of a minimal surface in $\s^3$, we denote $V_c(3,\Phi)$ the conformal volume of $\Phi$. Then, it holds
			\begin{align*}
				\Er(\Phi) = 2\, V_c(3,\Phi) - 4\pi\, \chi(\Sigma).
			\end{align*}
			
			\item  If $\Phi$ is a conformal transformation of an immersion $\zeta\colon \Sigma\to \R^3$ such that $\zeta(\Sigma)\cap \{x^3>0\}$ and $\left[ - \zeta(\Sigma)\cap \{x^3<0\}\right]$ are minimal in $\Hd^3$, we denote $\Ur \coloneqq \{x\in \Sigma : \zeta^3(x)=0\}$ the umbilic set of $\zeta$. It holds
			\begin{align*}
				\Er(\Phi) = -2\, \lim_{\ve\to 0}\left[ \int_{ \{d_{g_{\zeta}}(\cdot,\Ur)>\ve\} } \frac{d\vol_{g_{\zeta}} }{(\zeta^3)^2} - \frac{2}{\ve} \Hr^1\big(\zeta(\Sigma)\cap \{x^3=0\}\big) \right] - 4\pi\, \chi(\Sigma).
			\end{align*}
			In the above formula, $\Hr^1$ denotes the 1-dimensional Lebesgue measure.
		\end{enumerate}
	\end{proposition}
	
	We here stated all cases using the truly conformally invariant Willmore energy $\mathcal{E}$. However cases \textit{(1)} and \textit{(2)} are better known as equalities on $W$. A quick computation done in \Cref{sec:Application} will show that \textit{(1)} amounts to $W(\Phi)$ equals to $4\pi$ times the number of ends of the minimal surface, while \textit{(2)} reduces to $W(\Phi) = V_c(3,\Phi) $.  One can then interpret the following quantity as a renormalized conformal total curvature acting as an extension of the renormalized volume:
	$$ 
	\int_{\Sigma\setminus B_{\ve}(\Ur)} \Re\left( 4\Qr\otimes h_0^{-2}\right)\, d\vol_{g_\Psi}  + \sum_{k=1}^J \frac{ 2L_h^k }{\ve}.
	$$   
	
	\Cref{th:main} raises a few questions. The above \Cref{pr:Value_Conf_Min} shows that \Cref{th:main} can be used for concrete energy computations. It would be interesting to understand whether it has more theoretic applications. For instance, one can ask the following question.
	\begin{question}
		Let $\Sigma$ be a closed Riemann surface and $I\in\N$. Is there a threshold $\Lambda>0$ depending only on $\Sigma$ and $I$ such that any Willmore immersion $\Phi\colon \Sigma\to \R^3$ satisfying $\Er(\Phi)\leq \Lambda$ has at most $I$ umbilic curves?
	\end{question}
	
	It is possible that the existence of umbilic curves is linked to the notion of Morse index. Indeed, Palmer \cite[Theorem 3.5]{palmer1991} showed that if a Willmore immersions with conformal Gauss map $Y$ satisfies the pointwise inequality $|\Qr|_{g_Y}\geq 1$, with strict inequality somewhere, then the Willmore surface is unstable. He applied successfully this criterium to the Hopf tori introduced in \cite{pinkall1985} and provided a new proof of the instability of all the Hopf tori which are not the Clifford torus. Moreover, due to \cite[Equation (2.20)]{palmer1991}, it holds
	\begin{align*}
		|\Qr|_{g_Y}^2 = (1-K_Y)^2 + (K_Y^{\perp})^2,
	\end{align*}
	where $K_Y$ is the Gauss curvature of $Y$ and $K_Y^{\perp}$ is the curvature of the normal bundle of $Y$. Since the Gauss curvature of $Y$ blows up near the umbilic circles, one can ask the following question.
	
	\begin{question}
		Let $\Sigma$ be a closed Riemann surface and $\Phi\colon \Sigma\to \R^3$ be a Willmore immersion. Can we bound from above the number of umbilic circles of $\Phi$ by its Morse index?
	\end{question}
	
	In the same spirit, one can ask about stable Willmore surfaces. Marques--Neves \cite{marques2014b} proved that the Clifford torus is the only minimizer up to conformal transformation among all tori. In higher genus, Kusner \cite{kusner1989} conjectured that the Lawson surfaces are minimizing as well among the surfaces of same genus. A first step toward this conjecture would be to understand stable Willmore surfaces. For instance, one could try to adapt the result of Fischer-Colbrie \cite{fischer1985} stating that for a minimal surface in a positively curved space, one can prove that the $L^2$-norm of its second fundamental form is finite.
	
	\begin{question}
		Let $\Sigma$ be a closed Riemann surface and $\Phi\colon \Sigma\to \R^3$ be a stable Willmore immersion. Let $Y$ be its conformal Gauss map. Can we prove that either the Gauss curvature $K_Y$ or the curvature $K_Y^{\perp}$ of the normal bundle of $Y$ is integrable?
	\end{question}
	
	Focusing on minimal surfaces of $\Hd^3$, we obtain that the formula for the renormalized area is actually a Gauss--Bonnet formula for the conformal Gauss map. It would be interesting to know whether a similar relation still holds for higher dimensional minimal hypersurfaces in the hyperbolic spaces $\Hd^{2d+1}$. This would provide an alternative point of view for the computation of the renormalized volume as introduced in \cite{gover2017,graham1999,robingraham2020}.
	
	\subsubsection*{Structure of the paper.} In \Cref{sec:Preliminaries}, we fix the notations and define the conformal Gauss maps. In \Cref{sec:umbilic_set_willmore_surfaces}, we analyse the umbilic set of Willmore immersions, provide a alternative proof of \Cref{th:Umbilic} and prove \Cref{th:BB_curve}. In \Cref{sec:GaussBonnet}, we prove \Cref{th:main}. In \Cref{sec:Application}, we prove \Cref{pr:Value_Conf_Min} and discuss some consequences.
	
	\subsubsection*{Acknowledgement.} We would like to thank Tristan Rivière and Alexis Michelat for stimulating discussions. We also thank Anna Dall'Acqua for useful references. We are grateful to Reiner M. Schätzle for important clarifications on the structure of the umbilic set and pointing out mistakes in preliminary versions of this article. N.M. is supported by the ANR OrbiScaR ANR-24-CE40-0702. Part of this project was conducted while D.M. was visiting at the Institut Elie Cartan, he would like to thank it its hospitality and its excellent working conditions. D.M. is supported by Swiss National Science Foundation, project SNF 200020\textunderscore 219429.

	\section{Willmore immersions and their conformal Gauss map}\label{sec:Preliminaries}
	
	In the whole article, $\Sigma$ will denote a closed Riemann surface, $\delta$ the Euclidean metric in $\R^3$ and $\deltar$ the round metric in $\s^3$. We recall that one can go from one model to the other using a stereographic projection. 
	
	\subsection{Immersions in $\R^3$}
	Given an immersion $\Phi\colon \Sigma\to \R^3$, we denote $\vn_{\Phi}\colon \Sigma\to \s^2$ its Gauss map and $g_{\Phi}\coloneqq \Phi^*\delta$ the first fundamental form of $\Phi$. Let $A_\Phi$ be its second fundamental form, $H_\Phi$ its mean curvature and $\Arond_\Phi$ the traceless part of the second fundamental form. They are given by the following formulas in local coordinates:
	\begin{align}
		& (A_\Phi)_{ij} \coloneqq -\scal{\dr_i \Phi}{\dr_j \vn_\Phi}_{\delta}, \nonumber\\[1mm]
		& H_\Phi \coloneqq \frac{1}{2} \tr_{g_{\Phi}} (A_\Phi), \nonumber\\[2mm]
		& \big(\Arond_\Phi\big)_{ij} = (A_\Phi)_{ij} - H_\Phi (g_\Phi)_{ij}, \nonumber\\[2mm]
		& \dr_i \vn_\Phi = -H_\Phi\, \dr_i \Phi -(\Arond_\Phi)_i^{\ j}\, \dr_j \Phi. \label{eq:gvn_H_Arond}
	\end{align}
	In \eqref{eq:gvn_H_Arond}, the index $j$ is raised by the metric $g_{\Phi}$. The traceless part $\Arond$ is also described by the following differential form written in complex coordinates:
	\begin{align*}
		\begin{cases}
			h_0 \coloneqq 2\scal{\dr^2_{zz} \Phi}{\vn_{\Phi}}_{\delta} (dz)^2, \\[2mm]
			\vp_{\Phi} \coloneqq \big(\Arond_{\Phi}\big)_{11} - i\, \big(\Arond_{\Phi}\big)_{12} = 2\scal{\dr^2_{zz} \Phi}{\vn_{\Phi}}_{\delta}.
		\end{cases}
	\end{align*}
	The umbilic set of $\Phi$ is defined as
	\begin{align*}
		\Ur \coloneqq \{x\in\Sigma : h_0(x)=0\}.
	\end{align*}
	
	\subsection{Immersions in $\s^3$}
	
	Given an immersion $\Psi\colon \Sigma\to \s^3$, we denote $N_{\Psi}\colon \Sigma\to T_{\Psi}\s^3$ its Gauss map and $g_{\Psi}\coloneqq \Psi^*\deltar$ the first fundamental form of $\Psi$. Let $A_\Psi$ be its second fundamental form, $H_\Psi$ its mean curvature and $\Arond_\Psi$ the traceless part of the second fundamental form. They are given by the following formulas in local coordinates:
	\begin{align}
		& (A_\Psi)_{ij} \coloneqq -\scal{\dr_i \Psi}{\dr_j N_\Psi}_{\deltar}, \nonumber\\[1mm]
		& H_\Psi \coloneqq \frac{1}{2} \tr_{g_{\Psi}} (A_\Psi), \nonumber\\[2mm]
		& \big(\Arond_\Psi\big)_{ij} = (A_\Psi)_{ij} - H_\Psi\, (g_\Psi)_{ij}, \nonumber\\[2mm]
		& \dr_i N_\Psi = -H_\Psi\, \dr_i \Psi -(\Arond_\Psi)_i^{\ j}\, \dr_j \Psi. \label{eq:gvn_H_ArondS}
	\end{align}
	In \eqref{eq:gvn_H_ArondS}, the index $j$ is raised by the metric $g_{\Psi}$. The quantity $h_0$ and $\vp_{\Psi}$ are defined in a similar manner:
	\begin{align*}
		\begin{cases}
			h_0 \coloneqq 2\scal{\dr^2_{zz} \Psi}{N_{\Psi}}_{\delta} (dz)^2, \\[2mm]
			\vp_{\Psi} \coloneqq \big(\Arond_{\Psi}\big)_{11} - i\, \big(\Arond_{\Psi}\big)_{12} = 2\scal{\dr^2_{zz} \Psi}{N_{\Psi}}_{\delta}.
		\end{cases}
	\end{align*}
	
	\subsection{Conformal Gauss map for immersions in $\R^3$}
	For an introduction to conformal Gauss maps, see for instance \cite{hertrich-jeromin2003,marque2021}. By definition, the conformal Gauss map $Y$ of an immersion $\Phi\colon \Sigma\to\R^3$ is given by
	\begin{align}\label{def:CGM}
		Y := H_\Phi \begin{pmatrix}
			\Phi\\[1mm] \displaystyle \frac{|\Phi|^2-1}{2} \\[3mm] \displaystyle \frac{|\Phi|^2+1}{2}
		\end{pmatrix} + \begin{pmatrix}
			\vn_{\Phi}\\[2mm]
			\scal{\vn_{\Phi}}{\Phi}_{\delta}\\[2mm]
			\scal{\vn_{\Phi}}{\Phi}_{\delta}
		\end{pmatrix}.
	\end{align}
	In particular, it holds $H_\Phi = Y_5-Y_4$. The map $Y$ is valued in the de Sitter space $\s^{3,1}$, defined as follows.\\
	\begin{align*}
		\left\{
		\begin{aligned}
			& \R^{4,1} \coloneqq (\R^5,\eta),\\
			& \eta \coloneqq (dx^1)^2 + \ldots + (dx^4)^2-(dx^5)^2,\\
			& \s^{3,1} \coloneqq \{ x\in \R^{4,1} : |x|^2_{\eta} = 1\}.
		\end{aligned}
		\right.
	\end{align*}
	A direct computation yields 
	\begin{align}\label{eq:gY}
		\g Y = (\g H_\Phi)\begin{pmatrix}
			\Phi\\[1mm] \displaystyle \frac{|\Phi|^2-1}{2} \\[3mm] \displaystyle \frac{|\Phi|^2+1}{2}
		\end{pmatrix} - \Arond_\Phi \begin{pmatrix}
			\g \Phi\\[2mm] \scal{\Phi}{\g \Phi}_{\delta}\\[2mm] \scal{\Phi}{\g \Phi}_{\delta}
		\end{pmatrix}.
	\end{align}
	We have denoted $(\Arond_\Phi \g \Phi)_i = (\Arond_\Phi)_{ij}\g^j \Phi$, where the index is raised with the metric $g_\Phi$. Given any immersion $\Phi$, the conformal Gauss map is conformal to $\Phi$. Indeed it holds 
	\begin{align}\label{eq:Yconformal}
		Y^*\eta = \frac{\big|\Arond_\Phi \big|^2_{g_\Phi}}{2}\, g_\Phi.
	\end{align}
	Furthermore, $\Phi$ is Willmore if and only if $Y$ is harmonic. By \eqref{eq:Yconformal}, we obtain that $Y$ is a space-like immersion away from the umbilic set $\Ur\subset \Sigma$ of $\Phi$. Hence, one can compute its curvature. It is proved in \cite{palmer1991} that the curvature of $Y$ is mainly described by the Bryant's quartic:
	\begin{align*}
		\Qr \coloneqq \scal{\dr^2_{zz} Y}{\dr^2_{zz} Y}_{\eta} (dz)^4.
	\end{align*}
	We now state the relation between $\Qr$ and the curvature of $Y$, see for instance \cite[Equation (2.19)]{palmer1991}. For completeness, the proof of this result is detailed in \Cref{sec:curvatureY}.
	\begin{proposition}
		\label{Gaussbonnetpasdansunecarte}
		Let $\Phi\colon \Sigma \to \R^3$ be a smooth Willmore immersion, $g_{\Phi}$ its induced metric, $h_0$ its Weingarten tensor, $Y$ its conformal Gauss map, $\mathcal{Q}$ its Bryant's quartic, $K_Y$ its Gauss curvature and $K_Y^\perp$ the Gauss curvature of its normal bundle $N(Y)$.
		Then, the following relation holds away from the umbilic set of $\Phi$:
		$$
		4 \mathcal{Q} \otimes h_0^{-1} \otimes h_0^{-1}\,  d\mathrm{vol}_{g_{\Phi}} = \left( 1- K_Y +i\, K_Y^\perp \right)\, d\mathrm{vol}_{g_Y}.
		$$
	\end{proposition}

	\subsection{Conformal Gauss map for immersions in $\s^3$}
	We will also use the representations in $\s^3$. Given an immersion $\Psi\colon \Sigma\to\s^3$ with mean curvature $H_\Psi$ and Gauss map $N_\Psi$, the conformal Gauss map $Y$ of $\Psi$ is given by 
	\begin{align*}
		Y &= H_\Psi\begin{pmatrix}
			\Psi\\ 1 
		\end{pmatrix} + \begin{pmatrix}
			N_\Psi\\ 0
		\end{pmatrix}.
	\end{align*}
	The derivatives of $Y$ are given by
	\begin{align}\label{eq:gY_s3}
		\g Y &= (\g H_\Psi) \begin{pmatrix}
			\Psi\\ 1
		\end{pmatrix} - \Arond_\Psi\begin{pmatrix}
			\g \Psi\\ 0
		\end{pmatrix},
	\end{align}
	where $\Arond_\Psi$ is the second fundamental form of $\Psi$ and $(\Arond_\Psi\g \Psi)_i = (\Arond_\Psi)_{ij}\g^j \Psi$, with the index raised by the metric $\Psi^*\deltar$. If $\pi\colon \s^3\setminus\{\text{north pole} \} \to \R^3$ is the stereographic projection, and $\Phi = \pi\circ\Psi$, then the two definitions of the conformal Gauss map of $\Phi$ and $\Psi$ coincide, see for instance \cite[Equation (72)]{marque2021}:
	\begin{align}\label{eq:equality_CGM}
		H_\Phi \begin{pmatrix}
			\Phi\\[1mm] \displaystyle \frac{|\Phi|^2-1}{2} \\[3mm] \displaystyle \frac{|\Phi|^2+1}{2}
		\end{pmatrix} + \begin{pmatrix}
			\vn_{\Phi}\\[2mm]
			\scal{\vn_{\Phi}}{\Phi}_{\delta}\\[2mm]
			\scal{\vn_{\Phi}}{\Phi}_{\delta}
		\end{pmatrix} = H_\Psi \begin{pmatrix}
			\Psi\\ 1
		\end{pmatrix} + \begin{pmatrix}
			N_\Psi\\ 0
		\end{pmatrix}.
	\end{align}

	\section{Umbilic points of Willmore surfaces}\label{sec:umbilic_set_willmore_surfaces}

	In the whole section, $\Sigma$ will denote a closed Riemann surface.
	
	\subsection{Structure of the umbilic set}
	Given a Willmore immersion $\Psi\colon \Sigma\to \s^3$ not totally umbilic, we define $\Ur =  \{ h_0 = 0\} \subset \Sigma$ to be the set of umbilic points of $\Psi$. Bryant \cite{bryant1984} proved that in this case $\Sigma\setminus\Ur$ is dense in $\Sigma$. Later, Schätzle \cite[Theorem 3.1]{schatzle2017} proved that the umbilic set is a disjoint union of isolated points and closed curves without self-intersections. Here we will provide a new viewpoint on the proof based on the analysis of branched points for conformal maps from \cite[Lemmas 2.1 and 2.2]{eschenburg1988} applied directly to the conformal Gauss map. We will use the complex notations that can be found in \Cref{sec:Complex_notations}.
	
	\begin{theorem}{\cite[Theorem 3.1]{schatzle2017}}  \label{th:structure_umbilic_set}
		Let $\Sigma$ be a closed Riemann surface and $\Psi\colon \Sigma \to \s^3$ be a conformal Willmore immersion. If $\Psi$ is not totally umbilic, then the set $\Ur \subset \Sigma$ of umbilic points of $\Psi$ is closed and consists in a union of a finite number of isolated points and disjoint closed real-analytic curves of finite length without any self-intersection.
	\end{theorem}

	We first show that, in complex coordinates on the unit disk $\D\subset \C$, the map $\Yr \coloneqq \dr_z Y$ satisfies a system of the form $\dr_{\bz} \Yr = M\Yr$ for some $C^{\infty}$-map $M$ that is matrix-valued in \Cref{lem:syst_order1_Y}. Then we prove in \Cref{lem:divide_zm} that the zeros of $\Yr$ are polynomial in $z$. The conclusion will follow from a critical points analysis  of solutions of this system. In Section \ref{descriptionumbilicpoints}, we will discuss the disjunction of cases that leads to knowing whether an umbilic point is isolated or lies in a curve of umbilic points.
	
	\begin{lemma}\label{lem:syst_order1_Y}
		There exists a map $M\in C^{\infty}(\D ;\C^{5\times 5})$ such that the map $\Yr \coloneqq \dr_z Y$ satisfies
		\begin{align}\label{eq:syst_conformal}
			\dr_{\bz} \Yr = M\Yr.
		\end{align}
	\end{lemma}
	\begin{proof}[Proof]
		Consider complex coordinates $\D = B_1(0)\subset \C$ around a given point in $\Sigma$. Let $\vp \coloneqq \Arond_{11}-i\Arond_{12}$. We denote $\lambda$ is the conformal factor: $\Psi^*\xi = e^{2\lambda}(dx^2 + dy^2)$, $H$ is the mean curvature of $\Psi$, and the vector $\nu$ is given by
		\begin{align*}
			\nu \coloneqq \begin{pmatrix}
				\Psi\\ 1
			\end{pmatrix}.
		\end{align*}
		The derivatives of $Y$ are given by \eqref{expressionYz}:
		\begin{align}\label{eq:gY_complex}
			\Yr = \dr_z Y = (\dr_z H)\, \nu - \vp\, e^{-2\lambda}\, \dr_{\bz} \nu.
		\end{align}
		Given $a,b\in \R^5$, we denote $a\otimes b$ the matrix given by $(a\otimes b)^{ij} = a^i b^j$. Let $\eta$ be the matrix of the Lorentz product in $\R^{4,1}$:
		\begin{align*}
			\eta = \begin{pmatrix}
				1 & 0 & 0 & 0 & 0\\
				0 & 1 & 0 & 0 & 0\\
				0 & 0 & 1 & 0 & 0\\
				0 & 0 & 0 & 1 & 0\\
				0 & 0 & 0 & 0 & -1\\
			\end{pmatrix}.
		\end{align*}
		Since $Y$ is harmonic, we obtain from \eqref{eq:Y_harmonic} that
		\begin{align*}
			\dr_{\bz} \Yr = \dr^2_{z\bz} Y = - \frac{1}{2}|\vp |^2 e^{-2\lambda} Y = -Y \langle \partial_{\bz} Y, \partial_z Y \rangle_\eta = \Big( Y\otimes( -\dr_{\bz} Y\eta) \Big) \dr_z Y.
		\end{align*}
		We let $M\coloneqq  Y\otimes( -\dr_{\bz} Y\, \eta)$.
	\end{proof}
	
	Since Willmore surfaces are real-analytic, so is the map $\dr_z Y$. Thus it cannot vanish at infinite order unless $\Phi(\Sigma)$ is a round sphere.  We will now show that there exists $m\in\N$ such that the map $\Yr_0(z) \coloneqq z^{-m} \Yr(z)$ satisfies $\Yr_0(0) \neq 0$.
	
	\begin{lemma}\label{lem:divide_zm}
		Let $\Yr\colon \D\to \C^5$ be a solution to \eqref{eq:syst_conformal}. Assume that $\Yr$ has a zero of order $m\in\N$ at the origin, i.e. we have $\g^j \Yr(0) = 0$ for any integer $0\leq j<m$ and $\g^m \Yr(0) \neq 0$. Then, the map $\Yr_0(z)\coloneqq z^{-m} \Yr(z)$ is real-analytic and satisfies $\Yr_0(0) \neq 0$.
	\end{lemma}
	
	\begin{proof}[Proof]
		From \eqref{eq:gY_complex}, one can see that $\Yr$ is null only if $\varphi=0$ everywhere, that is if $\Psi$ is totally umbilic, which is excluded. Since Willmore surfaces are real-analytic,  we know that $m$ exists. By induction on $k\in\N^*$, we first show that for any integer $k\geq 1$ and $0\leq j<m$, it holds 
		\begin{align}\label{eq:taylor_y}
			\dr_{\bz}^k \g^j \Yr(0) = 0.
		\end{align}
		This follows from the system \eqref{eq:syst_conformal}. Indeed, it holds
		\begin{align*}
			\dr_{\bz}^k \g^j \Yr(0) = \dr_{\bz}^{k-1} \g^j (M\Yr)(0).
		\end{align*}
		Leibniz formula ensures that the right hand side can be written as a linear combination of $\partial_{\bz}^p \nabla^q \Yr(0)$, with $p \le k-1$, $q\le j \le m-1$. If \eqref{eq:taylor_y} stands for all $p\le k-1$, it must then stand for $k$. In addition, for $k=1$, $p=0$ and since by definition of $m$, we have $\nabla^q Y(0)=0$ for any $q\le m-1$, we also have $\partial_{\bz} \nabla^j \mathcal{Y}(0) = 0$. This shows \eqref{eq:taylor_y} by induction.\\ 
		
		Since $\Yr$ is real-analytic, it is equal to its Taylor expansion:
		\begin{align*}
			\forall z\in \D,\qquad \Yr(z) = \sum_{k+l\geq m} \frac{\dr_{\bz}^l \dr_z^k \Yr(0)}{k!\, l!}z^k\, \bz^l.
		\end{align*}
		By \eqref{eq:taylor_y}, we obtain that the map $(z\mapsto z^m)$ divides $\Yr$:
		\begin{align*}
			\forall z\in \D,\qquad \Yr(z) = \sum_{\substack{k\geq m\\ l\geq 0}} \frac{\dr_{\bz}^l \dr_z^k \Yr(0) }{k!\, l!}z^k\, \bz^l\,   = z^m \sum_{k,l\geq 0} \frac{ \dr_{\bz}^l \dr_z^{k+m} \Yr(0) }{(k+m)!\, l!}z^k\, \bz^l.
		\end{align*}
	\end{proof}
	
	We now conclude the proof of Theorem \ref{th:structure_umbilic_set}. 
	
	\begin{proof}[Proof of Theorem \ref{th:structure_umbilic_set}]
		Around a given point $p_0\in\Sigma$, we consider complex coordinates on the unit disk $\D$. Let $\Yr \coloneqq \dr_z Y$. Thanks to \Cref{lem:syst_order1_Y} and \Cref{lem:divide_zm}, there exists $m\in\N$ such that the map $\Yr_0(z) \coloneqq z^{-m} \Yr(z)$ is real-analytic on $\D$ with $\Yr_0(0) \neq 0$. By definition of $\Yr$, it holds
		\begin{align}\label{eq:arond_around_umbilic_points}
			\vp = -2\scal{\Yr}{\dr_{z}\nu}_\eta =-2 z^m \scal{\Yr_0}{\dr_{z} \nu}_\eta.
		\end{align}
		Therefore, the set of umbilic points is given by $\{0\}\cup\{\scal{\Yr_0}{\dr_{z} \nu}_\eta=0\}$ if $m\neq 0$ and by $\{\scal{\Yr_0}{\dr_{z} \nu}_\eta=0\}$ if $m=0$. If $\scal{\Yr_0}{\dr_{z} \nu}_\eta$ does not vanish at the origin, then the origin is an isolated zero.	If $\scal{\Yr_0}{\dr_{z} \nu}_\eta$ vanishes at the origin, then we claim that $\dr_{\bz}\scal{\Yr_0}{\dr_{z} \nu}_\eta (0)\neq 0$. Indeed, in that case, the map $z^{-m} \dr_z H$ has a nonzero limit around the origin:
		\begin{align*}
			z^m \Yr_0 &= (\dr_z H)\, \nu - \vp\, e^{2\lambda}\, (\dr_{\bz} \nu)  \\[1mm]
			& = (\partial_z H )\, \nu + 2\, z^m \scal{\Yr_0}{\dr_{z} \nu}_\eta (\partial_{\bz} \nu ) \\[1mm]
			& \ust{z\to 0}{=}(\dr_z H)\, \nu + o(z^m).
		\end{align*}
		Since $\Yr_0(0)\neq 0$, we obtain that $\displaystyle{\lim_{z\to 0}} \left( z^{-m} \dr_z H\right) \neq 0$. We now write the Gauss--Codazzi equation as $\dr_{\bz}\vp ={e^{2\lambda}} \dr_z H$  (see for instance \cite[Equation (63)]{marque2021}). Hence, we also have $\displaystyle{\lim_{z\to 0}} \left( z^{-m} \dr_{\bz}\vp\right) \neq 0$. From \eqref{eq:arond_around_umbilic_points}, we deduce that 
		\begin{align}\label{eq:dbar_phi_renorm}
			\dr_{\bz}\scal{\Yr_0}{\dr_{\bz} \nu}_\eta(0) \neq 0.
		\end{align}
		By the implicit function theorem, we conclude that the set $\{\Re(\vp) = 0\}$ or the set $\{\Im(\vp) = 0\}$ is a real-analytic graph containing the origin. Hence, the set $\{\vp = 0\} = \{\Re(\vp) = 0\} \cap \{\Im(\vp) = 0\}$ is contained in a real analytic graph. On the other hand, since $\vp$ is real-analytic, the set $\{\vp = 0\}$ is the union of isolated points and real-analytic curves, see for instance \cite[Theorem 6.3.3]{krantz2002}. Therefore, either $0$ is isolated or $\{\vp=0\}$ is a real-analytic graph containing the origin in its interior. \\
		
		In particular, there can be no intersections of curves in $\{\vp = 0\}$. Indeed, assume that $C_1,C_2\subset \{\vp =0\}$ are two curves meeting at a point $p\in C_1\cap C_2$, then the set $\{\vp=0\}$ is a graph around $p$. Hence, $C_1=C_2$ around $p$.  Consequently, we just proved that the set $\{\vp=0\}$ is the disjoint union of isolated points and real-analytic curves without self-intersection. Since $\vp$ is continuous, the set $\{\vp=0\}$ is a closed subset of a compact set, namely $\Sigma$. Hence, the set $\{\vp=0\}$ is compact as well. Consequently, the distance between two disjoint curves in $\{\vp = 0\}$ is strictly positive, and any umbilic curve is necessarily closed. \\
		
		Let $C\subset \{\vp = 0\}$ be a closed curve of umbilic points. The formula \eqref{eq:arond_around_umbilic_points} implies that a point $p\in\{\vp=0\}$ belongs to a curve of umbilic points only when $\scal{\Yr_0}{\dr_{\bz} \nu}_\eta=0$ at the point $p$. In this case, the curve $C$ can be parametrized around $p$ by a graph of a function $f$, thanks to the implicit function theorem. Hence the relation \eqref{eq:dbar_phi_renorm} and the $C^1$ regularity of $\scal{\Yr_0}{\dr_{\bz}\nu}_\eta$  show that the norm of $\g f$ is bounded around $p$. Thus, the graph of $f$ has finite length around $p$. Consequently, for each point $q\in C$, there exists a radius $r_q>0$ such that $B_\Sigma(q,r_q)\cap C$ has finite length. Thus, we obtain a covering by open sets:
		\begin{align*}
			C \subset \bigcup_{q\in C} B_\Sigma(q,r_q).
		\end{align*}
		Since $C$ is compact, we can extract a finite number of balls $\{B_i\}_{1\leq i\leq I} \subset \{ B(q,r_q)\}_{q\in C}$ such that 
		\begin{align*}
			C\subset \bigcup_{i=1}^I B_i.
		\end{align*}
		This shows that the curve $C$ has finite length:
		\begin{align*}
			L(C) \leq \sum_{i=1}^I L(C\cap B_i) \leq I\left( \max_{1\leq i\leq I} L(C\cap B_i) \right)<+\infty.
		\end{align*}
	\end{proof}

\subsection{Local description of umbilic points}
\label{descriptionumbilicpoints}
	As a complement to the proof of \Cref{th:structure_umbilic_set}, we will give explicit local descriptions of umbilic points that we will use in the proof of the Gauss--Bonnet formula. Let $p \in \Sigma$ be an umbilic point. From  \eqref{eq:arond_around_umbilic_points}, and \eqref{eq:dbar_phi_renorm}, one can find complex coordinates on a small disk $\D_{r_0} = \{ |z|<r_0\}$ around $p$ in which 
$$\varphi = C z^m \mathcal{F},$$
with $C \in \C$, $m \in \N$ and $\mathcal{F}$ a  real-analytic function converging on $\D_{r_0}$ and satisfying $\partial_{\bz} \mathcal{F}(0) \neq 0$.
We write the expansion for $\mathcal{F}$:
$$\mathcal{F}(z) \coloneqq \sum_{k=0}^{+\infty} \sum_{l=0}^k f_{k-l, l} z^{k-l} \bar{z}^l= f_{0,0} + f_{1,0}z + f_{0,1}\bz + \sum_{k=2}^{+\infty} \sum_{l=0}^k f_{k-l, l} z^{k-l} \bar{z}^l,$$
where $f_{0,1} = \partial_{\bz} \mathcal{F}(0) \neq 0$. We will use the expansion in the following way:
$$ \mathcal{F}(z) \coloneqq f_{0,1} \left( a + bz + \bz + \mathcal{G}(z) \right),$$ with $a \coloneqq \frac{f_{0,0}}{f_{0,1}}$, $b \coloneqq \frac{f_{1,0}}{f_{0,1}}$ and $\mathcal{G}(z) \coloneqq \frac{1}{f_{0,1}} \sum_{k=2}^{+\infty} \sum_{l=0}^k f_{k-l, l} z^{k-l} \bar{z}^l$. Up to rotating the coordinates, we can assume $b$ to be real. Indeed, if $b = |b|e^{i\beta}$ for some $\beta\in\R$, we define another complex coordinate on $\D_{r_0}$ as $\tilde z = e^{i \frac{\beta+\pi}{2}} z$. Then, one has $\partial_{\tilde z \tilde z } \Psi = e^{-i (\beta+\pi)} \partial_{zz} \Psi$, and thus 
$$
	\begin{aligned}
\tilde \varphi (\tilde z) &= 2 \langle \partial_{\tilde z \tilde z } \Psi, \tilde n \rangle = e^{-i (\beta+\pi)} \varphi \left[ z( {\tilde z} ) \right] \\[3mm]
&= C{\tilde z}^m f_{0,1} e^{-i(\frac{m}{2}+1)(\beta+\pi)} \left( a + b e^{-i\frac{\beta+\pi}{2}} \tilde z + e^{i \frac{\beta+\pi}{2}} \bar{\tilde z} + \mathcal{G}(e^{-i \frac{\beta+\pi}{2}}\tilde z )\right) \\[3mm]
 &=C{\tilde z}^m f_{0,1} e^{-i(\frac{m}{2}+1)(\beta+\pi)} e^{i \frac{\beta+\pi}{2}}\left( a e^{-i \frac{\beta+\pi}{2}} + b e^{-i(\beta+\pi)} \tilde z +  \bar{\tilde z} +e^{-i \frac{\beta+\pi}{2}} \mathcal{G}(e^{-i \frac{\beta+\pi}{2}}\tilde z )\right) \\[3mm]
&=
-2i C{\tilde z}^m f_{0,1} e^{-i(\frac{m}{2}+1)(\beta+\pi)} e^{i \frac{\beta+\pi}{2}}\left(\frac{i}{2} a e^{-i \frac{\beta+\pi}{2}} + \frac{ |b|  \tilde z -  \bar{\tilde z} }{2i} +\frac{i}{2}e^{-i \frac{\beta+\pi}{2}} \mathcal{G}(e^{-i \frac{\beta+\pi}{2}}\tilde z )\right).
\end{aligned}$$
We will simplify notations and keep denoting $C \in \C$ the constant factor, $a \in \C$ the constant term in the analytic function, $b\in \R_+$ the $z$ term and $\mathcal{G}$ the analytic function of order bigger than $2$. We have thus shown that  one can choose a local complex coordinate (obtained by a rotation of the starting one) such that:
\begin{equation}
\label{formenormalisedevarphi1} 
\begin{cases} 
	\displaystyle \varphi(z) = Cz^m \left( a + \frac{bz - \bz}{2i} + \mathcal{G}(z)\right),\\[2mm]
	a,C \in \C,\\[2mm]
	b \in [0,+\infty).
\end{cases} 
\end{equation}
The nature of the umbilic point depends on the values of $a$ and $b$: 
\begin{enumerate}[label=(\Roman*)]
\item \label{typeI} If $a \neq 0$, then for $|z|\leq r_1$ for some $r_1>0$ small enough depending on $b$ and $\|\frac{\mathcal{G}}{r}\|_{L^\infty}< +\infty$, one has 
$$ 
\left| \frac{\varphi}{Cz^m} \right| \ge |a| - \left|\frac{bz - \bz}{2i} + \mathcal{G}(z))\right| \ge \frac{|a|}{2}.
$$ 
The  point $p$ is then umbilic if $m\ge1$, and is then isolated on a disk $\D_{r_1}$, and its multiplicity is $n(p)=m $. We will also need to control  $\left|\partial_r\left( \frac{\varphi}{Cz^m}\right)\right| \le \frac{b+1}{2} + \left|\partial_r \mathcal{G}\right| \le C_1$, which depends on $b$, $r_0$ and $\|\partial_r\mathcal{G}\|_{L^\infty}<+\infty$.
\item \label{typeII} If $a= 0$ and $b \neq 1$, then for all $r\le r_1$ which depends on $|b-1|$ and $\left\| \frac{\mathcal{G}(z)}{r^2}\right\|_{L^\infty}< +\infty$, one has 
	$$
		\begin{aligned}
			\left|\frac{\varphi}{Cz^m r}\right| & \ge \frac{| b e^{2i\theta}- 1|}{2} - r\frac{2| \mathcal{G}(z)|}{r^2} \\[3mm]
			& \ge \frac{|b-1|}{2} - 2r \left\| \frac{\mathcal{G}(z)}{r^2}\right\|_{L^\infty} \ge \frac{|b-1|}{4}.
		\end{aligned}
	$$
The point $p$ is then umbilic, isolated and of multiplicity $n(p)=m+1$. As before we will need a control on $\partial_r \left( \frac{\varphi}{Cz^mr}\right)$:
$$\begin{aligned}
\left| \partial_r \left( \frac{\varphi}{Cz^mr}\right) \right| &\le  \left| \frac{\partial_r \mathcal{G}(z)}{r} \right|  + \left| \frac{ \mathcal{G}(z)}{r^2} \right| \le C_1,
\end{aligned}$$
where the constant $C$ depends on $\left\| \frac{\mathcal{G}(z)}{r^2}\right\|_{L^\infty}, \left\| \frac{\partial_r\mathcal{G}(z)}{r}\right\|_{L^\infty}< +\infty$.
\item \label{typeIII} If $b=1$, then umbilic curves may appear. Indeed writing $z = x+ i y$, and defining $A(z) \coloneqq \Re (\mathcal{G}(z))$ and  $B(z) \coloneqq \Im(\mathcal{G}(z))$, one has, from \eqref{formenormalisedevarphi1}:
	$$
	\varphi(z) = Cz^m (y+A(z) + i B(z)).
	$$ 
	Since $|A(z)|\leq |\mathcal{G}(z)|\leq C|z|^2$, we have 
	\begin{align}\label{eq:cond_curve}
		\partial_y( y +A(z))(0)= 1.
	\end{align}
	Thanks to the implicit function theorem, the set $\{ y+A(z)=0\}$ is an analytic graph in a disk $\D_{r_1}$ ($r_1$ depending on $A$), which we can parametrize by $Z(t) = t + i\, U(t)$, for some real-analytic function $U(t) = \sum_{k=2}^{+\infty} u_k t^k$ defined for $|t|\le t_1$. Then, the function $B\circ Z$ is a real-analytic function on an interval, which thus has isolated zeros or is null everywhere. In the first case, since $\{ \varphi = 0 \} = \{y + A(z)=0\} \cap \{ B(z)=0\}$, $p$ is an isolated umbilic point. By \eqref{eq:cond_curve}, the integer $n(p)\coloneqq m+1$ is the maximal integer such that $|\vp| r^{-m-1}$ is bounded near the origin and thus, its multiplicity is $n(p)= m+1$.
\item \label{typeIV} Finally, in the second case in the above alternative we have
	\begin{align*}
		\{ \varphi = 0 \} =\{y + A(z)=0\} = \{Z(t) : |t|\leq t_1\}.
	\end{align*}
 Hence, $p$ is a point of an umbilic curve, and is a singular umbilic point if and only if $m>0$, in which case it is isolated among the set of singular umbilic points. Its multiplicity is then $n(p)=m$. Whether or not $p$ is singular, the function $\Fr(z) \coloneqq y+A(z) +i B(z)$ satisfies $\Fr(Z(t))=0$ and thus $\partial_{Z'(t)}\Fr(Z(t))=0$. Since, $\nabla (y+A) \neq 0$ along $\mathrm{Im}(Z)$, one has, denoting $\nu(t)$ a normal to $Z'(t)$:

$$\left| \partial_{\nu(t)} \Fr (Z(t))\right| \ge \left| \partial_{\nu(t)} (y+A) (Z(t)) \right| \ge C_p$$
on $\D_{r_0}$.
\end{enumerate}

We  sum up those conclusions in the following remark.
\begin{remark}
\begin{itemize}
\item For isolated umbilic points $p$ of types \ref{typeI}-\ref{typeII} there exists local complex coordinates on $\D_{r_0}$ and constants $c_p,C_p>0$ such that 
\begin{equation}
\label{typeI-IIest}
\left\{
\begin{aligned}
&|\varphi|(z) = Cr^{n(p)} | \mathcal{F}(z)|, \\[2mm]
& |\mathcal{F}(z)| \ge c_p, \\[2mm]
& |\partial_r \mathcal{F}(z)| \le C_p.
 \end{aligned}
\right.
\end{equation}

\item For  any umbilic point $p$ on a curve, that is of type \ref{typeIV}, there exists a real-analytic 1-dimensional submanifold $\Gamma$ and a neighborhood $\mathcal{V}_p$ of $p$ on which there exists complex coordinates $z$, $n \in \N$, and $C_p>0$ such that on $\mathcal{V}_p$: 
\begin{equation}
\label{typeIVest} \left\{ \begin{aligned}
&|\varphi|(z) = |z|^n |\mathcal{F}|(z) \\[2mm]
&\mathcal{F}_{|\Gamma} = 0, \quad \partial_\tau \mathcal{F}_{|\Gamma}=0, \\[2mm]
& \partial_\nu \mathcal{F}_{|\Gamma} \ge C_p >0,
\end{aligned} \right.
\end{equation}
where $\tau$ and $\nu$ are respectively  tangent and normal vectors to $\Gamma$.
\end{itemize}
\end{remark}


	\subsection{Local geometry of the conformal Gauss map near umbilic circles}
	
	\subsubsection{Gauss curvature of the conformal Gauss map.} 
	We first consider the local impact of umbilic curves on the geometry of $Y$. We prove in \Cref{lm:KY_pos_circ} that the Gaussian curvature of $Y$ tends to $+\infty$ near umbilic circles. The proof of \Cref{lm:KY_pos_circ} will also enlighten the behaviour of the Gaussian curvature near isolated umbilic points, as explained in \Cref{rk:KY_isolated}. Then, we discuss two classes of umbilic curves.
	
	\begin{proposition}\label{lm:KY_pos_circ}
		Consider $\Psi\colon \Sigma\to \s^3$ a smooth Willmore surface not totally umbilic. Assume that $\Cr\subset \Sigma$ is a closed circle of umbilic points of $\Psi$. Let $x_0\in \Cr$. Then we have 
		\begin{align*}
			\frac{|\Arond|^2_{g_{\Psi}}(p) }{2}K_Y(p) \ust{\substack{ p\to x_0\\[1mm] p\in \Sigma\setminus \Cr}}{=} \frac{1}{\dist_{g_{\Psi}}(x_0,p)^2}\Big(1 + o(1) \Big).
		\end{align*}
		In particular, it holds $K_Y\to +\infty $ near $\Cr$.
	\end{proposition}

	\begin{proof}
		We consider complex coordinates near $x_0$. Thanks to \eqref{formenormalisedevarphi1}, there exists a function $\vp_0$ that vanishes only on $\Cr$ and such that
		\begin{align*}
			|\Arond|^2_{g_{\Psi}} = 2e^{4\lambda} |z|^{2m}|\vp_0(z)|^2.
		\end{align*}
		Hence, we obtain
		\begin{align}
			\lap_{g_{\Psi}} \log |\Arond|^2_{g_{\Psi}} & = 4\, e^{-2\lambda}\, \dr_{z\bz}\Big(4\log(\lambda) + 2m\log(|z|) + \log|\vp_0|^2\Big) \nonumber \\[2mm]
			& = 16\, \lap_{g_{\Psi}} \log(\lambda) + 8\pi\, m\, e^{-2\lambda(0)} \delta_0 + \lap_{g_{\Psi}} \log |\vp_0|^2. \label{eq:laplacian_conf_fact}
		\end{align}
		The first term is bounded across $\Cr$. The second term is integrable on the disk. To compute the last term, we use that $\dr_{\bz}\vp_0(0)\neq 0$. Hence we consider the Fermi coordinates $(r,\theta)$ of the tubular neighbourhood, see \cite[Proposition 5.26]{lee2018}. In these coordinates, the point $p$ corresponds to $(0,\theta_0)$ and we have $\vp_0(0,\theta)=0$ so that $\dr_{\theta}\vp_0(0,\theta)=0$. We obtain
		\begin{eqnarray*}
			\lap_{g_{\Psi}} \log|\vp_0|^2(r,\theta_0)& = & \dr^2_{rr} \log|\vp_0|^2(r,\theta_0) + \lap_{\Cr}\log|\vp_0|^2(r,\theta_0) \\[2mm]
			& \ust{r\to 0}{=} & \Big( \dr^2_{rr} \log\left(r^2|\dr_r\vp_0(0,\theta_0)|^2 + O(r^3)\right) \Big) \big(1+o(1)\big)\\[2mm]
			& \ust{r\to 0}{=} & \frac{-2}{r^2}+ o\left(\frac{1}{r^2}\right).
		\end{eqnarray*}
		Using Liouville equation, we obtain
		\begin{align*}
			\frac{|\Arond|^2_{g_{\Psi}} }{2}K_Y(r,\theta_0) - K_{g_{\Psi}}(r,\theta_0) = -\frac{1}{2}\lap_{g_{\Psi}} \log |\Arond|^2_{g_{\Psi}} \ust{r\to 0}{=} \frac{1}{r^2}+ o\left(\frac{1}{r^2}\right).
		\end{align*}
	\end{proof}

	\begin{remark}\label{rk:KY_isolated}
		The proof also shows that if $x_0$ is an isolated umbilic point of type \ref{typeI} or \ref{typeII}, then $\frac{|\Arond|^2_{g_{\Psi}} }{2}K_Y$ is bounded on $B_{g_{\Psi}}(x_0,\ve)\setminus \{x_0\}$ for $\ve>0$ small enough. Indeed, in this case, we have that $|\vp_0|$ is bounded from below by \eqref{typeI-IIest} and thus, the last term of \eqref{eq:laplacian_conf_fact} is bounded on $\Cr$. 
	\end{remark}

	\subsubsection{Description of umbilic curves with the conformal Gauss map.}\label{sec:Umbilic_curves} Let $\Cr\subset \Sigma$ be a closed curve of umbilic points of a non-totally umbilic Willmore immersion $\Psi\colon \Sigma\to \s^3$. In complex coordinates, we have $Y_z= H_z \nu$ on $\Cr$. Depending on the size of the critical points of $Y_{|\Cr}$, one can distinguish two cases.\\
	
	\textit{Case (1)}\\
	There exists an open set of $I\subset \Cr$ such that $H$ is constant on $I$. Since $H$ is a real-analytic map and $\Cr$ is a real-analytic curve, we deduce that $H$ must be constant on $\Cr$. Hence the set $Y(\Cr)\subset \s^{3,1}$ is a point, meaning that the curve $\Psi(\Cr)$ lies in a fixed hypersphere of $\s^3$. After stereographic projection from a point of this sphere not lying on $\Psi(\Sigma)$, we obtain a Willmore immersion $\Phi\colon \Sigma\to \R^3$ having the same curve of umbilic points $\Cr$ and now satisfying $\Phi(\Cr)\subset \R^2\times \{0\}$ with $Y_{\Phi}$ constant on $\Cr$. The third component of $Y_{\Phi}$ is given by
		\begin{align*}
			Y^3_{\Phi} = \Phi^3\, H_{\Phi} + \vn_{\Phi}^3.
		\end{align*}
		Since $\Phi^3=0$ on $\Cr$, we deduce that $\vn_{\Phi}^3$ is constant on $\Cr$. In other words, the intersection of $\Phi(\Sigma)$ with the plane $\R^2\times \{0\}$ has constant angle along $\Phi(\Cr)$. Such curves have been studied in \cite{palmer2000} when the image $\Phi(\Cr)$ is a round circle. In the next section, we will study the case where the intersection is orthogonal (but $H$ is not necessarily constant).\\
		
	\textit{Case (2)}\\
	The critical points of $H$ on $\Cr$ are isolated. Let $\tau$ be the tangent vector to $\Cr$ and $I\subset \Cr$ be an interval with $\g_{\tau} H\neq 0$ on $I$. We obtain $\g_{\tau} Y\neq 0$ on $I$. Hence, the set $Y(I)\subset \s^{3,1}$ defines a curve without critical point and can be parametrized by (Euclidean) arc-length. We consider a local parametrization with variable $\theta\in[0,L]$ such that $\g_{\theta} H=1$. Since $\g_{\theta} Y_5 = 1$, the function $Y_5$ is strictly increasing on $[0,L]$. If we could choose $I=\Cr$, then the curve $Y\colon \Cr\to \s^{3,1}$ would not be closed, which is impossible. Another way to state this phenomenon is to say that the map $Y_5\colon \Cr\to \R$ is continuous from a closed domain and thus, has at least one maximum and one minimum, in particular, at least two critical points. We have $\g_{\tau} Y = (\g_{\tau} H)\, \nu$. Hence, each critical point of $Y^5_{|\Cr}$ on $\Cr$ is actually a critical point of $Y_{|\Cr}$. Hence, the parametrized curve $Y\colon \Cr\to Y(\Cr)$ is not a regular parametrization.

	\subsection{Umbilic curves of Babich--Bobenko type}
	
	We consider $\Phi\colon \Sigma\to \R^3$ a Willmore immersion and $\Ur$ its umbilic set. One can wonder whether its umbilic curves satisfy a particular equation or not. In this section, we focus on the case of umbilic curves which are geodesics. We will consider the following versions of the 3-dimensional hyperbolic space endowed with the hyperbolic metric $\xi$:
	\begin{equation} \label{eq:Hpm}
	\begin{aligned}
		& \Hd^3_+\coloneqq \{x\in \R^3 : x^3>0\},\\
		& \Hd^3_-\coloneqq \{x\in \R^3 : x^3<0\},\\
		& \xi_{\alpha\beta} \coloneqq (x^3)^{-2} \delta_{\alpha\beta}.
	\end{aligned}
	\end{equation}
	\begin{defi}
		\label{defBBtype}
		We say that a Willmore immersion $\Phi\colon \Sigma\to \R^3$ is of Babich--Bobenko type if there exists $f\in \mathrm{Isom}(\R^3)$ such that the two immersions $(f\circ \Phi)_{|\{(f\circ \Phi)^3>0\}}$ and $(f\circ \Phi)_{|\{(f\circ \Phi)^3<0\}}$ are minimal in $(\Hd^3_+,\xi)$ and $(\Hd^3_-,\xi)$ respectively.
	\end{defi}
	The goal of this section is to prove \Cref{theogeodbb} below, stating that if $\Phi$ carries a umbilic curve which is also a geodesic in $(\Sigma,g_{\Phi})$, then $\Phi$ is of Babich--Bobenko type. To do so, we first prove in \Cref{geodmeansflat} that if an umbilic curve $\Cr$ is also geodesic, then $\Phi(\Cr)$ is contained in a plane $\pr$ orthogonal to $\Phi(\Sigma)$. Let $\pr_{\pm}$ be the two connected components of $\R^3\setminus \pr$. In \Cref{AcquaSchatz}, we prove that each component $\Phi(\Sigma)\cap \pr_{\pm}$ is minimal if $\Hd^3$ (up to a rigid motion). After proving \Cref{theogeodbb}, we state a conformally invariant version in \Cref{th:Conf_BB}.
	
	\begin{lemma}
		\label{geodmeansflat}
		Let $\Phi\colon \Sigma \to \R^3$ be a smooth immersion of a Riemann surface $\Sigma$. Assume that $\Cr$ is a smooth curve on $(\Sigma,g_{\Phi})$ which is both umbilic and geodesic. Then the curve $\Phi( \Cr)$ is contained in a plane orthogonal to the surface.
	\end{lemma}
	\begin{proof}
		Let  $p_0\in \Cr$ and $p \colon I \subset \R \rightarrow \Cr$ be a $g_{\Phi}$-arc length parametrization such that $p(0)=p_0$ and such that $\Phi$ defines an orientation for $\Sigma$ on $p(I)$. Let $\gamma = \Phi \circ p $ be the corresponding arc-length parametrization of $\Phi(\Cr)$. We denote $\tau = p'$ the tangent vector along $\Cr$ and $\vec{\tau} = \gamma' = d\Phi(\tau)$  the tangent vector along $\Phi (\Cr)$. 
		
		Let $\nabla$ be the Levi-Civita connection associated to $g_{\Phi}$, $\Gamma^k_{ab}$ be the Christoffel symbols, $k_g$ be the geodesic curvature of $\Cr$, and $\nu$  the unit direct normal vector to $\tau$. One has by definition:
		$$
		\forall s\in I,\qquad \nabla_{\frac{d}{ds}} \tau(s) = \tau'(s) + \Gamma^{\cdot}_{ab}\, \tau^a\, \tau^b = k_g(s)\, \nu(s).
		$$
		Similarly, we denote $\vec{\nu} = d\Phi(\nu)$, so that $(\tau, \nu)$ (respectively $(\vec{\tau}, \vec{\nu})$) defines a smooth orthonormal frame for $T_{p(s)}\Sigma$ (respectively $T_{\gamma(s)} \Phi(\Sigma)$).\\

		The normal to $\Phi(\Sigma)$ at $\gamma(s)$ is given by $\vec{n}(p(s)) = \vec{\tau}(s)\times \vec{\nu}(s)$, which can then be extended to a neighborhood of $\phi(\Cr)$ in $\Phi(\Sigma)$. Denoting $A$ the second fundamental form of $\Phi$ associated to this normal, one then has
		$$
		\begin{aligned}
			&\nabla_{ij} \Phi  = A_{ij}\, \vec{n}, \\[2mm]
			&\nabla_i \vec{n} = -A_i^{\ p}\, \nabla_p \Phi.
		\end{aligned}
		$$
		We first differentiate the vector field $\nu$:
		$$
		\begin{aligned}
			\nabla_{\frac{d}{ds}} \nu(s) &= \left\langle \nabla_{\frac{d}{ds}}\nu(s), \nu(s) \right\rangle \nu(s) + \left\langle \nabla_{\frac{d}{ds}} \nu(s) , \tau(s) \right\rangle \tau(s) \\[2mm]
			&= \frac{1}{2} \frac{d}{ds} \Big( \langle \nu(s) , \nu(s)\rangle \Big) \nu(s) + \left[ \frac{d}{ds} \big( \langle \nu(s), \tau(s) \rangle  \big) - \left\langle \nu(s), \nabla_{\frac{d}{ds}}\tau(s) \right\rangle \right] \tau(s) \\[2mm]
			&= - k_g(s)\, \tau(s).
		\end{aligned}
		$$ 
		We deduce the derivative of $\vec{\nu}$:
		\begin{equation} \label{derivvecnu}
			\begin{aligned}
				\vec{\nu}' &= d^2\Phi_{p} (p',\nu) + d\Phi(\nu') \\[1mm]
				&=\partial_{ij} \Phi(p)\, \tau^i\, \nu^j + \partial_k \Phi (p)\, (\nu^k)' \\[1mm]
				&= \left( \nabla_{ij} \Phi - \Gamma^k_{ij}\partial_k \Phi(s) \right) (p)\, \tau^i\, \nu^j + \partial_k \Phi(p)\, \left( \nabla_{\frac{d}{ds}} \nu^k + \Gamma^k_{ij}(p)\, \tau^i\, \nu^j \right) \\[1mm]
				&= \nabla^2 \Phi_{p} (\tau, \nu) + d\Phi\left(\nabla_{\frac{d}{ds}} \nu \right) \\[1mm]
				& = A(\tau, \nu)\, \vec{n}(p) - k_g\, \vec{\tau}. 
			\end{aligned}
		\end{equation}
		Decomposing the second fundamental form into its mean curvature and tracefree part yields 
		$A = \Arond + Hg_{\Phi}$. Since $g(\tau, \nu)= 0$, we obtain $A(\tau, \nu) = \Arond(\tau, \nu)$. Then, \eqref{derivvecnu} becomes 
		\begin{equation}\label{deriveenu}
			\begin{aligned}
				\vec{\nu}' &=  \Arond(\tau, \nu)\, \vec{n}(p) - k_g\, \vec{\tau}. 
			\end{aligned}
		\end{equation}
		Since $\Cr$ is umbilic and geodesic, we have $\Arond= 0$ and $k_g=0$. Therefore, \eqref{deriveenu} becomes $\vec{\nu}'=0$, meaning $\vec{\nu}$ is a constant vector. The curve is then contained in the plane $\mathcal{P}$ going through $p_0$ and normal to $\vec{\nu}$. Since $\vec{\nu}$ is tangent to $\Phi(\Sigma)$, we obtain that $\mathcal{P}$ is met orthogonally by the surface, which proves the result.
	\end{proof}

	\Cref{geodmeansflat} becomes relevant once juxtaposed with a result of Dall'Acqua--Schatzle (see \cite{dallacquaschatzle}, between Equation (2.2) and Equation (2.7)), where the following result was obtained for rotationally symmetric Willmore surfaces. We check that their ideas also work outside this context.
	\begin{proposition}
		\label{AcquaSchatz}
		Let $\Phi\colon \Sigma \to \R^3$ be a Willmore immersion. Consider a curve $\Cr\subset \Sigma$ such that 
		\begin{enumerate}
			\item \label{it1} $\Cr$ is umbilic,
			\item \label{it2} $\Phi(\Cr)$ is contained in a plane $\mathcal{P}$
			\item \label{it3} $\Phi(\Cr)$ meets $\mathcal{P}$ orthogonally.
		\end{enumerate}
		Denoting $\mathcal{P}_\pm$ the two half-spaces delimited by $\mathcal{P}$, we obtain that $\Phi(\Sigma)_\pm \coloneqq \Phi(\Sigma) \cap \mathcal{P}_\pm$ is a minimal surface in the hyperbolic space obtained by endowing $\mathcal{P}_\pm$ with the metric of the Poincaré half-space model. 
	\end{proposition}
	\begin{proof}
		Let us thus consider a Willmore immersion $\Phi\colon \Sigma \to \R^3$ satisfying \eqref{it1}-\eqref{it3}. Up to an isometry, one can assume that $\mathcal{P}$ is  the horizontal $\{x^3 = 0 \}$ plane in $\R^3$, which allows one to translate \eqref{it2}-\eqref{it3} into 
		\begin{equation}
			\Phi^3 = \vec{n}^3_{\Phi} = 0 \qquad \text{ on } \Cr,
		\end{equation}
		where $\vec{n}_{\Phi}$ denotes the Gauss map of the immersion  and $\Phi^3$, $\vec{n}^3_{\Phi}$ the third components of respectively $\Phi$ and $\vec{n}_{\Phi}$.
		
		We denote $\Sigma_\pm = \Phi^{-1}(\Hd^3_\pm)$ with the notations of \eqref{eq:Hpm}. This yields two immersions 
		$$
		\Phi_{\pm}\colon \left\{
		\begin{aligned}
			\Sigma_\pm &\rightarrow (\Hd^3_\pm, \xi) \\
			p &\mapsto \Phi(p).
		\end{aligned}
		\right.
		$$
		Since being Willmore is invariant under conformal changes of the ambiant metric, $\Phi_{\pm}$ are two Willmore immersions in the Poincaré half-space model of the hyperbolic space. One can compute the relevant quantities (marked with a subscript $\pm$) for $\Phi_{\pm}$ as functions of their avatars in euclidean spaces, namely:
		\begin{equation}\label{conversions}\begin{aligned}
				g_{\pm} = \frac{g_{\Phi}}{(\Phi^3)^2}, \qquad \vec{n}_{\pm} = \Phi^3\, \vec{n}_{\Phi}, \qquad  A_{\pm} = \frac{A_{\Phi}}{\Phi^3}  + \frac{\vec{n}^3_{\Phi}}{(\Phi^3)^2}\, g_{\Phi}.
		\end{aligned}\end{equation}
		Taking half of the trace with respect to $g_{\pm}$ of $A_{\pm}$ yields
		\begin{equation}
			\label{hconversions}
			H_{\pm} = \Phi^3\, H_{\Phi} + \vec{n}^3_{\Phi}.
		\end{equation}
		Hence the tracefree fundamental forms are linked by
		\begin{equation} \label{arconversions}
			\Arond_{\pm} = \frac{\Arond_{\Phi}}{\Phi^3}.
		\end{equation}
		Since $\Phi$ is Willmore, it is analytic and thus $\Phi^3 H_{\Phi} + \vec{n}^3_{\Phi}$ is defined on the whole surface $\Sigma$. The mean curvature of $\Phi_{\pm}$ can then be extended into a function $\mathcal{H} = \Phi^3 H_{\Phi} + \vec{n}^3_{\Phi}$ defined on $\Sigma$, and a fortiori on $\Cr$. Furthermore, denoting $\nu$ a normal to $\Cr$ in $\Sigma$, the hypotheses \eqref{it1}-\eqref{it3} provide a formula for $\mathcal{H}$ and $\partial_\nu \mathcal{H}$ on $\Cr$. Indeed, given $p\in\Cr$, it holds
		\begin{align*}
			\mathcal{H} (p) = \Phi^3 (p)\, H_{\Phi}(p) + \vec{n}^3_{\Phi}(p) = 0.
		\end{align*}
		Its normal derivative is given by
		\begin{align*}
			\partial_\nu \mathcal{H} (p) & = \partial_\nu \Phi^3 (p)\, H_{\Phi} (p) + \Phi^3(p)\, \partial_\nu H_{\Phi}(p) + \partial_\nu \vec{n}^3_{\Phi}(p) \\[1mm]
			&  = \partial_\nu \Phi^3 (p)\, H_{\Phi}(p) + \Phi^3(p)\, \partial_\nu H_{\Phi}(p) - H_{\Phi}(p)\, \partial_\nu \Phi^3(p)  - \Arond_{\Phi}(p) (\nu , \nabla \Phi^3) \\[1mm]
			&= \Phi^3(p)\, \partial_\nu H_{\Phi}(p)  - \Arond_{\Phi}(p) (\nu , \nabla \Phi^3) = 0.
		\end{align*}
		Since the immersions $\Phi_{\pm}$ are Willmore, they satisfy the Willmore equation on $\Sigma_{\pm}$:
		$$
		\Delta_{g_{\pm}} H_{\pm} + \big|\Arond_{\pm}\big|^2_{g_{\pm}} H_{\pm} =0.
		$$
		From \eqref{conversions} and \eqref{arconversions} one deduces that $\Delta_{g_{\pm}} = (\Phi^3)^2 \Delta_{g_{\Phi}}$  and $\big|\Arond_{\pm}\big|^2_{g_{\pm}} = (\Phi^3)^2 \big|\Arond_{\Phi}\big|^2_{g_{\Phi}}$, meaning that $\mathcal{H}$ satisfies 
		$$
		\Delta_{g_{\Phi}}\mathcal{H} + \big|\Arond_{\Phi}\big|_{g_{\Phi}}^2 \mathcal{H} = 0 \qquad \text{ on } \Sigma_+\cup \Sigma_-.
		$$ 
		This equation is well defined on the whole manifold, and thus $\mathcal{H}$ is a solution of the evolution problem starting from $\Cr$:
		$$
		\begin{cases}
			\displaystyle \Delta_{g_{\Phi}}\mathcal{H} + \big|\Arond_{\Phi}\big|_{g_{\Phi}}^2 \mathcal{H}  = 0 & \text{ in } \Sigma,\\[1mm]
			\mathcal{H}= 0 & \text{ on } \Cr, \\[1mm]
			\partial_\nu \mathcal{H}  = 0 & \text{ on } \Cr.
		\end{cases}
		$$
		The uniqueness in Cauchy--Kovalevski Theorem together with the analyticity of $\mathcal{H}$ then ensure that $\mathcal{H} =0$ in a neighbourhood of a point of $\Cr$, and thus on the whole surface. The hyperbolic curvatures then satisfy $H_{\pm} = \mathcal{H}_{| \Sigma_{\pm}}= 0$. Therefore, the immersions $\Psi_{\pm}$ are minimal surfaces in the Poincaré half-space models of the hyperbolic space.
	\end{proof}

	One can then deduce the following result for geodesic umbilic curves of Willmore surfaces:
	
	\begin{theorem}
		\label{theogeodbb}
		If $\Phi\colon \Sigma \to \R^3$ is a Willmore immersion with an umbilic and geodesic curve $\Cr$, then $\Phi$ is of Babych--Bobenko type. 
	\end{theorem}
	\begin{proof}
		By \Cref{geodmeansflat}, all the hypothesis of \Cref{AcquaSchatz} are fulfilled. 
	\end{proof}
	One of the consequences of this theorem is to help highlight global phenomena in the behavior of these Willmore umbilic curves
	\begin{corollary}\label{cor:geodesic}
		If $\Phi\colon \Sigma \rightarrow \R^3$ is a Willmore immersion with an umbilic geodesic curve, then all umbilic curves are geodesic.
	\end{corollary}
	\begin{proof}
		Theorem \ref{theogeodbb} ensures that such a surface is of Babych--Bobenko type, and thus that all umbilic curves are exactly the orthogonal intersections of the surface and the plane $\{x^3 =0\}$ (see \cite{babisch1993}).
		
		Using the notations in the proof of \Cref{geodmeansflat}, this ensures that the normal vector to the plane is the normal vector to the umbilic curve in $\Sigma$, which we denoted $\vec{\nu}$. Hence, this vector field is constant. Using \eqref{deriveenu}, we obtain $k_g=0$.
	\end{proof}

	We now state a conformally invariant version of \Cref{theogeodbb}.
	
	\begin{theorem}\label{th:Conf_BB}
		Let $\Phi\colon \Sigma\to \R^3$ be a Willmore immersion. We denote 
		\begin{align*}
			\nu \coloneqq \begin{pmatrix}
				\Phi \\[2mm] \displaystyle \frac{|\Phi|^2-1}{2} \\[3mm] \displaystyle\frac{|\Phi|^2 +1}{2}
			\end{pmatrix} \in \R^{4,1}.
		\end{align*}
		Let $\Cr\subset \Sigma$ be a closed umbilic curve. Assume that there exists a point $e\in \s^{3,1}$ such that 
		\begin{align*}
			\forall x\in \Cr,\qquad 
			\scal{\nu(x)}{e}_{\eta} = \scal{Y(x)}{e}_{\eta} =0.
		\end{align*}
		Then up to a conformal transformation in $\R^3$, the immersion $\Phi$ is of Babich--Bobenko type.
	\end{theorem}
	
	\begin{proof}
		Let $S(e)\subset \R^3$ be the sphere associated to $e$. Since the action of $SO(4,1)$ on $\s^{3,1}$ is transitive, we can assume that $S(e) = \R^2\times \{0\}$, that is to say $e=(0,0,1,0,0)$. Up to this conformal transformation, we obtain $\Phi^3=0$ and $Y^3=0$. By definition of the conformal Gauss map, it holds 
		\begin{align*}
			\Phi^3 = \vn_{\Phi}^3 =0 \qquad \text{on }\Cr.
		\end{align*}
		Thus, all the assumptions of \Cref{AcquaSchatz} are verified. 
	\end{proof}

	\section{Gauss--Bonnet formula}\label{sec:GaussBonnet}
	
	The goal of this section is to prove \Cref{th:main}.
	
	\subsection{Setting} In this section, we consider a smooth Willmore immersion $\Psi\colon \Sigma\to \s^3$, denote $g$ its induced metric, $Y$ its conformal Gauss map and $g_Y$ the metric it induces. Let $h$ be a smooth metric on $\Sigma$ in the conformal class of $g_\Psi$ (and necessarily of $g_Y$), and  let $\lambda$, $\rho$  be the respective conformal factors of $g$ and $g_Y$: 
	\begin{equation}\label{eq:conformal_relations}
		\begin{cases}
			\displaystyle g_{\Psi}  \coloneqq \Psi^*\xi = e^{2\lambda} h,\\[2mm]
			\displaystyle g_Y \coloneqq Y^*\eta = e^{2\rho} h = \frac{|\Arond|^2_{g_\Psi}}{2} e^{2\lambda} h.
		\end{cases}
	\end{equation}
	We denote the set of umbilic points by $\Ur\coloneqq \{ x\in \Sigma : \Arond(x) = 0\}$.  From Theorem \ref{th:structure_umbilic_set} 
 it is known that $\mathcal{U}$ is the union of a finite number of isolated umbilic points $\left(p_i\right)_{i=1 \dots q}$ of respective multiplicities $n_ji$ and a finite number of isolated smooth closed curves $\left( \Gamma_i \right)_{i=1 \dots n}$ of finite  lengths $L_i$ computed with the background metric $h$, each containing $(\tilde p_{ij})_{j=1 \dots q_i}$ critical point of $Y$ having respective multiplicities $\tilde n_{ij}$. We will use the following formula for Gauss--Bonnet on a  surface with boundary:
	\begin{proposition} {\cite[Section 4-5]{docarmo2016}}
		\label{formuleGaussbonnet}
		Let $(M,g)$ be a Riemannian manifold of dimension $2$ with boundary. Then it holds:
		$$
			\int_M K_g\, d\mathrm{vol}_g = 2\pi \chi(M) - \int_{\partial M} k_g\, d\mathrm{vol}_g,
		$$
		where $K_g$ denotes the Gauss curvature of $M$, $\chi(M)$ its Euler characteristic, and $k_g$ the geodesic curvature of $\partial M \subset (M,g)$.
	\end{proposition}
	
	Let $(\Ur_{\ve})_{\ve>0}$ be a decreasing sequence of sets to be chosen later such that 
	$$
	\displaystyle{\bigcap_{\varepsilon>0} \Ur_\varepsilon} = \mathcal{U}.
	$$
	The conformal change of geodesic curvature on the boundary is a well-known formula, that we prove in \Cref{sec:GB_appendix} for completeness' sake. For any $\ve>0$, using that $h$ is a smooth metric on $\Sigma$, it holds
	\begin{equation}
		\label{formulegaussbonnetintegraleaubord}
		\int_{ \Sigma\backslash \Ur_\varepsilon } K_Y\, d\mathrm{vol}_{g_Y}  \ust{\ve\to 0}{=} 2\pi\, \chi(\Sigma) + \int_{\partial \mathcal{U}_\varepsilon} \partial_{\nu} \rho\, d\mathrm{vol}_h + O(\varepsilon).
	\end{equation}
	In the above formula, the symbol $\nu$ denotes the normal to $\partial \mathcal{U}_\varepsilon$ pointing inside $\mathcal{U}_\varepsilon.$\\
	
	In \Cref{contributionofisolatedbranchedumbilicpoints} we will build $\mathcal{U}_\varepsilon$ around isolated singular umbilic points, and compute their boundary contribution by taking away a small coordinate disk around the point, while in \Cref{contributionofumbiliccurves} we will use tubular neighbourhoods to handle  umbilic curves,  with the added difficulty that they may contain singular umbilic points.
	
	\subsection{Contribution of isolated umbilic points}
		\label{contributionofisolatedbranchedumbilicpoints}
		We will work out the contribution of isolated umbilic points according to their type. 
		\subsubsection*{Isolated umbilic points of types \ref{typeI} and \ref{typeII}}
		The contribution of those isolated umbilic points is akin to that of a branch point on a classical immersion in $\R^n$ (compare for instance \cite{lamm2013},\cite[Theorem 1.2.5]{nmthese} to \eqref{contribbranchedumbilic} below). The proof is similar, and yields a contribution proportional to their multiplicity.
	
	Let $p \in \Sigma$ be an isolated umbilic point of type \ref{typeI} or \ref{typeII} of multiplicity $n(p)$, let $D$ be a small geodesic disk centered on $p$ such that the local conformal coordinates in which \eqref{typeI-IIest} stands are defined, meaning that there exists complex coordinates defined on $\D_{r_0}$, $D \subset \D_{r_0}$, and constants $c_p,C_p>0$ such that 
$$ 
\left\{
\begin{aligned}
&|\varphi|(z) = Cr^{n(p)} | \mathcal{F}(z)|, \\[2mm]
& |\mathcal{F}(z)| \ge c_p, \\[2mm]
& |\partial_r \mathcal{F}(z)| \le C_p.
 \end{aligned}
\right.
$$
	In radial coordinates, we have the equality
	\begin{align}\label{eq:rho_isolated} 
		\rho = \log \left( |\vp e^{-\lambda} |\right) = n(p) \log r + \frac{1}{2} \log \left( \left|\mathcal{F}(z)\right|^2 \right)  - \lambda(z) + \log C.
	\end{align} 
	In this case the $\varepsilon$-neighborhood will be $\mathcal{U}_\varepsilon(p)= \{ |z|\leq \varepsilon\}$, the coordinate ball of radius $\ve$ in $D$. Since $g$ and $h$ are conformal, there exists $ \tau\in C^{\infty}(D;\R)$ such that $h= e^{2\tau} \delta$. Hence the exterior normal to $ \partial \mathcal{U}_\varepsilon(p)$ is $\nu = e^{-\tau} \partial_r$ while the volume element is $e^{\tau} \varepsilon\, d\mathrm{vol}_{\s^1}$. We have
\begin{align} 
\int_{\partial \mathcal{U}_\varepsilon} \partial_\nu \rho\ d\mathrm{vol}_h&=\int_0^{2\pi} e^{-\tau(\varepsilon e^{i\theta} )} \left[ \frac{n(p)}{\varepsilon} + \frac{1}{2} \left( \frac{ \partial_r \mathcal{F}}{\mathcal{F}}  + \frac{ \partial_r \overline{\mathcal{F}}}{ \overline{\mathcal{F} }} \right)(\varepsilon e^{i\theta}) - \partial_r \lambda (\varepsilon e^{i\theta}))  \right] \varepsilon\, e^{\tau(\varepsilon e^{i\theta})}\, d\theta  \nonumber \\[3mm]
&= 2\pi n(p) + \varepsilon \left[ \int_0^{2\pi} \Re\left(  \frac{ \partial_r \mathcal{F}}{\mathcal{F}}   \right)(\varepsilon e^{i\theta})\, d \theta -  \int_0^{2\pi} \partial_r \lambda (\varepsilon e^{i\theta})\, d\theta \right]. \label{eq:rho_TypeI_II}
\end{align}
Since $\Psi$ is a smooth Willmore surface, there exists a constant $C_{r_0}$ such that $\|\nabla \lambda \|_{L ^\infty (\D_{r_0})} \le C_{r_0}$, thus 
\begin{equation}\label{facteurconformenopb} 
	\left|  \int_0^{2\pi} \partial_r \lambda (\varepsilon e^{i\theta})\, d\theta \right| \le 2\pi C_{r_0}.
\end{equation}
Moreover, we deduce from \eqref{typeI-IIest} that
$$
\left| \int_0^{2\pi} \Re \left(  \frac{ \partial_r \mathcal{F}}{\mathcal{F}}  \right)(\varepsilon e^{i\theta})\, d \theta \right| \le \int_0^{2\pi} \frac{|\partial_r \mathcal{F}|}{|\mathcal{F}|}\, d \theta \le 4 \pi \frac{C_p}{c_p}.
$$
Combing back to \eqref{eq:rho_TypeI_II}, we obtain
\begin{equation} \label{contribbranchedumbilic}  
	\begin{aligned}
		\int_{\partial \mathcal{U}_\varepsilon} \partial_\nu \rho\, d\mathrm{vol}_h&= 2\pi\, n(p) + O(\varepsilon),
	\end{aligned}
\end{equation}
	where the constant hidden in the $O$ depends on the Willmore immersion $\Psi$, but is uniform in $\varepsilon$.

\subsubsection*{Isolated umbilic points of types \ref{typeIII}}

Let now $p$ be an isolated umbilic point of type \ref{typeIII}. Then there exists a small geodesic disk $D$ centered on $p$ on which there exists (see \eqref{formenormalisedevarphi1}) local conformal coordinates defined on $\D_{r_0}$, $D\subset \D_{r_0}$, an integer $m$, a complex constant $C$ and two real-valued analytic functions  $A$ and $B$ defined on $\D_{r_0}$ such that
$$\begin{aligned}&\varphi(z) =C z^m ( y + A(z) + i B(z) ).
 \end{aligned}$$
In addition, $A$ and $B$ satisfy $A(0) =B(0)=0$ and $ \nabla A(0) = \nabla B(0)=0$. The map $\big( z \mapsto y +A(z) \big)$ is an analytic submersion on $\D_{r_0}$, which can be described by an analytic graph $(t, U(t))$, where $U$ is an analytic real function such that $U(0)= U'(0)= 0$ (this follows from \eqref{eq:cond_curve}) and having a convergence radius $t_0>r_0$ (which can be obtained by limiting the domain of study). Moreover $p$ is the only umbilic point in $\D_{r_0}$, meaning that 
$$
	\{y+A(z)=0 \} \cap \{B(z)=0\} = \{ B(t+ i U(t))= 0 \} = \{z=0\}= \{p\}. 
$$
As before, in radial coordinates, we have the equality
	\begin{align}\label{eq:rho_isolatedbis} 
		\rho = \log \left( |\vp e^{-\lambda} |\right) = m \log r + \frac{1}{2} \log \left( \left|y +A(z) +iB(z)\right|^2 \right)  - \lambda(z) + \log C.
	\end{align} 
	The $\varepsilon$-neighborhood will still be $\mathcal{U}_\varepsilon(p)= \{ |z|\leq \varepsilon\}$, the coordinate ball of radius $\ve$ in $D$, and since $g$ and $h$ are conformal, there exists $ \tau\in C^{\infty}(D;\R)$ such that $h= e^{2\tau} \delta$. Hence the exterior normal to $ \partial \mathcal{U}_\varepsilon(p)$ is $\nu = e^{-\tau} \partial_r$ while the volume element is $e^{\tau} \varepsilon\, d\mathrm{vol}_{\s^1}$. \\
We obtain:
$$
	\begin{aligned}
 \int_{\partial \mathcal{U}_\varepsilon} \partial_\nu \rho\ d\mathrm{vol}_h&=\int_{-\frac{\pi}{2}}^{\frac{3\pi}{2}}  \left[ \frac{m}{\varepsilon} +\Re \left( \frac{ \sin \theta + \partial_r A + i \partial_r B}{ r \sin\theta + A + i B}   \right)(\varepsilon e^{i\theta}) - \partial_r \lambda (\varepsilon e^{i\theta}))  \right] \varepsilon\,  d\theta  \\[3mm]
&= 2\pi m -  \varepsilon  \int_{-\frac{\pi}{2}}^{\frac{3\pi}{2}} \partial_r \lambda (\varepsilon e^{i\theta})\, d\theta     \\[3mm]
&\qquad  +\int_{-\frac{\pi}{2}}^{\frac{3\pi}{2}} \Re\left( \frac{ \sin \theta + \frac{A}{r} + i \frac{B}{r}  + \partial_r A - \frac{A}{r} + i  \left[ \partial_r B- \frac{B}{r} \right]}{  \sin\theta + \frac{A}{r} + i \frac{B}{r}} \right)(\varepsilon e^{i\theta})d \theta \\[3mm] 
&=  2\pi m -  \varepsilon  \int_{-\frac{\pi}{2}}^{\frac{3\pi}{2}} \partial_r \lambda (\varepsilon e^{i\theta})\, d\theta  \\[3mm]
 & \qquad   +   \int_{-\frac{\pi}{2}}^{\frac{3\pi}{2}} \Re\left( 1+\frac{  \partial_r A - \frac{A}{r} + i  \left[ \partial_r B- \frac{B}{r} \right]}{  \sin\theta + \frac{A}{r} + i \frac{B}{r}} \right)(\varepsilon e^{i\theta})d \theta.
 \end{aligned} 
$$
Since for type \ref{typeIII} umbilic points, it holds $n(p)=m+1$, we obtain
\begin{align}\label{eq:decompo_typeIII}
	\int_{\partial \mathcal{U}_\varepsilon} \partial_\nu \rho\ d\mathrm{vol}_h= 2\pi n(p) -  \varepsilon  \int_{-\frac{\pi}{2}}^{\frac{3\pi}{2}} \partial_r \lambda (\varepsilon e^{i\theta}) d\theta + \int_{-\frac{\pi}{2}}^{\frac{3\pi}{2}}\Re\left( \frac{  \partial_r A - \frac{A}{r} + i  \left[ \partial_r B- \frac{B}{r} \right]}{  \sin\theta + \frac{A}{r} + i \frac{B}{r}} \right)(\varepsilon e^{i\theta})d \theta,
\end{align}
As in \eqref{facteurconformenopb}, the second term will be a $O(\varepsilon)$. There remains only to estimate the last one, with the core difficulty stemming from  its denominator  no longer being uniformly bounded away from $0$. Controlling it will require classic estimates for real-analytic functions on their domain of convergence. We will, for completeness' sake, detail how to obtain them starting with the following lemma.
\begin{lemma}[Lemma 2.1 in \cite{blatt2024analyticitysolutionsnonlinearelliptic}]
	A function $u :\, \Omega \rightarrow \R^m$, $\Omega \subset \R^n$ is analytic if and only if for every compact set $K \subset \Omega$ there exists constants $C= C_K$, $A= A_K< \infty$ such that for every multi-index $\alpha \in \R^n$ we have
	$$\|\partial^\alpha u \|_{L^\infty (K)} \le C A^{|\alpha|} |\alpha|!.$$
\end{lemma}

By direct application of Lemma 4.2, we obtain the following estimates.

\begin{claim}
\label{lesinegalitessurlesfonctionsanalytiques}
There exists $r_1 \le r_0$ and a constant $C_1>0$ such that for all $z\in \D_{r_1}$: 
$$ 
\frac{1}{r}| A|(z)+ \frac{1}{r} |B|(z)+  | \partial_r A|(z)+ | \partial_r B|(z)+ \frac{1}{r} | \partial_\theta A |(z) + \frac{1}{r} | \partial_\theta B |(z) \le C_1\, r.
$$
\end{claim}
\begin{proof}
Since $A$ and $B$ are analytic on $\D_{r_0}$, one has by Lemma 4.2 with $K = \overline{\D_{\frac{1}{2}r_0}}$, the existence of constants $C_0$ and $M$ such that 
$$
	\| \partial^\alpha A \|_{L^\infty \left(\overline{\D_{\frac{1}{2}r_0}} \right)} + \| \partial^\alpha B \|_{L^\infty \left(\overline{\D_{\frac{1}{2}r_0}} \right)} \le C_0\, M^{|\alpha|}\, |\alpha|!.
$$
Let $r_1= \frac{1}{3} \min\left(\frac{1}{M},\frac{1}{2} r_0\right)$, so that one can transform the above estimate into 
$$
\| \partial^\alpha A \|_{L^\infty \left({\D_{r_1}} \right)} + \| \partial^\alpha B \|_{L^\infty \left({\D_{r_1}} \right)} \le  \frac{C_0 |\alpha|!}{(3r_1)^{|\alpha|}}.
$$
We start with the estimates on $A$. By definition, we have $A(0)= 0$ and $\nabla A (0)=0$. Thus, the first term of the Taylor expansion of $A$ is a $O(|z|^2)$. For all $|z|\le r_1$, it holds:
$$
\begin{aligned}
	\left| \sum_{k=2}^{\infty}  \sum_{l=0}^k\frac{\partial_z^{k-l} \partial_{\bz}^l A(0)}{(k-l)!\ l!}\,  z^{k-l}\, z^l \right|  &\le \sum_{k=2}^{\infty} \sum_{l= 0}^k \frac{C_0\ k!}{(3r_1)^k\ (k-l)!\ l!}\, r^k \\[3mm]
	& \le \sum_{k=2}^{\infty} \frac{C_0}{3^k} \sum_{l=0}^k \begin{pmatrix} k \\ l \end{pmatrix} \\[3mm]
	&\le  C_0 \sum_{k=2}^{\infty} \left(\frac{2}{3} \right)^k = \frac{4}{3}\, C_0.
\end{aligned}
$$
Thus the Taylor series for $A$ and all its derivatives converge uniformly on $\D_{r_1}$ and we obtain 
$$
	A(z) = \sum_{k=2}^\infty \sum_{l=0}^k\frac{\partial_z^{k-l} \partial_{\bz}^l A(0)}{(k-l)!\ l!}\  z^{k-l}\ z^l .
$$
In  particular for all $r=|z|\le r_1$, we have:

$$
\begin{aligned}
	\frac{1}{r^2}\left|  A(z)\right| &\le \sum_{k=2}^\infty  \sum_{l= 0}^k \frac{C_0\ k!}{(3r_1)^k\ (k-l)!\ l!}\ r^{k-2} \\[3mm]
	& \le \sum_{k=2}^\infty \frac{C_0\ r^{k-2}}{(3\, r_1)^k}  \sum_{l= 0}^k \begin{pmatrix}
		k\\ l
	\end{pmatrix} \\[3mm]
	&\le  \frac{C_0}{r_1^2} \sum_{k=2}^{\infty} \left(\frac{r}{r_1}\right)^{k-2} \left(\frac{2}{3} \right)^k \le \frac{4\ C_0 }{3\ r_1^2}.
\end{aligned}
$$
Concerning the angular derivative, we have
$$\begin{aligned}
	\frac{1}{r^2}\left| \partial_\theta  A(z)\right| &\le \sum_{k=2}^\infty  \sum_{l= 0}^k \frac{C_0\ k!\ |k-2l|}{(3\, r_1)^k\ (k-l)!\ l!} \ r^{k-2}\\[3mm] 
	& \le \sum_{k=2}^\infty \frac{C_0\ k\ r^{k-2}}{(3\, r_1)^k}  \sum_{l= 0}^k \begin{pmatrix}
		k\\ l
	\end{pmatrix} \\[3mm]
	&\le  \frac{C_0}{r_1^2} \sum_{k=2}^{\infty} k\left(\frac{r}{r_1}\right)^{k-2} \left(\frac{2}{3} \right)^k \le \frac{C_0\ C}{r_1^2},
\end{aligned}$$
For the radial derivative, it holds
$$\begin{aligned}
	\frac{1}{r}\left| \partial_r  A(z)\right| &\le \sum_{k=2}^\infty  \sum_{l= 0}^k \frac{C_0\ k!\ k}{(3\, r_1)^k\ (k-l)!\ l!}\ r^{k-2}\\[3mm]
	&  \le \sum_{k=2}^\infty \frac{C_0\ k\ r^{k-2}}{(3\, r_1)^k}  \sum_{l= 0}^k \begin{pmatrix} k\\ l
		\end{pmatrix}  \\[3mm]
	&\le  \frac{C_0}{r_1^2} \sum_{k=2}^{\infty} k\left(\frac{r}{r_1}\right)^{k-2} \left(\frac{2}{3} \right)^k \le \frac{C_0\ C}{r_1^2}.
\end{aligned}$$
Taking the biggest of the constants yields the result for $A$. Since $B$ satisfies the same estimates as $A$ this proves the claim.
\end{proof}

This claim allows one to estimate the numerator in the last term of \eqref{eq:decompo_typeIII}: for $\varepsilon \le r_1$, it holds
\begin{equation}
\label{lecontroledunum}
\left| \partial_r A - \frac{A}{r} + i \left[ \partial_r B - \frac{B}{r} \right] \right|(\varepsilon e^{i\theta}) \le 4C_1\varepsilon.
\end{equation}
Away from the graph $\{y+ A=0\}$, one can also control the denominator by straight-forward computations.
\begin{claim}
\label{lecontroledeyplusAoucava}
There exists $r_2 \le r_1$ such that for any $r\le r_2$ it holds,
\begin{equation} \label{lecontroledeyplusAoucava1} 
	\forall \theta \in \left[ -\frac{\pi}{2} , -\frac{\pi}{4} \right] \cup \left[ \frac{5\pi}{4}, \frac{3\pi}{2} \right] \cup \left[ \frac{\pi}{4}, \frac{3\pi}{4} \right], \qquad\left| \sin \theta + \frac{A}{r} \right| \ge \frac{\sqrt{2}}{4},
\end{equation}
\begin{equation} \label{lecontroledeyplusAoucava2} 
	\forall \theta \in \left[ -\frac{\pi}{4} , \frac{\pi}{4} \right], \qquad  \partial_\theta \left( \sin \theta + \frac{A}{r}\right)  \ge \frac{\sqrt{2}}{4},
\end{equation}
\begin{equation} \label{lecontroledeyplusAoucava2bis} 
	\forall \theta \in \left[ \frac{3\pi}{4}, \frac{5\pi}{4} \right],   \qquad \partial_\theta \left( \sin \theta + \frac{A}{r}\right) \le - \frac{\sqrt{2}}{4}.
\end{equation}
\end{claim}
\begin{proof}
Using claim \ref{lesinegalitessurlesfonctionsanalytiques} one has in the first studied angular intervals for $r\le r_2 \coloneqq \min\left( \frac{\sqrt{2}}{4C_1}, r_1\right)$ that
$$
\forall \theta \in \left[ -\frac{\pi}{2} , -\frac{\pi}{4} \right] \cup \left[ \frac{5\pi}{4}, \frac{3\pi}{2} \right] \cup \left[ \frac{\pi}{4}, \frac{3\pi}{4} \right], \qquad \left| \sin \theta + \frac{A}{r} \right| \ge  \left
|\sin \theta \right| - C_1r  \ge \frac{\sqrt{2}}{2} - C_1r \ge \frac{\sqrt{2}}{4}.
$$
While  in the second case $\partial_\theta \left( \sin \theta + \frac{A}{r}\right)  = \cos\theta + \frac{\partial_\theta A}{r}$ and thus, the following inequality holds for $r\leq r_2$, still by using claim \ref{lesinegalitessurlesfonctionsanalytiques}:
$$ 
	\forall \theta \in \left[ -\frac{\pi}{4} , \frac{\pi}{4} \right], \qquad \partial_\theta \left( \sin \theta + \frac{A}{r} \right) \ge  \cos \theta  - C_1r  \ge \frac{\sqrt{2}}{2} - C_1r \ge \frac{\sqrt{2}}{4}.
$$
The inequality \eqref{lecontroledeyplusAoucava2bis} follows in the same manner.
\end{proof}

Thus, combining claim \ref{lecontroledeyplusAoucava} and \eqref{lecontroledunum} yields 
\begin{equation}
\label{lintegralelaoucestok}
\begin{aligned}
\left| \int_{\left[ -\frac{\pi}{2} , -\frac{\pi}{4} \right] \cup \left[ \frac{5\pi}{4}, \frac{3\pi}{2} \right] \cup \left[\frac{\pi}{4} , \frac{\pi}{2}\right]}\Re\left( \frac{  \partial_r A - \frac{A}{r} + i  \left[ \partial_r B- \frac{B}{r} \right]}{  \sin\theta + \frac{A}{r} + i \frac{B}{r}} \right)(\varepsilon e^{i\theta})d \theta \right| &\le  \int_{ \left[ -\frac{\pi}{2} , -\frac{\pi}{4} \right] \cup \left[ \frac{5\pi}{4}, \frac{3\pi}{2} \right]\cup \left[\frac{\pi}{4} , \frac{\pi}{2} \right]} \frac{4C_1\varepsilon}{\left|\sin\theta + \frac{A}{r}  \right|} \\[3mm]
&\le \int_{\left[ -\frac{\pi}{2} , -\frac{\pi}{4} \right] \cup \left[ \frac{5\pi}{4}, \frac{3\pi}{2} \right] \cup \left[\frac{\pi}{4} , \frac{\pi}{2} \right]} \frac{16C_1\varepsilon}{ \sqrt{2}}=O(\varepsilon).
\end{aligned}
\end{equation}
However, near the graph, the real part of the denominator no longer produces a good enough bound by itself, and the imaginary part also converges toward $0$. Indeed let $Z(t)=t +i U(t)$. Since $U(0)=U'(0)=0$, one has the expansion $U(t)= \sum_{k=2}^{+\infty} u_k t^k$, which is converging for all $|t|\leq t_1$ such that $Z(t) \in \D_{r_1}$. Moreover, we have the following estimate
$$
|Z(t)| = \sqrt{t^2 + U^2(t)} \ge |t| .
$$ 
Using the same argument as in the proof of \Cref{lesinegalitessurlesfonctionsanalytiques}, the function $\big( t\mapsto |U(t)| t^{-2}\big)$ is uniformly bounded and this shows there exists $r_3 \le r_2$ and a constant $C_2>0$ such that for all $Z(t) \in \D_{r_3} $
\begin{equation}
\label{lercestlet}
|t| \le |Z(t)| = |t| \sqrt{ 1 + t^2\left(\frac{U}{t^2}\right)^2}\le C_2 |t|.
\end{equation}
Now the function $B\circ Z$ is a non-null (since the origin is an isolated umbilic point) real-analytic function, such that $B\circ Z(0)= (B\circ Z)'(0)=0$ since $B(0)=0$ and $\g B(0)=0$. Hence, there exist $s_1 \ge 2$, $r_4 \le r_3$ and $C_3>0$ such that for all $Z(t) \in \D_{r_4}$ it holds
\begin{equation}
\label{Bsurgraph1}
|B(Z(t))| \ge C_3 |t|^{s_1} \ge \frac{C_3}{C_2^{s_1} } |Z(t)|^{s_1}.
\end{equation}
If we denote $t_r^{\pm}$ the positive (respectively negative) parameters for which $|Z(t_r^\pm)|= r$, then we define $\theta_r^{\pm}$ be measures of the argument of $Z(t_r^\pm)$ respectively in $\left[-\frac{\pi}{2}, \frac{\pi}{2} \right]$ and $\left[\frac{\pi}{2}, \frac{3\pi}{2} \right]$ by the formulas
$$
\begin{aligned}
\theta_r^+ &\coloneqq \arctan \left( \frac{U(t_r^+)}{t_r^+} \right) = \arctan\left(\sum_{k=1}^{\infty} u_{k+1}\, (t_r^+)^k \right) \in \left[-\frac{\pi}{2}, \frac{\pi}{2} \right],  \\
\theta_r^- &\coloneqq \pi + \arctan \left( \frac{U(t_r^-)}{t_r^-} \right) = \pi+ \arctan\left( \sum_{k=1}^{\infty} u_{k+1}\, (t_r^-)^k \right) \in \left[\frac{\pi}{2}, \frac{3\pi}{2} \right].
\end{aligned}
$$
Once more, up to restricting the disk of study by a fixed factor, the estimate $|U(t)|\leq C\, |t|^2$ together with \eqref{lercestlet} and $|Z(t_r^{\pm})|=r$, imply that there exists a constant $C_4$ such that:
\begin{equation}\label{eq:est_theta_pm}
\left| \theta_r^+\right|+ \left| \theta_r^- - \pi \right| \le C_4\, r.
\end{equation}
One can then rewrite \eqref{Bsurgraph1} into 
\begin{equation}
\label{Bsurgraph2}
\forall r \le r_4, \qquad \left| B\left(r e^{i\, \theta^{\pm}_r} \right) \right| \ge C_5\, r^{s_1},
\end{equation}
with $C_4 =\frac{C_3}{C_2^{s_1}}$ and $s_1\ge2$.

We will now work  around each $\theta^\pm_r$ separately on a given circle. Let us consider for $r \le r_4$,  the function $f_r \colon \theta \mapsto B(re^{i\theta})$. Since the map $\left( \theta\in\R\mapsto e^{i\, \theta} \right)$ is real-analytic, the function $f_r $ is also a real-analytic function of $\theta$ which does not cancel at $\theta_r$ thanks to \eqref{Bsurgraph2}. Its zeroes are then isolated which allows one to define 
$$
	\mu_r^{+} \coloneqq \min \left( \frac{\pi}{8}, \min \{ \mu \ge0 \text{ s.t. } f_r(\theta_r + \mu)=0 \} \right)>0.
$$
\begin{claim}
	There exists a radius $r_5>0$ and a constant $C_6>0$ such that for all $r\leq r_5$, it holds
	\begin{align}\label{controlemu+r}
		\mu_r^+ \geq C_6\, r^{s_1-1}.
	\end{align}
\end{claim}

\begin{proof} 
Since $f_r$ is a $C^1$-function, one has thanks to claim \ref{lesinegalitessurlesfonctionsanalytiques}:
$$\begin{aligned}
\left| f_r(\theta_r^+) - f_r(\theta_r^++ \mu_r^+) \right| &\le \sup_{\theta \in [\theta_r^+, \theta_r^++ \mu_r^+]} \left| f_r'(\theta) \right| \mu_r^+ \le\sup_{\theta \in [\theta_r^+, \theta_r^++ \mu_r^+]} \left| \partial_\theta B\left(r e^{i\theta}\right) \right| \mu_r^+ \le r^2\, C_1 \, \mu_r^+.
\end{aligned}$$
If $\theta_r +\mu_r^+$ is  a zero of $f_r$, then we deduce from the above inequality that
$$
\begin{aligned}
	\mu^+_r \ge \frac{|f_r(\theta_r^+)|}{C_1 r^2} \ge \frac{\left| B\left(r e^{i \theta^\pm_r} \right) \right| }{C_1r^2} \ge \frac{C_5}{C_1}r^{s_1-2} \ge \frac{C_5}{C_1}r^{s_1-1}. 
\end{aligned}
$$
Else,  since $s_1\ge 2$, one has for $r\le \min\left( \frac{C_1\pi}{8C_5},  r_4\right)$, 
$$\begin{aligned}
\mu^+_r = \frac{\pi}{8} \ge \frac{C_5}{C_1} r  \ge \frac{C_5}{C_1} r^{s_1-1}.
\end{aligned}
$$
Therefore, in all cases the following stands: there exists a radius $r_5\le r_4$  and a constant $C_6$ such that for all $r\le r_5$
\begin{equation*}
	\mu^+_r \ge C_6 r^{s_1-1}.
\end{equation*}
\end{proof} 

\begin{remark}
Of course taking $s_1-1$ is not optimal. In particular when $s_1= 2$, one could then prove that $\mu^+_r$ is at a fixed distance of $\theta_r^+$ and thus the zeros of $y+A$ and $B$ do not interfere.  However the case $s_1>2$ cannot a priori be excluded (it is precisely what happens when the closest zero of $B$ on a circle gets arbitrarily close to the graph as we approach the umbilic point), and then taking $s_1-1$ does not fundamentally change the behavior of $\mu^+_r$. This non-optimal choice then allows us to avoid writing a disjunction of cases.
\end{remark}

We now turn to the real part. We define the function $g_r \colon \theta \mapsto \frac{y+A}{r}(re^{i\theta})$ which is a real-analytic function. Thanks to \eqref{lecontroledeyplusAoucava2} and \eqref{eq:est_theta_pm} we have that 
\begin{align}\label{eq:der_gr} 
	\forall \theta \in \left[\theta_r^+, \frac{\pi}{4} \right], \qquad g_r'(\theta) = \partial_\theta \left(\frac{ y +A}{r} \right) > \frac{\sqrt{2}}{4}. 
\end{align} 
Thus $g_r$ is increasing and we obtain
\begin{equation}
	\label{controley+Amur+0}
	\forall \theta \in \left[ \theta_r^+ + \mu^+_r, \frac{\pi}{4} \right], \qquad \left( \frac{y+A}{r}\right) \left(r e^{i\theta} \right) \ge  \left( \frac{y+A}{r}\right) \left(r e^{( \theta_r^+ + \mu^+_r)} \right) .
\end{equation}
\begin{remark}
	\label{lesignedey+A}
	Since $\{y+A=0\}$ in a small disk is a graph, the sign of $\left( \frac{y+A}{r}\right) \left(r e^{i\left[\theta^+_r + \mu^+_r\right]} \right)$ is the same as $\left( \frac{y+A}{r}\right) \left(r e^{i\frac{\pi}{4}} \right)$, meaning positive. 
\end{remark}
Thanks to Rolle's theorem combined with \eqref{eq:der_gr} and \eqref{controlemu+r}, we have 
$$
	\begin{aligned}
	\left| g_r(\theta_r^+) - g_r(\theta_r^++ \mu_r^+) \right| &\ge \inf_{\theta \in [\theta_r^+, \theta_r^++ \mu_r^+]} \left| g_r'(\theta) \right| \mu_r^+ \ge \frac{C_6 \sqrt{2}}{4}r^{s_1-1}.
	\end{aligned}
$$
Since $g_r(\theta_r^+)= 0$, \Cref{lesignedey+A} yields 
\begin{equation}
\label{controley+Amur+}
\left( \frac{y+A}{r}\right) \left(r e^{i\left[\theta^+_r + \mu^+_r\right]} \right) \ge  \frac{C_6 \sqrt{2}}{4}r^{s_1-1}.
\end{equation}
Combined with \eqref{controley+Amur+0}, we obtain
\begin{equation}
	\label{controley+Amur+bis}
	\forall \theta \in \left[ \theta_r^+ + \mu^+_r, \frac{\pi}{4} \right], \qquad \left( \frac{y+A}{r}\right) \left(r e^{i\theta} \right) \ge  \frac{C_6 \sqrt{2}}{4}r^{s_1-1}.
\end{equation}
The above estimate provides an asymptotic estimate for the last term of \eqref{eq:decompo_typeIII} on $\left[\theta_r + \mu^+_r, \frac{\pi}{4} \right]$:
\begin{equation} \label{eq:lintegralelaoucestmoinsok_0}
\begin{aligned}
	& \left| \int_{ \theta^+_r+ \mu^+_r}^{\frac{\pi}{4} }\Re\left( \frac{  \partial_r A - \frac{A}{r} + i  \left[ \partial_r B- \frac{B}{r} \right]}{  \sin\theta + \frac{A}{r} + i \frac{B}{r}} \right)(\varepsilon e^{i\theta})\, d \theta \right| \\[3mm] 
	&\le  \int_{ \theta^+_r+ \mu^+_r}^{\frac{\pi}{4} } \frac{4\, C_1\, \varepsilon}{\left|\sin\theta + \frac{A}{r}  \right|} (\varepsilon e^{i\theta}) \, d\theta \\[3mm]
	&\le  \int_{ \theta^+_r+ \mu^+_r}^{\frac{\pi}{4} } \frac{4\, C_1\, \varepsilon}{\frac{1}{2} \left( \sin\theta + \frac{A}{r}  \right) +\frac{1}{2} \left( \sin\theta + \frac{A}{r}  \right)  }(\varepsilon\, e^{i\theta})\, d\theta  \\[3mm]
	&\le  \int_{ \theta^+_r+ \mu^+_r}^{\frac{\pi}{4} } \frac{4\, C_1\, \varepsilon\, \frac{4}{\sqrt{2}}\, \partial_\theta \left(\sin \theta + \frac{A}{r} \right)}{\frac{1}{2} \left( \sin\theta + \frac{A}{r}  \right) +\frac{C_6 \sqrt{2}}{8}r^{s_1-1}    }(\varepsilon\, e^{i\theta})\, d\theta  \\[3mm]
	&\le \left[ C \varepsilon \log \left( \frac{1}{2} \left( \sin \theta + \frac{A}{r} \right)(r= \varepsilon)  + \tilde C \varepsilon^{s_1-1} \right) \right]_{\theta=\theta_r^+ + \mu_r^+}^{\theta= \frac{\pi}{4}}.
\end{aligned}
\end{equation} 
Using \eqref{controley+Amur+bis}, we obtain
\begin{align*}
	\left[ C \varepsilon \log \left( \frac{1}{2} \left( \sin \theta + \frac{A}{r} \right)(r= \varepsilon)  + \tilde C \varepsilon^{s_1-1} \right) \right]_{\theta = \theta_r^+ + \mu_r^+}^{\theta = \frac{\pi}{4}} \leq C\, \ve \log\left[\frac{C}{\ve^{s_1-1}}\right].
\end{align*}
We end up with
\begin{equation}
\label{lintegralelaoucestmoinsok}
\left| \int_{ \theta^+_r+ \mu^+_r}^{\frac{\pi}{4} }\Re\left( \frac{  \partial_r A - \frac{A}{r} + i  \left[ \partial_r B- \frac{B}{r} \right]}{  \sin\theta + \frac{A}{r} + i \frac{B}{r}} \right)(\varepsilon e^{i\theta})d \theta \right|  = O\big(\varepsilon |\log \varepsilon|\big).
\end{equation}
To obtain the control we need on $[\theta^+_r , \theta^+_r + \mu^+_r ]$, we introduce the function 
$$
h_r\colon \theta \in [\theta^+_r , \theta^+_r + \mu^+_r ] \mapsto \frac{1}{2} \left| \frac{(y+A)(r e^{i\theta})}{r}\right|   + \left| \frac{B(re^{i\theta})}{r}\right|.
$$
From Remark \ref{lesignedey+A} and since $B(re^{i\theta})$ does not vanishes on $[\theta^+_r , \theta^+_r + \mu^+_r ]$, the sign of all involved quantities is fixed and there exists a fixed number $\sigma\in \{0,1\}$ such that 
$$
\forall \theta\in [\theta^+_r , \theta^+_r + \mu^+_r ],\qquad  h_r(\theta) = \frac{(y+A)(r e^{i\theta})}{r} + (-1)^{\sigma} \frac{B(re^{i\theta})}{r}.
$$ 
Thus $h_r$ is real-analytic and from Claim \ref{lesinegalitessurlesfonctionsanalytiques}, Claim \ref{lecontroledeyplusAoucava} and \eqref{eq:est_theta_pm}, we have
$h_r'(\theta)  \ge \frac{\sqrt{2}}{8}$ for $r$ small enough. By \eqref{Bsurgraph2}, we obtain that for any $\theta \in [\theta^+_r , \theta^+_r + \mu^+_r ]$, it holds
\begin{equation}
\label{lecontrolelaoucestpasok}
\frac{1}{2} \left| \frac{(y+A)(r e^{i\theta})}{r}\right|   + \left| \frac{B(re^{i\theta})}{r}\right| \ge \left| \frac{B(re^{i\theta_r^+})}{r}\right| \ge C_5 r^{s_1-1}.
\end{equation} 
This yields 
$$
\begin{aligned}
	& \left| \int_{ \theta^+_r}^{\theta^+_r+ \mu^+_r}\Re\left( \frac{  \partial_r A - \frac{A}{r} + i  \left[ \partial_r B- \frac{B}{r} \right]}{  \sin\theta + \frac{A}{r} + i \frac{B}{r}} \right)(\varepsilon e^{i\theta})d \theta \right|\\[3mm]
	&\le \int_{ \theta^+_r}^{\theta^+_r+ \mu^+_r} \frac{8C_1\varepsilon}{\left|\sin\theta + \frac{A}{r}  \right| + \left|\frac{B}{r}\right|} (\varepsilon e^{i\theta}) d\theta \\[3mm]
	&\le  \int_{\theta^+_r}^{ \theta^+_r+ \mu^+_r} \frac{8C_1\varepsilon}{\frac{1}{2} \left( \sin\theta + \frac{A}{r}  \right) + \left[\frac{1}{2} \left( \sin\theta + \frac{A}{r}  \right) + \left|\frac{B}{r}\right| \right] }(\varepsilon e^{i\theta}) d\theta  \\[3mm]
&\le  \int_{ \theta^+_r}^{\theta^+_r+ \mu^+_r} \frac{8C_1}{\frac{1}{2} \left( \sin\theta + \frac{A}{r}  \right) +C_5r^{s_1-1}    }(\varepsilon e^{i\theta}) d\theta. 
\end{aligned}
$$
Using \eqref{lecontroledeyplusAoucava2} and proceeding in the same manner as in \eqref{eq:lintegralelaoucestmoinsok_0}, we can bound from above the numerator by the tangential derivative of the denominator. Integrating as a logarithm, we obtain as in \eqref{lintegralelaoucestmoinsok}:
\begin{equation}
\label{lintegralelaoucestpasok}
\left| \int_{ \theta^+_r}^{ \theta^+_r + \mu^+_r}\Re\left( \frac{  \partial_r A - \frac{A}{r} + i  \left[ \partial_r B- \frac{B}{r} \right]}{  \sin\theta + \frac{A}{r} + i \frac{B}{r}} \right)(\varepsilon e^{i\theta})d \theta \right| = O\big(\varepsilon |\log \varepsilon|\big).
\end{equation}
We now combine \eqref{lintegralelaoucestok}, \eqref{lintegralelaoucestmoinsok} and \eqref{lintegralelaoucestpasok} to obtain
\begin{align*}
	\left| \int_{\theta^+_r}^{ \frac{\pi}{2}}\Re\left( \frac{  \partial_r A - \frac{A}{r} + i  \left[ \partial_r B- \frac{B}{r} \right]}{  \sin\theta + \frac{A}{r} + i \frac{B}{r}} \right)(\varepsilon e^{i\theta})d \theta \right| = O\big(\varepsilon |\log \varepsilon|\big).
\end{align*}
Working similarly, first by introducing $\mu^-_r$ the closest zero of $B$ below $\theta^+_r$, and using the same estimates, and then  around $\theta^-_r$ with the same method yields
\begin{equation} \label{contribpointumbisolsing}\begin{aligned}
\int_{\partial \mathcal{U}_\varepsilon} \partial_\nu \rho\, d\mathrm{vol}_h
&= 2\pi\, n(p) + O( \varepsilon)  + \int_{-\frac{\pi}{2}}^{\frac{3\pi}{2}}\Re\left( \frac{  \partial_r A - \frac{A}{r} + i  \left[ \partial_r B- \frac{B}{r} \right]}{  \sin\theta + \frac{A}{r} + i \frac{B}{r}} \right)(\varepsilon e^{i\theta})d \theta \\[3mm]
&=  2\pi\, n(p)  + O\big(\varepsilon |\log \varepsilon|\big) .
\end{aligned}\end{equation}
This concludes the study of isolated umbilic points.

	\subsection{Contribution of umbilic curves}
	\label{contributionofumbiliccurves}

	Let us now consider an umbilic curve $\Gamma$ containing some singular umbilic points $p_1,\dots,p_q$  of respective order $n_1,\dots,n_q$. From \Cref{th:structure_umbilic_set} we know that $\Gamma$ is a smooth closed submanifold of $\Sigma$ of finite length. Let $L_h$ be the length of $\Gamma$ computed with the background metric $h$  and let $\gamma\colon \s_{L_h/(2\pi)} \rightarrow \Sigma$ be an arc-length parametrization of $\Gamma$ for $h$. We choose a $h$-unit normal $\nu\colon \s_{L_h/(2\pi)} \rightarrow T\Sigma$ to $\Gamma$ and build smooth coordinates in a tubular neighborhood of $\Gamma$ in the following manner:
	\begin{equation}
		\label{definitionlediffeo}
		\Phi\colon \left\{ \begin{aligned}
			(-a, a) \times  \s_{\frac{L_h}{2\pi}} &\rightarrow \mathcal{V} \subset \Sigma \\[2mm]
			(s, \theta) &\mapsto \exp_{\gamma(\theta)} \left( s \nu (\theta) \right),
		\end{aligned} \right.
	\end{equation}
	where $a >0$ is chosen small enough for $\Phi$ to be a smooth diffeomorphism, and such that $\mathcal{U}\cap \mathcal{V} = \Gamma$. We then define 
	\begin{equation}
		\label{definitionlevoisinage}
		\mathcal{U}_\varepsilon (\Gamma) = \Phi \left( (- \varepsilon, \varepsilon) \times \s_{\frac{L_h}{2\pi}} \right).
	\end{equation}
	Denoting $\gamma_s = \Phi \left( \{ s \} \times . \right)$, and $\nu_s$ the unit normal to $\gamma_s$ such that $\lim_{s\rightarrow 0} \nu_s = \nu$. Noticing that the exterior pointing normal  to $\mathcal{U}_{\varepsilon} (\Gamma)$ is respectively $\nu_\varepsilon$ and $-\nu_{-\varepsilon}$, one has
	\begin{equation}
		\label{differencedesdeuxintegrales}
		\begin{aligned}
			\int_{\partial \mathcal{U}_\varepsilon(\Gamma) } \partial_{\vec{n}} \rho\, d\mathrm{vol}_h&= 
			\int_{\gamma_\varepsilon \left(  \s_{\frac{L_h}{2\pi}} \right) } \partial_{\nu_\varepsilon( \theta)} \rho (\varepsilon, \theta)\, d\mathrm{vol}_h(\varepsilon, \theta)-\int_{\gamma_{-\varepsilon} \left(  \s_{\frac{L_h}{2\pi}} \right) } \partial_{\nu_{- \varepsilon}( \theta)} \rho (-\varepsilon, \theta)\, d\mathrm{vol}_h(-\varepsilon, \theta) \\[2mm]
			&=\int_{- \frac{L_h}{2}}^{\frac{L_h}{2}} \Big[ \partial_{\nu_\varepsilon} \rho (\varepsilon, \theta)\, |\gamma_\varepsilon'|_h -\partial_{\nu_{-\varepsilon}} \rho (-\varepsilon, \theta)\, |\gamma_{-\varepsilon}'|_h \Big]\ d\theta.
		\end{aligned}
	\end{equation}
	
	Let us now consider $p \in \Gamma$, which may be a singular umbilic point. From \eqref{typeIVest}, we know there exist complex coordinates $z$ defined on a small neighbourhood $p \in \mathcal{V}_p \subset \mathcal{V}$, an integer $n \in \mathbb{N}$, a $C^{\infty}$ complex-valued function $\mathcal{F}$ and a constant $C_p$ such that
	\begin{equation}\label{defF}
		\left\{
		\begin{aligned}
			&|\varphi|\, e^{-\lambda} = |z|^n\, |\mathcal{F}|, \\[3mm]
			& \forall \gamma(\theta) \in \mathcal{V}_p, \quad \mathcal{F}(0, \theta) = \partial_\theta \mathcal{F}(0,\theta)= 0, \\[3mm]
			& \forall \gamma(\theta) \in \mathcal{V}_p, \quad \left| \partial_s \mathcal{F}(0,\theta) \right| \ge C_p >0.
		\end{aligned}
		\right.
	\end{equation}

It must be noted that since we  are working on a smooth compact Willmore immersion the insertion of the conformal factor in \eqref{defF} from \eqref{typeIVest} does not change the relevant estimates.

	The definition $\rho = \frac{1}{2} \log \left( \left|\varphi e^{-\lambda} \right|^2\right)$ together with \eqref{defF} leads to 
	\begin{equation}\label{eq:derivative_rho}
		\partial_{\nu_s} \rho = \frac{n}{2}\, \frac{\partial_{\nu_s} |z|^2}{|z|^2} + \Re\left( \frac{\partial_{\nu_s} \mathcal{F}}{\mathcal{F}}  \right).
	\end{equation}
	We define the two following quantities
	\begin{align}
		\label{defVW}
		V(s,\theta) \coloneqq \frac{\partial_{\nu_s} \mathcal{F}(s, \theta)}{ \mathcal{F}(s, \theta)} |\gamma'_s|, & & W(s,\theta) \coloneqq \frac{ \partial_{\nu_s} |z|^2}{2|z|^2} |\gamma'_s| =  \frac{\langle \nu_s . \left(z\circ \Phi (s, \theta)\right), z \rangle }{|z|^2} |\gamma'_s|.
	\end{align}
	The proof will rely on an expansion in $s$ of the meaningful terms in \eqref{differencedesdeuxintegrales}. The true curve contribution will be given by $V$ which produces a singularity, while the point contribution $W$ will converge, but requires an expansion in both $s$ and $\theta$. We first compute the asymptotic expansion of the curves $\gamma_s := \Phi(s,\cdot)$, its normal vector and its derivative. Then we compute the asymptotic expansion of $V$. Finally, we will compute the one of $W$.
	
	\subsubsection*{Line element and normal}
	
	In this first part we detail expansions for Fermi coordinates around a curve, see for instance \cite[chapter 5]{lee2018}. 
	
	Denoting $\Phi^i(s,\theta)= \gamma_s^i(\theta)$ the $i$-th coordinate of $\Phi$ in the $z$ chart and $^h \Gamma$ the Christoffel symbols of the metric $h$, one can expand using that $s\mapsto \gamma_s(\theta)$ is a geodesic:
	\begin{equation}
		\left\{
		\begin{aligned}
			&\Phi^i(s, \theta) = \Phi^i(0, \theta) + s\, \partial_s \Phi^i(0, \theta) + \frac{s^{2}}{2}\, \partial_s^2 \Phi^i(0,\theta) + O(s^3),\\[2mm]
			&\partial_s \Phi^i(0, \theta) = \nu^i(\theta),\\[2mm]
			&\partial_s^2 \Phi^i(0 , \theta) = - {^h\Gamma}^{\, i}_{\, ab}\ \nu^a(\theta)\, \nu^b(\theta).
		\end{aligned}
		\right.
	\end{equation}
	In the rest of the section, we will denote $'$ the differentiation with respect to the variable $\theta$. Since in the $z$ coordinates, $h= e^{2\tau} \delta$, and since by definition $|\nu|_h = |\gamma'|_h =1$, while $\langle \nu, \gamma'\rangle_h = 0$, an explicit computation of the Christoffel symbols in a conformal chart implies 
	\begin{equation}
		\label{expansionds2Phi}
		\partial_s^2 \Phi^i(0 , \theta) =(\gamma'.\tau)\, {\gamma^i}' - (\nu.\tau)\, \nu^i.
	\end{equation}
	Thus we obtain the following asymptotic expansion:
	\begin{equation}
		\label{developgammas}
		\gamma^i_s(\theta) \ust{s\to 0}{=} \gamma^i(\theta) + s\, \nu^i(\theta) + \frac{s^2}{2} \left( (\gamma'.\tau)\, {\gamma^i}' - (\nu.\tau)\, \nu^i \right) + O(s^3).
	\end{equation}
	We then differentiate the following relations:
	$$
	\left\{
	\begin{aligned}
		& |\nu|^2_h =1, \\[2mm]
		& \langle \nu, \gamma' \rangle_h= 0.
	\end{aligned}
	\right.
	$$
	We obtain: 
	\begin{equation*}
		{\nu^i}'(\theta)= -\left( k_h(\theta)+ \nu. \tau \right) {\gamma^i}'(\theta) -(\gamma'.\tau)\, {\nu^i}(\theta),
	\end{equation*}
	with $k_h$ the curvature of $\Gamma$. Coming back to \eqref{developgammas}, it holds
	\begin{equation}
		\label{developgammathetas}
		{\gamma^i}'_s(\theta) = \partial_\theta\Phi^i(s, \theta) \ust{s\to 0}{=} \big(1-s \left[  k_h(\theta)+ \nu. \tau \right]\big)\, {\gamma^i}'(\theta) - s\, (\gamma'.\tau)\, \nu^i(\theta)+O(s^2).
	\end{equation}
	Consequently the line element satisfies
	\begin{align*}
		|\gamma_s'(\theta)|^2_h &\ust{s\to 0}{=} e^{2\tau\left( \gamma(\theta) + s\, \nu(\theta) + O(s^2) \right)}\left|\big(1-s \left[  k_h(\theta)+ \nu. \tau \right]\big)\, {\gamma^i}'(\theta) - s\, (\gamma'.\tau)\, \nu^i(\theta)+O(s^2) \right|^2 \\[3mm]
		& \ust{s\to 0}{=} e^{2\tau\big(\gamma(\theta)\big)}\left[1+2s\, (\nu.\tau)+O(s^2)\right]\left( \big(1-s \left[  k_h(\theta)+ \nu. \tau \right]\big)^2 |\gamma'(\theta)|+s^2(\gamma'.\tau)^2 |\nu(\theta)|^2 + O(s^4) \right) \\[3mm]
		& \ust{s\to 0}{=}  \left[ 1-2s\, k_h(\theta) + O(s^2) \right].
	\end{align*}
	We end up with the following expansion:
	\begin{equation}
		\label{developlineelement}
		|\gamma'_s(\theta)|_h \ust{s\to 0}{=} e^{\tau \big(\gamma(\theta) + s\, \nu(\theta) + O(s^2) \big)}\, |\gamma'_s (\theta) | \ust{s\to 0}{=} \left(1  - s\, k_h(\theta) +O(s^2)\right).
	\end{equation}
	Similarly, by definition of $\nu_s$ and $\gamma_s$, we have:
	$$\left\{
	\begin{aligned}
		& \langle \nu_s, \nu_s \rangle_h= \langle \nu_s , \nu_s \rangle e^{2\tau} = 1, \\[2mm]
		& \langle \nu_s, \gamma_s' \rangle_h= \langle \nu_s , \gamma_s' \rangle e^{2\tau} = 0.
	\end{aligned} \right.
	$$
	Taking the $s$-derivative at $0$ and using \eqref{developgammathetas} yields:
	$$\left\{
	\begin{aligned}
		& \langle \partial_s {\nu_s}_{|s=0}, \nu \rangle_h= - \nu.\tau, \\[2mm]
		& \langle  \partial_s {\nu_s}_{|s=0}, \gamma' \rangle_h= \gamma'.\tau.
	\end{aligned} \right.
	$$
	The expansion of the normal $\nu_s$ is:
	\begin{equation}
		\label{developnormalzcoordinates}
		\nu_s(\theta) \ust{s\to 0}{=} \nu(\theta) + s\, \big( (\gamma'.\tau)\, \gamma'(\theta) - (\nu.\tau)\, \nu(\theta) \big) +O(s^2).
	\end{equation}
	In other words, it holds $\nu_s(\theta) \ust{s\to 0}{=} \partial_s \Phi(0,\theta) + O(s^2)$ using \eqref{expansionds2Phi} and \eqref{developgammathetas}. Thus, in $(s,\theta)$ coordinates one has
	\begin{align}
		\label{developnormalzcoordinates2}
		\nu_s(\theta)^s \ust{s\to 0}{=} 1 +O(s^2), & &\nu_s(\theta)^\theta \ust{s\to 0}{=} O(s^2).
	\end{align}
	Hence, for any $f\in C^\infty(D)$, it holds $\nu_s.f \ust{s\to 0}{=} \partial_s f + O(s^2)$.
	
	\subsubsection*{Expansion of $V$: the singular contribution of the curve}
	
	In this section we will compute the contribution of the umbilic curve proper which as announced will come from the $V$ terms, see \eqref{defVW}. One may first notice that, since we are working with a symmetric domain around the curve, the contribution in $V(\varepsilon)- V(-\varepsilon)$ will cancel out all the even exponents of the expansions and keep only the odd terms. We will nonetheless keep the constant order terms in the computations below for clarity.
	
	\begin{claim}
		It holds
		\begin{equation}
			\label{lacontributiondelacourbe}
			\begin{aligned}
				\left[V(\varepsilon, \theta) - V(-\varepsilon,\theta)\right] d\theta = \left[\frac{2}{\varepsilon} + O(\varepsilon) \right] d\theta.
			\end{aligned}
		\end{equation}
	\end{claim}
	
	\begin{proof}
		We develop $\mathcal{F}$ and its derivative using \eqref{defF}:
		$$ \left\{ \begin{aligned} 
			\mathcal{F} (s, \theta) &= \mathcal{F}(0, \theta) + s\, \partial_s \mathcal{F}(0, \theta) + \frac{s^2}{2}\, \partial_s^2\mathcal{F}(0,\theta) + O(s^3) \\[2mm]
			& =s\, \partial_s \mathcal{F}(0, \theta)  \left(  1+ \frac{s}{2\, \partial_s \mathcal{F}(0, \theta)}\, \partial_s^2\mathcal{F}(0,\theta) + O(s^2) \right),  \\[3mm]
			\partial_s \mathcal{F}(s, \theta) &= \partial_s \mathcal{F}(0, \theta) + s\, \partial_s^2\mathcal{F}(0,\theta) + O(s^2), \\[3mm]
			\partial_\theta \mathcal{F}(s, \theta) &= \partial_\theta \mathcal{F}(0, \theta) + O(s) = O(s),
		\end{aligned} \right.$$
		where the constants hidden in the $O(s^i)$ depend on $\| \mathcal{F}\|_{C^3}$ and $C_p$. From this, \eqref{developlineelement} and \eqref{developnormalzcoordinates}, we obtain:
		$$\begin{aligned} 
			V(s,\theta) &= \frac{ \partial_{\nu_s(\theta)} \mathcal{F}(s, \theta)}{\mathcal{F}(s,\theta)}\, |\gamma'|_h(s,\theta) \\[3mm]
			&= \frac{ \left( 1  - s\, k_h(\theta) +O(s^2) \right) \left( \partial_s \mathcal{F}(s,\theta)   + O(s^2) \right)}{\mathcal{F}(s,\theta)} \\[3mm]
			&=\frac{ \partial_s \mathcal{F}(s,\theta) - k_h\,  s\, \partial_s \mathcal{F}(s,\theta)+O(s^2) }{\mathcal{F}(s,\theta)} \\[3mm]
			&=\frac{ \partial_s \mathcal{F}(0,\theta) + s\, \big( \partial_s^2 \mathcal{F}(0, \theta) - k_h(\theta)\, \partial_s \mathcal{F}(0, \theta) \big) +O(s^2) }{ s\, \partial_s \mathcal{F}(0, \theta)\left(  1+ \frac{s}{2\, \partial_s \mathcal{F}(0, \theta)}\, \partial_s^2\mathcal{F}(0,\theta) + O(s^2) \right)} \\[3mm]
			&= \frac{1}{s} - \left( \frac{\partial_s^2\mathcal{F}(0,\theta)}{2\partial_s \mathcal{F}(0, \theta)} + k_h(\theta) \right) + O(s).
		\end{aligned}$$
		Thus, it holds:
		\begin{equation*}
			\begin{aligned}
				\big[V(\varepsilon, \theta) - V(-\varepsilon,\theta)\big] d\theta = \left[\frac{2}{\varepsilon} + O(\varepsilon) \right] d\theta.
			\end{aligned}
		\end{equation*}
	\end{proof}

	\subsubsection*{Expansion of $W$: the contribution of singular points}
	We now turn to the contribution of singular points on umbilic curves, which will proceed from the $W$ terms, see \eqref{defVW}. In this case the difficulty stems from the coordinate $\frac{1}{r}$ singularity, in local conformal coordinates centred on a (possibly umbilic singular) point on a curve. Up to a translation on the parameter $\theta$, one can always assume that $z(\Phi(0,0))=\gamma(0)=0$, and then control $W$ by doing a joint expansion on $\theta$ and $s$.
	
	\begin{claim}
		There exists $\theta_0>0$ such that for $\ve>0$ small enough and $|\theta| \le 2 \theta_0$ , it holds
		\begin{equation} \label{lacontributiondespoints}
			\left[  W(\varepsilon,\theta) -W(-\varepsilon,\theta)  \right]d \theta 
			=\frac{2\varepsilon d \theta}{\theta^2 + \varepsilon^2} + O (\varepsilon) = 2\, d\left( \arctan \left( \frac{\theta}{\varepsilon}\right) \right) + O(\varepsilon).
		\end{equation}
	\end{claim}
	
	\begin{proof}
		We start by expanding the denominator of $W$, namely $|\gamma_s(\theta)|^2$. From \eqref{developgammas}, we deduce that 
		\begin{align*} 
			|\gamma_s(\theta)|^2 & =\ |\gamma(\theta)|^2 +2s\, \langle \gamma(\theta), \nu(\theta) \rangle + s^2 \Big( e^{-2\tau(\theta)} + ( \gamma'.\tau )\, \langle\gamma', \gamma\rangle - (\nu.\tau)\, \langle \nu, \gamma \rangle \Big) \\[3mm]
			&\qquad - s^3\, (\nu.\tau)\, e^{-2\tau(\theta)}  +\langle O(s^3), \gamma \rangle +O(s^4).
		\end{align*} 
		We here purposefully kept the $O(s^i)$ terms separated to compare them one by one to the leading order term is given by $|\gamma(\theta)|^2 + s^2 \simeq \theta^2 +s^2$. Knowing that $\gamma(0)= 0$, one can expand in $\theta$ in a neighborhood of $0$: 
		\begin{equation} \label{ordredesgammatheta}
			\left\{
			\begin{aligned} 
				& \gamma (\theta) = \theta\, \gamma'(0) + O(\theta^2), \\[3mm]
				& |\gamma(\theta)|^2 = \theta^2\, e^{-2\tau(0)} + O(\theta^3), \\[3mm]
				& \langle \gamma(\theta), \nu(\theta) \rangle = O(\theta^2), \\[3mm]
				& \langle \gamma(\theta), \gamma'(\theta) \rangle =\theta\, e^{-2\tau(0)} + O(\theta^2).
			\end{aligned}
			\right.
		\end{equation}
		Thus we can write 
		\begin{equation*}
			\begin{aligned}
				|\gamma_s(\theta)|^2 &= \Big( |\gamma(\theta)|^2 +s^2\, e^{-2\tau(\theta)} \Big)\Bigg[ 1+2s\, \frac{ \langle \gamma(\theta), \nu(\theta) \rangle}{|\gamma(\theta)|^2 +s^2\, e^{-2\tau(\theta)}} \\[3mm]
				& \qquad + s^2 \left( (\gamma'.\tau)\, \frac{ \langle\gamma'(\theta), \gamma(\theta)\rangle}{|\gamma(\theta)|^2 +s^2\, e^{-2\tau(\theta)}} - (\nu.\tau)\, \frac{\langle \nu(\theta), \gamma(\theta) \rangle}{|\gamma(\theta)|^2 +s^2\, e^{-2\tau(\theta)}}\right)  \\[3mm]
				&\qquad  - \frac{ s^3\, (\nu.\tau)\, e^{-2\tau(\theta)}}{|\gamma(\theta)|^2 +s^2\, e^{-2\tau(\theta)}} + \frac{\langle O(s^3), \gamma \rangle }{|\gamma(\theta)|^2 +s^2\, e^{-2\tau(\theta)}} +\frac{O(s^4)}{|\gamma(\theta)|^2 +s^2\, e^{-2\tau(\theta)}} \Bigg].  
			\end{aligned}
		\end{equation*}
		Let us now estimate all terms in the expansion.
		
		\begin{claim}\label{claim1}
			There exist $C_1, \theta_0>0$ such that for all $|\theta|<2\theta_0$, it holds
			\begin{align*}
				\left| \frac{ \langle \gamma(\theta), \nu(\theta) \rangle}{|\gamma(\theta)|^2 +s^2\, e^{-2\tau(\theta)}} \right| \le C_1.
			\end{align*}
		\end{claim}
		\begin{proof}
			From \eqref{ordredesgammatheta}, we can estimate 
			$$
			\frac{ \langle \gamma(\theta), \nu(\theta) \rangle}{|\gamma(\theta)|^2 +s^2\, e^{-2\tau(\theta)}} = \frac{ O(\theta^2) }{\theta^2 +s^2 + O(\theta^3) + O(s^2 \theta)} = \frac{O \left( \frac{\theta^2}{\theta^2+s^2}\right)}{1 + O\left(\frac{ \theta^3 }{\theta^2 +s^2 } \right)+ O\left(\frac{ \theta s^2 }{\theta^2 +s^2 } \right)}.
			$$
			The conclusion follows from the fact that the maps $\left( (\theta, s) \mapsto \frac{\theta^2}{\theta^2+s^2}\right)$ and $\left( (\theta, s) \mapsto \frac{\theta s}{\theta^2+s^2}\right)$ are uniformly bounded.
		\end{proof}
		All other terms estimates follow the same pattern:
		
		\begin{claim}\label{claim2}
			There exist $C_2,\theta_0>0$ such that for all $|\theta|<2\theta_0$, it holds
			\begin{align*}
				\left| \frac{s \langle\gamma'(\theta), \gamma(\theta)\rangle}{|\gamma(\theta)|^2 +s^2\, e^{-2\tau(\theta)}} \right| \le C_2
			\end{align*}
		\end{claim}
		\begin{proof}
			From \eqref{ordredesgammatheta}, we can estimate 
			$$ 
			\frac{s \langle\gamma'(\theta), \gamma(\theta)\rangle}{|\gamma(\theta)|^2 +s^2\, e^{-2\tau(\theta)}} = \frac{s \theta +s O(\theta^2) }{\theta^2 +s^2 + O(\theta^3) +O(s^2\theta)} = \frac{\frac{s\theta}{\theta^2+s^2}+s O \left( \frac{\theta^2}{\theta^2+s^2}\right)}{1 + O\left(\frac{ \theta^3 }{\theta^2 +s^2 } \right) +  O\left(\frac{ \theta s^2 }{\theta^2 +s^2 } \right)}.
			$$
			The conclusion follows from the fact that the maps $\left( (\theta, s) \mapsto \frac{\theta^2}{\theta^2+s^2} \right)$ and $\left( (\theta, s) \mapsto \frac{s\theta}{\theta^2+s^2}\right)$ are uniformly bounded.
		\end{proof}
		
		\begin{claim}\label{claim3}
			There exist $C_3, \theta_0>0$ such that for $|\theta|<2\theta_0$, it holds
			\begin{align*}
				\left|  \frac{\langle O(s), \gamma \rangle }{|\gamma(\theta)|^2 +s^2\, e^{-2\tau(\theta)}} \right| \le C_3.
			\end{align*}
		\end{claim}
		\begin{proof}
			From \eqref{ordredesgammatheta}, we can estimate 
			$$ 
			\frac{\langle O(s), \gamma \rangle }{|\gamma(\theta)|^2 +s^2 e^{-2\tau(\theta)}} = \frac{ O(\theta) O(s) }{\theta^2 +s^2 + O(\theta^3) + O(s^2\theta)} = \frac{O \left( \frac{ s \theta}{\theta^2+s^2}\right)}{1 + O\left(\frac{ \theta^3 }{\theta^2 +s^2 } \right) + O\left(\frac{ \theta s^2 }{\theta^2 +s^2 } \right)}.
			$$ 
			The conclusion follows from the fact that the maps $\left( (\theta, s) \mapsto \frac{\theta^2}{\theta^2+s^2} \right)$ and $\left( (\theta, s) \mapsto \frac{s\theta}{\theta^2+s^2} \right)$ are uniformly bounded.
		\end{proof}
		
		\begin{claim}\label{claim4}
			There exist $C_4,\theta_0>0$ such that for $|\theta|<2\theta_0$, it holds
			\begin{align*}
				\left| \frac{ O(s^2) }{|\gamma(\theta)|^2 +s^2\, e^{-2\tau(\theta)}} \right| \le C_4.
			\end{align*}
		\end{claim}
		\begin{proof}
			From \eqref{ordredesgammatheta}, we can estimate 
			$$
			\frac{ O(s^2)}{|\gamma(\theta)|^2 +s^2\, e^{-2\tau(\theta)}} = \frac{ O(s^2) }{\theta^2 +s^2 + O(\theta^3)+ O(s^2\theta)} = \frac{O \left( \frac{s^2}{\theta^2+s^2}\right)}{1 + O\left(\frac{ \theta^3 }{\theta^2 +s^2 } \right)+ O\left(\frac{ \theta s^2 }{\theta^2 +s^2 } \right)}.
			$$ 
			The conclusion follows from the fact that $(\theta, s) \mapsto \frac{\theta^2}{\theta^2+s^2}$  and $(\theta, s) \mapsto \frac{s^2}{\theta^2+s^2}$ are uniformly bounded.
		\end{proof}
		Thanks to Claims \ref{claim1}-\ref{claim4}, it follows that:
		\begin{align*} 
			|\gamma_s(\theta)|^2 &= \Big( |\gamma(\theta)|^2 +s^2\, e^{-2\tau(\theta)} \Big)\Bigg(1+2s\, \frac{ \langle \gamma(\theta), \nu(\theta) \rangle}{|\gamma(\theta)|^2 +s^2\, e^{-2\tau(\theta)}} \\[3mm]
			& \qquad + s^2\,  (\gamma'.\tau)\, \frac{ \langle\gamma'(\theta), \gamma(\theta)\rangle}{|\gamma(\theta)|^2+s^2\, e^{-2\tau(\theta)}} - \frac{ s^3\, (\nu.\tau)\, e^{-2\tau(\theta)}}{|\gamma(\theta)|^2 +s^2\, e^{-2\tau(\theta)}} +O(s^2) \Bigg).
		\end{align*} 
		In the above relation, we wrote $O(s^2)$ to denote a function $A$ such that $s^{-2} A$ is uniformly bounded in $s$ and $\theta$.  Hence, the following expansion holds, where all the terms in the brackets are of order $O(s)$:
		\begin{equation}
			\label{devgamma2}
			\begin{aligned} 
				|\gamma_s(\theta)|^{-2} &= \frac{1}{ |\gamma(\theta)|^2 +s^2\,  e^{-2\tau(\theta)} }\Bigg(1+\left[- 2s\, \frac{ \langle \gamma(\theta), \nu(\theta) \rangle}{|\gamma(\theta)|^2 +s^2\, e^{-2\tau(0)}} \right. \\[3mm]
				& \qquad \left. - s^2\,  (\gamma'.\tau)\, \frac{ \langle\gamma'(\theta), \gamma(\theta)\rangle}{|\gamma(\theta)|^2+s^2\, e^{-2\tau(\theta)}} + \frac{ s^3\, (\nu.\tau)\, e^{-2\tau(\theta)}}{|\gamma(\theta)|^2 +s^2\, e^{-2\tau(\theta)}}  \right]+O (s^2) \Bigg).
			\end{aligned}
		\end{equation}
		Now we expand the numerator of $W$, see \eqref{defVW}, namely the term $\scal{\nu_s.\gamma_s}{\gamma_s}\, |\gamma_s'|$. We combine \eqref{developnormalzcoordinates2} and \eqref{developlineelement} to obtain:
		\begin{align*}
			(\partial_{\nu_{s}} \gamma^i_s)\, |\gamma'_s|_h &= \left( \nu^i + s \left( (\gamma'.\tau)\, {\gamma^i}' - (\nu.\tau)\, \nu^i \right) +O(s^2) \right) \left( 1-s\, k_h +O(s^2) \right) \\[3mm]
			&= \nu^i + s \left( (\gamma'.\tau)\, {\gamma^i}' - (\nu.\tau)\, \nu^i-  k_h\, \nu^i \right)+O(s^2).
		\end{align*}
		Taking the scalar product with $\gamma_s$ and applying its expansion \eqref{developgammas}, we deduce that
		\begin{equation*}
			\begin{aligned}
				\langle \partial_{\nu_{s}} \gamma_s, \gamma_s \rangle\, |\gamma'|_h &= \langle \nu, \gamma\rangle +s \Big( (\gamma'.\tau)\, \langle \gamma', \gamma\rangle -  (k_h+ \nu.\tau)\,  \langle \nu, \gamma\rangle + \langle \nu ,\nu \rangle \Big)  \\[3mm]
				&\qquad -s^2  \left( (k_h+ \nu.\tau)\, \langle\nu, \nu\rangle + \frac{\nu.\tau}{2}\, \langle \nu, \nu\rangle \right) + O (s^3) + \langle O(s^2), \gamma(\theta) \rangle   \\[3mm]
				&=  \langle \nu, \gamma\rangle + s \Big( (\gamma'.\tau)\, \langle \gamma',\gamma\rangle -(k_h+ \nu.\tau)\, \langle \nu, \gamma\rangle  +e^{-2\tau} \Big)\\[3mm]
				&\qquad - e^{-2\tau}\, s^2 \left(k_h+  \frac{3\, \nu.\tau}{2} \right)  + O(s^3) + \langle O(s^2), \gamma(\theta) \rangle .
			\end{aligned}
		\end{equation*}
		From \eqref{ordredesgammatheta}, we have:
		\begin{align*}
			s\, (k_h + \nu.\tau)\, \scal{\nu}{\gamma} + O(s^3) + \langle O(s^2), \gamma(\theta) \rangle = O(s^3 + s\, \theta^2).
		\end{align*}
		We obtain
		\begin{equation}\label{eq:num_W}
			\begin{aligned}
				\langle \partial_{\nu_{s}} \gamma_s, \gamma_s \rangle\, |\gamma'|_h &= \langle \nu, \gamma\rangle + s \Big( (\gamma'.\tau)\, \langle \gamma',\gamma\rangle  +e^{-2\tau} \Big) - e^{-2\tau}\, s^2 \left(k_h+  \frac{3\, \nu.\tau}{2} \right)  + O(s^3 + s\, \theta^2).
			\end{aligned}
		\end{equation}
		In order to expand the full expression of $W$ in \eqref{defVW} by multiplying \eqref{devgamma2} and \eqref{eq:num_W}, we first multiply \eqref{eq:num_W} with the first coefficient of \eqref{devgamma2}. Claims \ref{claim1}-\ref{claim4} then ensure: 
		\begin{align*}
			\frac{\langle \partial_{\nu_{s, \theta}} \gamma_s( \theta), \gamma_s(\theta) \rangle\, |\gamma'_h|}{ |\gamma(\theta)|^2 +s^2\, e^{-2\tau(\theta)} } &= \frac{e^{-2\tau(\theta)}\, s }{ |\gamma(\theta)|^2 +s^2\, e^{-2\tau(\theta)} } +\frac{s\, (\gamma'.\tau)\, \langle \gamma', \gamma\rangle }{ |\gamma(\theta)|^2 +s^2\, e^{-2\tau(\theta)} } \\[3mm]
			& \qquad +\frac{\langle \nu , \gamma \rangle - \left(\frac{3}{2} \nu.\tau + k_h\right)s^2\, e^{-2\tau(\theta)} }{ |\gamma(\theta)|^2 +s^2\, e^{-2\tau(\theta)} }+O(s).
		\end{align*}
		In the above expansion, the leading term is the first one with order $O\left(\frac{s}{s^2+\theta^2}\right)$. The second term has order $O\left(\frac{s\theta}{s^2 + \theta^2}\right)$. The third term has order $O(1)$. Assembling it with \eqref{devgamma2} then yields:
		\begin{equation}
			\label{finallyaformula}
			\begin{aligned}
				W =&\ \frac{\langle \partial_{\nu_{s}} \gamma_s, \gamma_s \rangle\, |\gamma'_s|_h }{ |\gamma_s|^2} \\[3mm]
				=&\ \frac{e^{-2\tau}s }{ |\gamma|^2 +s^2 e^{-2\tau} } +\frac{s (\gamma'.\tau) \langle \gamma', \gamma\rangle}{ |\gamma|^2 +s^2 e^{-2\tau} } +\frac{\langle \nu(\theta), \gamma(\theta)\rangle - \left(\frac{3}{2} \nu.\tau + k_h\right)s^2 e^{-2\tau(\theta)} }{ |\gamma(\theta)|^2 +s^2 e^{-2\tau(\theta)} }\\[3mm]
				&\qquad -\frac{2s^2 e^{-2\tau}\langle \gamma, \nu\rangle}{ \left( |\gamma|^2 +s^2 e^{-2\tau}\right)^2}-\frac{s^3e^{-2\tau} (\gamma'.\tau) \langle \gamma', \gamma\rangle}{ \left( |\gamma|^2 +s^2 e^{-2\tau}\right)^2}+ \frac{s^4 (\nu.\tau) e^{-4\tau}}{ \left( |\gamma|^2 +s^2 e^{-2\tau}\right)^2}+O(s).
			\end{aligned}
		\end{equation}
		Since we will consider $W(\varepsilon)-W(-\varepsilon)$, all the even powers of $s$ will be eliminated. Thus, we regroup them in the following term:
		\begin{equation}
			\label{onjettelestermespairs}
			\begin{aligned}
				B(s,\theta) & \coloneqq \frac{\langle \nu(\theta), \gamma(\theta) \rangle - \left(\frac{3}{2} \nu.\tau(\theta) + k_h(\theta)\right)s^2 e^{-2\tau(\theta)} }{ |\gamma(\theta)|^2 +s^2 e^{-2\tau(\theta)} }\\[3mm]
				& \qquad -\frac{2s^2 e^{-2\tau(\theta)}\langle \gamma(\theta), \nu(\theta)\rangle}{ \left( |\gamma(\theta)|^2 +s^2 e^{-2\tau(\theta)}\right)^2}
				+ \frac{s^4 \nu.\tau(\theta) e^{-2\tau(\theta)}}{ \left( |\gamma(\theta)|^2 +s^2 e^{-2\tau( \theta)}\right)^2}.
			\end{aligned}
		\end{equation}
		One has $B(s,\theta)= B(-s, \theta)$ and we write \eqref{finallyaformula} as:
		\begin{equation}
			\label{itgetsslimmer}
			\begin{aligned}
				W &=\frac{e^{-2\tau}\, s }{ |\gamma|^2 +s^2\, e^{-2\tau} } +\frac{s\, (\gamma'.\tau)\, \langle \gamma', \gamma\rangle}{ |\gamma|^2 +s^2\, e^{-2\tau} }-\frac{s^3\, e^{-2\tau}\, (\gamma'.\tau)\, \langle \gamma', \gamma\rangle }{ \left( |\gamma|^2 +s^2\, e^{-2\tau}\right)^2}+ B+O(s).
			\end{aligned}
		\end{equation}
		We will detail the first three terms. We begin by expanding \eqref{ordredesgammatheta} to the next order:
		$$
		\gamma(\theta) = \theta\, \gamma'(0) + \frac{\theta^2}{2} \Big(  -(\gamma'.\tau)\, \gamma'(0) + (k_h(0) + \nu.\tau)\, \nu(0) \Big) + O(\theta^3).
		$$
		This yields
		$$ 
		\begin{aligned}
			|\gamma|^2 +s^2e^{-2\tau( \theta)}&= \Big(\theta^2 +s^2 -\theta^3\, (\gamma'.\tau)(0) -2\, s^2\, \theta\, (\gamma'.\tau)(0)+ O(\theta^4) + O(s^2 \theta^2)\Big)e^{-2\tau(0)} \\[3mm]
			&= \left( \theta^2 + s^2 \right)e^{-2\tau(0)} \left( 1 -\left( \theta + \frac{\theta s^2}{\theta^2 + s^2} \right)(\gamma'.\tau)(0) + O(\theta^2) \right).
		\end{aligned}
		$$
		For the first term of \eqref{itgetsslimmer}, it holds
		$$
		\begin{aligned}
			\frac{e^{-2\tau(\theta)}\, s }{ |\gamma(\theta)|^2 +s^2\, e^{-2\tau(\theta)} }  &= \frac{s\, e^{-2\tau(0)} (1 - 2 \theta\, (\gamma'.\tau)(0)  + O(\theta^2))}{\left( \theta^2 + s^2 \right)e^{-2\tau(0)} \left( 1 -\left( \theta + \frac{\theta s^2}{\theta^2 + s^2} \right)  (\gamma'.\tau)(0) + O(\theta^2) \right)} \\[3mm]
			&= \frac{s}{s^2 +\theta^2} \left( 1 - 2 \theta (\gamma'.\tau)(0)  +\left( \theta + \frac{\theta s^2}{\theta^2 + s^2} \right)  (\gamma'.\tau)(0) + O(\theta^2) \right)\\[3mm]
			&= \frac{s}{s^2 +\theta^2} - \frac{\theta^3 s}{(\theta^2 + s^2)^2}   (\gamma'.\tau)(0)  +   O(s).
		\end{aligned}
		$$
		For the second term of \eqref{itgetsslimmer}, it holds
		$$
		\begin{aligned} 
			\frac{s\, (\gamma'.\tau)\, \langle \gamma', \gamma\rangle}{ |\gamma(\theta)|^2 +s^2\, e^{-2\tau(\theta)} } &= \frac{s\, \theta\, (\gamma'.\tau)(0) + O(s\, \theta^2)}{(\theta^2+s^2 )\left( 1 -\left( \theta + \frac{\theta s^2}{\theta^2 + s^2} \right)(\gamma'.\tau)(0) + O(\theta^2) \right)} \\[3mm]
			&=\frac{s\theta }{\theta^2+s^2}\, (\gamma'.\tau)(0) +O(s).
		\end{aligned}
		$$
		For the third term of \eqref{itgetsslimmer}, it holds
		$$\begin{aligned}
			\frac{s^3\, (\gamma'.\tau)\, e^{-2\tau}\, \langle \gamma', \gamma\rangle}{ \left( |\gamma|^2 +s^2\, e^{-2\tau}\right)^2} = \frac{s^3\theta }{(\theta^2 + s^2)^2}\, (\gamma'.\tau)(0)  +O(s).
		\end{aligned}$$
		Coming back to \eqref{itgetsslimmer}, we obtain
		\begin{align*}
			W(s,\theta) &= \frac{s}{s^2+\theta^2}- \frac{\theta^3 s}{(\theta^2 + s^2)^2}   (\gamma'.\tau)(0)  + \frac{s\theta }{\theta^2+s^2}(\gamma'.\tau)(0) \\[2mm]
			&\qquad -\frac{s^3\theta}{(\theta^2 +s^2)^2}(\gamma'.\tau)(0)  + O(s)+B(s,\theta) \\[2mm]
			& = \frac{s}{s^2+\theta^2}+ \frac{s\theta(\theta^2+s^2)-\theta^3 s - s^3 \theta }{(\theta^2 + s^2)^2}   (\gamma'.\tau)(0)  + O(s)+B(s,\theta) \\[2mm]
			&= \frac{s}{s^2+\theta^2} + O(s)+B(s,\theta).
		\end{align*}
		Hence, we deduce that:
		\begin{equation*} 
			\begin{aligned}
				\big[  W(\varepsilon,\theta) -W(-\varepsilon,\theta)  \big]d \theta 
				&= 2\, \frac{\varepsilon}{\varepsilon^2+\theta^2}\, d\theta + O(\varepsilon) \\[2mm]
				&= 2\, \frac{ d\left( \frac{\theta}{\varepsilon} \right)}{1+ \left( \frac{\theta}{\varepsilon} \right)^2} +O(\varepsilon) \\[2mm]
				& = 2\, d\left( \arctan \left( \frac{\theta}{\varepsilon}\right) \right) + O(\varepsilon) .
			\end{aligned}
		\end{equation*}
	\end{proof} 
	We now compute the integral contribution of $W$ thanks to \eqref{lacontributiondespoints}. Consider $\eta\in C^{\infty}_c(\gamma([-2\theta_0,2\theta_0]);[0,1])$ a cut-off function such that $\eta(\gamma([-\theta_0,\theta_0])) =1$. We multiply by $\eta$ the above quantity and integrate:
	\begin{equation}\label{lacontributionintegree}\begin{aligned}
			\int_{-\frac{L_h}{2}}^{\frac{L_h}{2} }  \eta ( \gamma(\theta)) \left[  W(\varepsilon,\theta) -W(-\varepsilon,\theta)  \right]d \theta
			=&\ \int_{-\theta_0}^{\theta_0}  \left[  W(\varepsilon) -W(-\varepsilon)  \right] d\theta + O(\ve) \\
			=&\  \int_{-\theta_0}^{\theta_0} 2 d\left( \arctan \left( \frac{\theta}{\varepsilon}\right) \right) +O(\varepsilon) \\
			=&\ 4\arctan\left( \frac{\theta_0}{\varepsilon}\right)  + O(\varepsilon) = 2\pi + O(\varepsilon).
	\end{aligned}\end{equation}
	In the first line we estimated 
$$\begin{aligned}
\left| \int_{[-2\theta_0 , 2\theta_0] \backslash [-\theta_0, \theta_0]} \eta(\gamma(\theta)) \left[  W(\varepsilon) -W(-\varepsilon)  \right] d\theta \right| &\le \int_{[-2\theta_0 , 2\theta_0] \backslash [-\theta_0, \theta_0]} \eta(\gamma(\theta))  \frac{2\varepsilon}{\varepsilon^2 + \theta^2 } d\theta + O(\varepsilon) \\
&\le \int_{[-2\theta_0 , 2\theta_0] \backslash [-\theta_0, \theta_0]} \eta(\gamma(\theta))  \frac{2\varepsilon}{ \theta_0^2 } d\theta + O(\varepsilon) =O(\varepsilon).
\end{aligned}$$
	
	\subsubsection*{The contribution of $\Gamma$}
	Here, we assemble the two contributions (curve and singular umbilic points) using an adapted partition of unity.
	
	We first consider the neighborhoods $\left(\mathcal{V}_{p_i}\right)_{i=1\dots q}$ of the singular umbilic points  $p_1, \dots, p_q$ on $\Gamma$. Up to restricting them, one can assume they cover a small enough part of $\Gamma$ such that the estimates on $\theta$ in claims \ref{claim1}-\ref{claim4} are satisfied for all $p_i$. We complete them into a covering of $\Gamma$ by a finite number of neighborhoods $\left( \mathcal{V}_{\kappa}\right)_{\kappa\in \inter{1}{K}}$, such that there exists a small neighborhood around each $p_i$ that only intersects $\mathcal{V}_{p_i}$ in the covering.  Taking a partition of unity $(\eta_\kappa )_{\kappa\in \inter{1}{K}}$ adapted to this covering, \eqref{differencedesdeuxintegrales} and \eqref{eq:derivative_rho} ensure:
	\begin{align*}
		& \int_{\partial \mathcal{U}_\varepsilon(\Gamma) } \partial_{\nu} \rho\, d\mathrm{vol}_h \\[3mm] 
		=& \int_{-\frac{L_h}{2}}^{\frac{L_h}{2}}  \sum_{\kappa\in \inter{1}{K}} \eta_\kappa (\gamma(\theta))\, \Big[ n_\kappa \big( W_\kappa(\varepsilon) - W_\kappa(-\varepsilon) \big) + \Re\big( V_\kappa(\varepsilon) -V_\kappa(-\varepsilon)\big) \Big]\, d\theta \\[3mm]
		=& \sum_{i=1}^q n_i \int_{-\frac{L_h}{2}}^{\frac{L_h}{2}}  \eta_i(\gamma(\theta))\, \big( W_i(\varepsilon) - W_i(-\varepsilon) \big)\, d\theta + \int_{-\frac{L_h}{2}}^{\frac{L_h}{2}}  \sum_\kappa \eta_\kappa (\gamma(\theta))\, \Re( V(\varepsilon) -V(-\varepsilon))\, d\theta.
	\end{align*}
	The last equality holds because $n_\kappa=0$ except for the $p_i$. The first integrals have already been computed in \eqref{lacontributionintegree} while we inject \eqref{lacontributiondelacourbe} into the last one to obtain:
	$$\begin{aligned}
		\!\!\!\!\!\!\! \int_{\partial \mathcal{U}_\varepsilon(\Gamma) } \partial_{\vec{n}} \rho\, d\mathrm{vol}_h &= \sum_{i=1}^q n_i\, (2\pi +O(\varepsilon)) + \int_{-\frac{L_h}{2}}^{\frac{L_h}{2}}  \sum_\kappa \eta_\kappa (\gamma(\theta)) \left( \frac{2}{\varepsilon} +O(\varepsilon) \right) d\theta.
	\end{aligned}$$
	From this, one concludes
	\begin{equation}
		\label{enfinlacontributiondetoutelacourbe}
		\int_{\partial \mathcal{U}_\varepsilon(\Gamma) } \partial_{\vec{n}} \rho\, d\mathrm{vol}_h  =\frac{2 L_h}{\varepsilon} + 2\pi \sum_{i=1}^q n_i +  O(\varepsilon).
	\end{equation}

	\subsection{Proof of the Gauss--Bonnet formula}

	We will now prove the Gauss--Bonnet formula.
	
	\begin{theorem}\label{th:GB}
Let $\Sigma$ be a closed Riemann surface. Let $\Psi\colon \Sigma\to\s^3$ be a smooth Willmore immersion not totally umbilic, $\Ur\subset \Sigma$ be the umbilic set of $\Psi$. Let $Y\colon \Sigma\to \s^{3,1}$ be its conformal Gauss map and $K_Y$ the Gauss curvature of $Y$. 
		
		Let $h$ be a smooth metric on $\Sigma$ conformal to $g_\Psi$. Let $L^1_h,\ldots,L^J_h$ the length of the closed umblic curves contained in $\mathcal{L} \subset \Ur$ computed with respect to the metric $h$. We denote $p_1, \dots, p_m$ the singular umbilic points away from the umbilic curves and $p_m+1, \dots, p_q$ those on umbilic curves. Let $n_1,\ldots,n_q$ denote the multiplicities of the singular umbilic points. For each $p_i$, $i\le m$ we consider local centred complex coordinates and denote $\D_{r,i} = \{ |z|(p) < r\}$. Given $\ve>0$, we denote the $\varepsilon$-neighbourhood of $\Ur$ for the metric $h$ by
		$$
		\mathcal{U}_\varepsilon \coloneqq \left( \bigcup_{i=1}^m \D_{\ve,i} \right)\cup \{ p\in \mathcal{L} : d_h(p,\Ur) <\ve \}.
		$$
		Then it holds: 
		$$ 
		\int_{\Sigma\setminus \mathcal{U}_\varepsilon} K_Y\, d\vol_{g_Y} \ust{\ve\to 0}{=} \sum_{k=1}^J \frac{ 2L_h^k }{\ve} + 2\pi\chi(\Sigma) +2\pi \sum_{i=1}^p n_i+ O\big(\varepsilon |\log \varepsilon| \big).
		$$
		This can be reformulated as:
		$$
		\Er(\Psi) = 2\, \lim_{\ve\to 0} \left( \int_{\Sigma\setminus \Ur_{\ve}} \Re\left( 4\Qr\otimes h_0^{-2}\right)\, d\vol_{g_\Psi}  + \sum_{k=1}^J \frac{ 2L_h^k }{\ve} \right)+ 4\pi\chi(\Sigma) +4\pi \sum_{i=1}^p n_i.
		$$
	\end{theorem}
	\begin{proof}
		As mentioned in \Cref{th:structure_umbilic_set}, the umbilic set is a disjoint union of a finite number of closed curves of finite length and singular umbilic points. For $\varepsilon$ small enough (depending on the geometry on the Willmore surface), the set $\mathcal{U}_\varepsilon$ is a disjoint union of coordinate disks $\D_{\varepsilon, i}$ around singular umbilic points of order $m_i$ and tubular neighborhoods $T_j$ around umbilic curves.\\ 
		One can then apply \eqref{formulegaussbonnetintegraleaubord}:
		$$
		\int_{\Sigma\setminus \mathcal{U}_\varepsilon} K_Y\, d\vol_{g_Y}  = 2\pi \chi(\Sigma) + \sum_{i} \int_{\partial B_i}  \partial_\nu \rho\, d\vol_h + \sum_{i} \int_{\partial T_j}  \partial_\nu \rho\, d\vol_h.
		$$
		We have computed in \eqref{contribbranchedumbilic} the contributions of isolated singular umbilic points :
		$$
		\int_{\partial \D_{\ve,i}}  \partial_\nu \rho\, d\vol_h = 2\pi\, n_j +  \left\{ \begin{aligned} &O(\varepsilon) \text{ if } p_i \text{ is of type \ref{typeI}-\ref{typeII}} \\ & O(\varepsilon \log \varepsilon)\text{ if } p_i \text{ is of type \ref{typeIII}}\end{aligned} \right.
		$$
		For each $1\leq j\leq J$, we denote $n_{l_1},\ldots, n_{l_{q_j}}$ the multiplicities of the singular umbilic points lying on the curve $j$. In \eqref{enfinlacontributiondetoutelacourbesharp}, we have shown that the contribution around an umbilic curve is given by:
		$$	
		\int_{\partial T_j } \partial_{\vec{n}} \rho\, d\mathrm{vol}_h  =\frac{2L_h^j}{\varepsilon} + 2\pi \sum_{l=1}^{q_i} n_{l_i} +  O(\varepsilon ).
		$$
		Thus, combining these last two equalities yields the first formula of \Cref{th:GB}.\\
		Using \Cref{Gaussbonnetdansunecarte} ensures that 
		\begin{align*}
			K_Y\, d\mathrm{vol}_{g_Y} =  d\mathrm{vol}_{g_Y} - 4 \Re\left( \mathcal{Q} \otimes h_0^{-2} \right)d\mathrm{vol}_{g_\Psi}.
		\end{align*}
		Thus, we obtain
		\begin{equation}\label{laformulepourlevolumerenormalise}
			\frac{1}{2}\, \mathcal{E}(\Psi) - \int_{\Sigma\setminus \Ur_{\ve}} \Re\left( 4\Qr\otimes h_0^{-2}\right)\, d\vol_{g_\Psi}  = \sum_{k=1}^J \frac{ 2L_h^k }{\ve} + 2\pi\, \chi(\Sigma) +2\pi \sum_{i=1}^p n_i+ O(\varepsilon \log \varepsilon).
		\end{equation}
		This is the second formula of \Cref{th:GB}.
	\end{proof}

	\begin{remark}
		Here the result depends on the chosen metric $h$ conformal to $g_Y$. While the choice was made to use a uniformization metric, any other smooth, compact, non  branched one would have yielded a similar result. One could then have chosen to work entirely with $g_\Psi$ and obtained the expansion 
		$$
		\frac{2 L_{g_\Psi}}{\varepsilon} + 2\pi \sum_{i=1}^q n_i +  O(\varepsilon).
		$$
		This dependency on an arbitrary $g_Y$-conformal metric is due to the nonintegrability of $K_Y\, d\mathrm{vol}_{g_Y}$ caused by umbilic curves, see \Cref{lm:KY_pos_circ}. Indeed, the resulting improper integral then depends on the chosen domains to cover $\Sigma$. One can avoid this dependency on the background metric by controlling the thickness of the tubular neighbourhood around curves in an homogeneous manner:  still considering $\Phi$ defined in a neighbourhood of $\Gamma$ (see \eqref{definitionlediffeo}), we define
		\begin{equation} \label{definitionlevoisinagecompte}
			\mathcal{U}^\sharp_{\varepsilon}(\Gamma) :=\Phi \left( (-L_h \varepsilon, L_h \varepsilon ) \times \s_{\frac{L_h}{2\pi} } \right).
		\end{equation}  
		In this case \eqref{enfinlacontributiondetoutelacourbe} becomes
		\begin{equation}
			\label{enfinlacontributiondetoutelacourbesharp}
			\int_{\partial \mathcal{U}^\sharp_\varepsilon(\Gamma) } \partial_{\vec{n}} \rho\, d\mathrm{vol}_h  =\frac{2}{\varepsilon} + 2\pi \sum_{i=1}^q n_i +  O(\varepsilon).
		\end{equation}
		Summing these contributions then yields twice the number of umbilic curves as the singular term in the expansion of the Gauss--Bonnet formula, instead of the length. 
	\end{remark}

\begin{remark}
From remark \ref{rk:KY_isolated} one can see that the singularities induced by umbilic points of types \ref{typeI} and \ref{typeII} are integrable. Thus one can replace the coordinate disks by geodesic balls around these points without changing the result.  
\end{remark}
	
	\section{Applications to conformally minimal Willmore surfaces}\label{sec:Application}

	The aim of this section is to apply the Gauss--Bonnet formula for the conformal Gauss map to conformally minimal surfaces in the three models: $\R^3$, $\s^3$ and $\mathbb{H}^3$ in order to prove Proposition \ref{pr:Value_Conf_Min}, which we state again here for convenience.
	
	\begin{proposition}\label{pr:Appli}
		Let $\Sigma$ be a closed Riemann surface and $\Phi\colon \Sigma\to \R^3$ be a Willmore immersion. Let $n_1,\ldots,n_p$ be the multiplicities of its singular umbilic points.
		\begin{enumerate}
			\item If $\Phi$ is a conformal transformation of a minimal surface in $\R^3$, then it holds
			\begin{align*}
				\Er(\Phi) = 4\pi \left( \chi(\Sigma) + \sum_{i=1}^p n_i\right).
			\end{align*}
			
			\item If $\Phi$ is a conformal transformation of a minimal surface in $\s^3$, we denote $V_c(3,\Phi)$ the conformal volume of $\Phi$. Then, it holds
			\begin{align*}
				\Er(\Phi) = 2\, V_c(3,\Phi) - 4\pi\, \chi(\Sigma).
			\end{align*}
			
			\item If $\Phi$ is a conformal transformation of a minimal surface $\zeta\colon \Sigma\to \R^3$ such that $\zeta(\Sigma)\cap \{x^3>0\}$ and $\left[ - \zeta(\Sigma)\cap \{x^3<0\}\right]$ are minimal in $\Hd^3$, we denote $\Ur \coloneqq \{x\in \Sigma : \zeta^3(x)=0\}$ the umbilic set of $\zeta$. It holds
			\begin{align*}
				\Er(\Phi) = -2\, \lim_{\ve\to 0}\left[ \int_{ \{d_{g_{\zeta}}(\cdot,\Ur)>\ve\} } \frac{d\vol_{g_{\zeta}} }{(\zeta^3)^2} - \frac{2}{\ve} \Hr^1\big(\zeta(\Sigma)\cap \{x^3=0\}\big) \right] - 4\pi\, \chi(\Sigma).
			\end{align*}
			In the above formula, $\Hr^1$ denotes the 1-dimensional Lebesgue measure.
		\end{enumerate}
	\end{proposition}
	
	\subsection{Conformally minimal surfaces in $\R^3$}
	
	We will use formulas for the Bryant quartic. If $\Phi\colon\Sigma \rightarrow \R^3$, then in local complex coordinates, see for instance \cite[Equation (91)]{marque2021}:
	\begin{equation}
		\label{formulequartiqueexpliciter3}
		\mathcal{Q}_{\Phi} =  \left[ \left( {\varphi}_{z\bz} \varphi  - \varphi_z \varphi_{\bz} \right)e^{-2\lambda} + \varphi^2 \frac{H^2}{4} \right](dz)^4.
	\end{equation} 
	If $\Phi$ is conformally minimal in $\R^3$, there exists a conformal transformation $F$ of $\R^3$  such that the  immersion $f\coloneqq F\circ \Phi$ is minimal, and thus non compact with $\ell$ ends. The only conformal transformation that can change the mean curvature of $f$ is an inversion at a point on the surface. Up to a translation, we can assume that this point it at the origin and thus that $\Phi = \frac{f}{|f|^2}$.
	
	Since $\Phi$ is immersed and a smooth Willmore surface, all ends are planar of multiplicity $1$ (see for instance \cite[Section 4]{bryant1984} or \cite{bibpointremov}), meaning around $a \in \Sigma$,  there exists a vector $v$ such that $|f| \simeq \frac{v}{|z-a|}$. In addition denoting with a $f$ index the relevant quantities for $f$, one  has $H_{f}= 0$. Gauss--Codazzi equations (see \cite[Equation (63)]{marque2021}) ensures that $\partial_{\bz} \varphi  = e^{-2\lambda}\, \partial_z H$, and thus that $\partial_{\bz} \varphi_{f}=0$. Since $\mathcal{Q}$ is a conformal invariant, it holds
	$$
	\mathcal{Q}_{\Phi} = \mathcal{Q}_f = 0.
	$$ 
	Since minimal surfaces in $\R^3$ have no umbilic line, \eqref{laformulepourlevolumerenormalise} yields
	\begin{equation} \label{formuleainterpreter}
		\frac{1}{2}\, \mathcal{E}(\Phi) = 2\pi \left( \chi(\Sigma) + \sum_{i=1}^p n_i \right).
	\end{equation}
	This proves the case (1) in \Cref{pr:Appli}. \\
	
	To interpret it and make contact with known formulas in the study of Willmore surfaces, let us first notice that the $n_i$ denote the orders of singular umbilic points for the compact, regular immersions $\Phi$, not the complete non compact $f$. The difference can be understood with quick computations in a local conformal chart:
	\begin{equation} \label{transfconfR3min}\begin{aligned}
		\Phi_z &= \frac{f_z}{|f|^2} -2 \left\langle f_z , \frac{f}{|f|^2}  \right\rangle\frac{f}{|f|^2}, \\
		\vec{n}_{\Phi} &= \vec{n}_f - 2 \langle \vec{n}_f , f\rangle \frac{f}{|f|^2}, \\
		\varphi_{\Phi} &= \frac{\varphi_f}{|f|^2}.
	\end{aligned}\end{equation}
	Thus, at a regular point $p$ of $f$, the map $|f|$ is bounded and thus the multiplicities of the umbilic points for $\Phi$ and $f$ are the same:  $n_{p,\Phi}= n_{p, f}$. In addition injecting $|f|= a + O(r)$ and $\varphi_f = z^m \tilde a + O(z^{m+1})$ into the third equality of \eqref{transfconfR3min} ensures that these points are of type \ref{typeI}.  However,  if $p$ is an end of $f$, $\varphi_f$ is not defined on $p$. However, it is holomorphic around $p$ and thus meromorphic at $p$, meaning there exists an integer $n_{p,f}$ possibly negative such that $\varphi_f \simeq z^{n_{p,f}}$, then we have
	$$
	\varphi_{\Phi} \simeq \frac{z^{n_{p,f}}}{|f|^2} \simeq \frac{z^{n_{p,f}}}{\left( \frac{1}{r} \right)^2} \simeq z^{n_{p,f}+1}\bz.
	$$
	Thus these points are of type \ref{typeII}, and of multiplicity $n_{p,f} +2 \ge 0$ since $\varphi_\Phi$ is bounded across $p$. Each end of $f$ thus adds $2$ to the total  multiplicity of singular umbilic points of $\Phi$. If $\ell$ is the number of ends of $f$, then the equality \eqref{formuleainterpreter} can then be rephrased as
	$$\frac{1}{2}\, \mathcal{E}(\Phi) = 2\pi \left(\chi(\Sigma) + \sum_{i=1}^p n_{i,f} + 2\, \ell \right).
	$$
	To conclude, we use the fact that since $f$ is minimal, the 2-form $h_0 = \varphi_f dz^2$ is an meromorphic $2$-form. By Riemann--Roch Theorem (see \cite[Proof of Theorem F.3]{tromba1992}), if $\gf$ is the genus of $\Sigma$, it holds
	$$
	\sum_{i=1}^p n_{i,f}= 4\, (\gf-1)=-2\, \chi( \Sigma).
	$$ 
	This implies 
	$$
	\frac{1}{2}\, \mathcal{E}(\Phi) = -2\pi\, \chi(\Sigma) +4\pi\, \ell.
	$$
	Going back to \eqref{eq:EW}, this yields the well known formula for conformally minimal surfaces in $\R^3$ (see for instance \cite{lamm2013,li1982,marque2021}) where $W(\Phi)$ is obtained by taking $4\pi$ times the number of ends of its minimal version: $W(\Phi)  = 4\pi \ell$.

	\subsection{Conformally minimal surfaces in $\s^3$}
	
	If $\Psi\colon \Sigma \to \s^3$, then we have in local complex coordinates, see for instance \cite[Equation (97)]{marque2021}:
	\begin{equation}
		\label{formulequartiqueexplicites3}
		\mathcal{Q}_{\Psi} =  \left[ \big( (\dr^2_{z\bz} \varphi_\Psi)\, \varphi_\Psi\, \varphi_\Psi - (\dr_z \varphi_\Psi)\, (\dr_{\bz} \varphi_\Psi) \big)\, e^{-2\lambda} + \varphi_\Psi^2\, \frac{H_\Psi^2+1}{4}\right](dz)^4.
	\end{equation}
	If $\Psi$ is conformally minimal in $\s^3$, there exists a conformal transformation $F$ of $\s^3$ such that $F\circ \Psi $ is minimal. In this case $H_{F\circ \Psi}=0$, and Gauss--Codazzi equations in $\s^3$ (see for instance \cite[Equation (77)]{marque2021}) ensures that $\partial_{\bz} \varphi_{F\circ \Psi} = 0$. Since $\varphi_{\Psi} e^{-\lambda_\Psi} = \varphi_{F\circ \Psi} e^{-\lambda_{F\circ \Psi}}$ and since $\lambda_\Psi$ and $\lambda_{F\circ \Psi}$ are bounded this yields that all umbilic points of $\Psi$ are of type \ref{typeI} and $n_{p,\Psi} = n_{p, F\circ \Psi}$. Using once more that $\mathcal{Q}$ is a conformal invariant, \eqref{formulequartiqueexplicites3} implies
	$$
	\mathcal{Q}_\Psi = \mathcal{Q}_{F\circ \Psi} = \frac{\varphi_{F\circ \Psi}^2}{4} (dz)^4.
	$$
	Since there are no umbilic lines in minimal surfaces in $\s^3$,  \Cref{th:main} and the conformal invariance of the involved terms yields 
	$$
	\frac{1}{2}\, \mathcal{E}(\Psi) = \frac{1}{2}\, \mathcal{E}(F\circ \Psi) =  \lim_{\varepsilon \rightarrow 0} \left(\int_{\Sigma \backslash \Ur_{\ve}} d\mathrm{vol}_{g_{F\circ\Psi}} \right) + 2\pi\, \chi(\Sigma) + 2\pi \sum_{i} n_{i, F\circ \Psi}.
	$$
	The first term on the right-hand side converges toward the volume of $F\circ \Psi$ in $\s^3$,  while Riemann--Roch Theorem (see for instance the proof of Theorem F.3 in \cite{tromba1992})  applied to the holomorphic two form $h_{0, F\circ \Psi}$ yields $\sum_{i} n_{i, F\circ \Psi} = -2\chi(\Sigma)$. We end up with the following relation:
	$$
	\frac{1}{2}\, \mathcal{E}(\Psi) = \mathrm{Vol}(F \circ \Psi) -2\pi\, \chi(\Sigma).
	$$
	As was noticed in \cite{li1982}, the volume of a minimal immersion in $\s^3$ equals the conformal  volume of the associated conformal class, thus 
	$$
	\frac{1}{2}\, \mathcal{E}(\Psi) = \mathrm{V}_c(3, \Psi) -2\pi\, \chi(\Sigma).
	$$
	Taking a stereographic projection from a point of $\s^3$ outside $\Psi(\Sigma)$ yields a compact Willmore immersion $\Phi$, for which the above formula remains, by conformal invariance of the involved quantities:
	$$
	\frac{1}{2}\, \mathcal{E}(\Phi) = \mathrm{V}_c(3, \Phi) -2\pi\, \chi(\Sigma).
	$$
	This concludes the proof of \textit{(2)}. Injecting \eqref{eq:EW} into the above equality, we recover the value of Willmore surfaces for conformally minimal surfaces in $\s^3$ obtained in \cite{li1982}:
	$$
	W(\Phi) = \mathrm{V}_c(3, \Phi).
	$$

	\subsection{Surface of Babych--Bobenko type}
	
	Let us consider a surface $\zeta\colon \Sigma \to \R^3$ of Babych--Bobenko type, see \Cref{defBBtype}. Discounting the isometry we can then  assume that $\zeta_{\pm} = \zeta_{|\{ \pm \zeta^3 >0\}}$ is minimal $(\mathbb{H}^3_{\pm}, \xi)$. As seen in \eqref{hconversions} and \eqref{arconversions}, we have
	$$
	H_{\pm}  =\zeta^3 \, H + \vec{n}^3, \qquad \varphi_{\pm}\, \zeta^3 = \varphi, \qquad e^{2\lambda_{\pm}} = \frac{e^{2\lambda}}{{\zeta^3}^2},
	$$
	where we denoted $H$, $\varphi$, $\lambda$ and $\vec{n}$ the mean curvature, complex tracefree second fundamental form, conformal factor and Gauss map of $\zeta$ in $\R^3$, and $H_\pm$, $\varphi_\pm$, $\lambda_\pm$  the mean curvature, complex tracefree second fundamental form and conformal factor in $\mathbb{H}^3_\pm$.
	One can then compute  
	$$ 
	\begin{aligned}
		& \left(  {\varphi}_{z\bz}\, \varphi  - \varphi_z\, \varphi_{\bz} \right)e^{-2\lambda} + \varphi^2\,  \frac{H^2}{4}\\[2mm]
		=&\ \Big(( {\zeta^3 \varphi_\pm})_{z\bz}\, ( \zeta^3\varphi_\pm)  - (\zeta^3 \varphi_\pm)_z\, (\zeta^3 \varphi_\pm)_{\bz} \Big)\, \frac{e^{-2\lambda_\pm}}{(\zeta^3)^2} + (\zeta^3 \varphi_\pm)^2\, \frac{H^2}{4} \\[2mm]
		=&\ \frac{H\, \vec{n}^3}{2}\, \zeta^3\, \vp_{\pm}^2 -\zeta^3_z\, \zeta^3_{\bz}\, \varphi_{\pm}^2\, e^{-2\lambda}   + \Big(  \zeta^3\, \zeta^3_z\, (\varphi_\pm)_{\bz}\, \varphi_\pm  + \zeta^3\, \zeta^3_{\bz}\, (\varphi_\pm)_{z}\, \varphi_\pm + (\zeta^3)^2\, (\varphi_{\pm})_{z \bz}\, \varphi_\pm   \\[2mm]
		&\quad - \zeta^3\, \zeta^3_z\, (\varphi_\pm)_{\bz}\, \varphi_\pm  -  \zeta^3\, \zeta^3_{\bz}\, (\varphi_\pm)_{z} \varphi_\pm - (\zeta^3)^2\, (\varphi_\pm)_z\, (\varphi_\pm)_{\bz} \Big)\frac{e^{-2\lambda_\pm}}{(\zeta^3)^2} + \varphi_\pm^2\, \frac{(\zeta^3\, H)^2}{4} \\[2mm]
		=&\  \big( (\varphi_\pm)_{z\bz}\, \varphi_\pm - (\varphi_\pm)_z\, (\varphi_\pm)_{\bz} \big)\, e^{-2\lambda_\pm} + \varphi_{\pm}^2\, \frac{\left(H\, \zeta^3 + \vec{n}^3 \right)^2 }{4} \\[2mm]
		& - \varphi_{\pm}^2\, \frac{\left(\vec{n}^3\right)^2 + (\zeta^3_x)^2\, e^{-2\lambda}+ (\zeta^3_y)^2\, e^{-2\lambda} }{4}.
	\end{aligned}
	$$
	Since $(\vec{n}, \zeta_x e^{-\lambda}, \zeta_y e^{-\lambda})$ is an orthonormal basis of $\mathbb{R}^3$, we obtain
	\begin{align*}
		\left(  {\varphi}_{z\bz}\, \varphi  - \varphi_z\, \varphi_{\bz} \right)e^{-2\lambda} + \varphi^2\,  \frac{H^2}{4} = \big( (\varphi_\pm)_{z\bz}\, \varphi_\pm - (\varphi_\pm)_z\, (\varphi_\pm)_{\bz} \big)\, e^{-2\lambda_\pm} + \varphi_{\pm}^2 \frac{H_\pm^2-1 }{4}.
	\end{align*}
	Thus, around points in $\mathbb{H}_\pm$, the Bryant's quartic can be written:
	\begin{equation}
		\label{BryantquarticBBtype}
		\mathcal{Q}_{\zeta} =  \left[  \big( (\varphi_\pm)_{z\bz}\, \varphi_\pm - (\varphi_\pm)_z\, (\varphi_\pm)_{\bz} \big)\, e^{-2\lambda_\pm} + \varphi_{\pm}^2\, \frac{H_\pm^2-1 }{4}\right](dz)^4.
	\end{equation}
	In addition, Gauss--Codazzi equations for immersions in $\mathbb{H}_\pm$ stands as (see for instance \cite[Proposition 2.1]{schatzle2017}):
	\begin{equation}
		\partial_{\bz} \varphi_{\pm} = e^{2\lambda_\pm}\, \partial_z H_\pm .
	\end{equation}
	Now, since $\zeta$ is of Babych--Bobenko type, $\partial_zH_\pm= 0$ and thus $\varphi_\pm$ are holomorphic.  Therefore, \eqref{BryantquarticBBtype} becomes 
	$$
	\mathcal{Q} = - \frac{\varphi^2_\pm}{4} (dz)^4 = -\frac{1}{4\, (\zeta^3)^2}\, h_0^2.
	$$
	In other words, it holds
	$$ 
	-4\mathcal{Q}_{\zeta}\otimes h_0^{-2}\, d \mathrm{vol}_{g_\zeta} = \frac{d\mathrm{vol}_{g_\zeta}}{(\zeta^3)^2} = d\mathrm{vol}_{g_{\zeta_\pm}}.
	$$
	Injecting this into \eqref{laformulepourlevolumerenormalise} for $h= g_\zeta$ then yields:
	$$\begin{aligned}
		\frac{1}{2}\, \mathcal{E}(\Psi) + \int_{\Sigma\setminus B_{\ve}(\Ur)} d\mathrm{vol}_{g_{\zeta_\pm}} \ust{\ve\to 0}{=} \sum_{k=1}^K \frac{ 2\, L^k }{\ve} + 2\pi\chi(\Sigma) +2\pi \sum_{i=1}^p m_i+ O(\varepsilon).
	\end{aligned}$$
	Equivalently, one has 
	$$
	\begin{aligned}
		\mathcal{A}(\varepsilon) \coloneqq \int_{\Sigma\setminus B_{\ve}(\Ur)} d\mathrm{vol}_{g_{\zeta_\pm}} - 2\sum_{k=1}^K \frac{ L^k }{\ve}  \ust{\ve\to 0}{=} 2\pi\, \chi(\Sigma) +2\pi \sum_{i=1}^p m_i- \frac{1}{2}\, \mathcal{E}(\Psi) + O(\varepsilon).
	\end{aligned}
	$$
	Since $\vp_{\pm}$ are holomorphic, one has 
	\begin{equation}\label{eq:vp_bz}
		\partial_{\bz} \varphi = (\partial_{\bz}\zeta^3 )\, \varphi_{\pm}.
	\end{equation}
	Since the surface is normal to the $\{x^3=0\}$ plane, we also have $\lim_{\zeta^3 \rightarrow 0} \partial_{\bz}\zeta^3 \neq 0$. Hence the function $\varphi_\pm$ can be extended continuously accross the umbilic lines $\{\zeta^3=0\}$. Taking the $\partial_z$ derivative of \eqref{eq:vp_bz}, we obtain that $\partial_z \varphi_{\pm}$ can also be extended, while $\partial_{\bz} \varphi_\pm=0$ extends naturally. Hence, the complex tracefree second fundamental form in $\mathbb{H}_{\pm}$ can then be extended into a holomorphic $2$-form on the whole of $\Sigma$:
	$$
	h_{\pm}=\varphi_{\pm}dz^2.
	$$
	The zeros of $h_{\pm}$ are exactly the singular umbilic points of the surface. Since $\varphi = \varphi_{\pm} \zeta^3$, those away from the $\{\zeta^3=0\}$ lines will be of type \ref{typeI}, while those on the umbilic lines are by definition of type \ref{typeIV}. By Riemann--Roch Theorem (see for instance the proof of Theorem F.3 in \cite{tromba1992}), we have by denoting $\gf$ the genus of $\Sigma$:
	$$
	\sum_{i=1}^p m_i= 4(\gf-1)=-2\, \chi( \Sigma).
	$$ 
	This yields
	$$
	\begin{aligned}
		\mathcal{A}(\varepsilon)   = -2\pi\, \chi(\Sigma) - \frac{1}{2}\, \mathcal{E}(\Psi) + O(\varepsilon).
	\end{aligned}
	$$
	In the above expansion, we recognize Alexakis--Mazzeo's formula for the sum of the  renormalized volume of $\zeta_+$ and $\zeta_-$ (see  \cite[Proposition 2.1]{alexakis2010}), and thus proving the relation \textit{(3)} in Proposition \ref{pr:Appli}.

We conclude by summing up which types of umbilic points appear in conformally minimal surfaces:
\begin{proposition}
If $\Phi :\, \Sigma \rightarrow \R^3$ is a Willmore immersion of a closed Riemann surface $\Sigma$, then
\begin{itemize}
\item if $\Phi$ is a conformal transformation of a minimal surface in $\s^3$ then all its umbilic points are of type \ref{typeI}.
\item if $\Phi$ is a conformal transformation of a minimal surface in $\R^3$, then the ends of  the minimal surface yield umbilic points of type \ref{typeII}, all others are of type \ref{typeI}.
\item if $\Phi$ is conformally a surface of Babich-Bobenko type, then it has umbilic lines, and any singular umbilic point on such a line is of type \ref{typeIV}. All others are of type \ref{typeI}.
\end{itemize}
\end{proposition}

	\appendix
	
	\section{Curvature of the conformal Gauss map}\label{sec:curvatureY}

	In this section, we detail the proof of \Cref{Gaussbonnetpasdansunecarte} by B. Palmer \cite{palmer1991} linking the curvature of the conformal Gauss map and its normal bundle to its Bryant's quartic.
	
	\subsection{Conformal factors and complex notations}\label{sec:Complex_notations}
	In the following we will study immersions of a Riemann surface $\Sigma$ in the $3$-dimensional unit sphere $\Psi\colon \Sigma \to \s^3$. Given any uniformization metric  $h$ for $(\Sigma, g_{\Psi})$, we denote respectively $\lambda$ and $\rho$ the conformal factors of $\Psi$ and $Y$, meaning:
	\begin{align}\label{def:conf_factor}
		\begin{cases}
			\displaystyle \Psi^*\deltar = e^{2\lambda} h,\\[1mm]
			\displaystyle Y^*\eta = e^{2\rho} h.
		\end{cases}
	\end{align}
	When working in local conformal coordinates, we will slightly abuse notations and keep denoting the respective conformal factors $\lambda$ and $\rho$. 
	In the induced complex coordinates in such a local conformal chart, we will denote the differentiation with respect to $z$ or $\bz$ by indices instead of $\dr_z$ or $\dr_{\bz}$. Since $Y$ is conformal, it holds on one hand:
	\begin{align}
		\scal{Y_z}{Y_z}_\eta & = \frac{1}{4} \para{ \scal{Y_x - i Y_y}{Y_x - iY_y}_\eta }= \frac{1}{4} \para{ |Y_x|^2_\eta - |Y_y|^2_\eta -2i \scal{Y_x}{Y_y}_\eta } = 0. \label{eq:Yz_norm}
	\end{align}
	On the other hand, we have
	\begin{align}
		\scal{Y_z}{Y_{\bz} }_\eta &= \frac{1}{4} \scal{Y_x - iY_y}{Y_x + i Y_y}_\eta = \frac{1}{4}\para{ |Y_x|^2_\eta + |Y_y|^2_\eta }
		= \frac{e^{2\rho}}{2}.\label{eq:Ybz_norm}
	\end{align}
	Denoting  $\nu \coloneqq \begin{pmatrix} \Psi \\ 1 \end{pmatrix}$, one has 
	$$
	\langle Y, \nu\rangle_\eta =  \langle Y_z, \nu\rangle_\eta= \langle Y_{\bz}, \nu\rangle_\eta=0.
	$$ 
	The orthogonal family $(Y,Y_z,Y_{\bz})$ can then be completed into a basis of $\R^{4,1}$, where $(\nu, \nu^*)$ is an isotropic basis of $NY$. We refer the reader to \cite{marque2021} for a detailed construction of this basis, with explicit expressions. We denote $\vp$ the complexified tracefree curvature:
	\begin{align}\label{def:varphi}
		\varphi \coloneq \Arond_{11} - i\Arond_{12}.
	\end{align}
	Here we will notice that one can differentiate \eqref{def:CGM}:
	\begin{equation}
		\label{expressionYz}
		Y_z = H_z\, \nu - \varphi\, e^{-2\lambda}\, \nu_{\bz}.
	\end{equation}
	From \eqref{eq:Ybz_norm}, one then recovers: 
	\begin{align}\label{eq:norm_phi}
		e^{2\rho} = |\varphi|^2\, e^{-2\lambda}.
	\end{align}

	\subsection{Structure equations}
	In this section, we assume that $\Psi\colon \Sigma\to \s^3$ is Willmore, thus its conformal Gauss map $Y\colon \Sigma\to \s^{3,1}$ is harmonic. Let $\Ur\subset \Sigma$ be the umbilic set of $\Psi$. The goal of this section is to prove the Gauss equation in \Cref{lm:Gauss_equation} and the Ricci equation in \Cref{lm:Ricci_equation} for $Y\colon \Sigma\setminus \Ur\to \s^{3,1}$. To do so, we first compute in \Cref{pr:structure_equations} the structure equations of $Y$ on a disk $\D$ providing complex coordinates on an open set of $\Sigma\setminus \Ur$. Since $Y$ is harmonic, one has (see for instance \cite{bibeschenburg,palmer1991, marque2021})
	$$
	\langle Y_{z\bz}, \nu \rangle_\eta = \langle Y_{z\bz}, \nu^* \rangle_\eta=0.
	$$
	In addition, \eqref{eq:Yz_norm} and \eqref{eq:Ybz_norm} allow one to compute the $Y_z$, $Y_{\bz}$ and $Y$ components yielding:
	\begin{align}\label{eq:Y_harmonic}
		Y_{z\bz} = -\frac{e^{2\rho}}{2} Y.
	\end{align}
	From \eqref{eq:Yz_norm}, we deduce that
	\begin{align*}
		0= \para{\scal{Y_z}{Y_z}_\eta }_z = 2\scal{Y_{zz}}{Y_z}_\eta.
	\end{align*}
	On the other hand, it holds
	\begin{align*}
		\scal{Y_{zz} }{Y_{\bz}}_\eta &= \para{ \scal{Y_z}{Y_{\bz}}_\eta }_z - \scal{ Y_z}{Y_{z\bz}}_\eta.
	\end{align*}
	Since $Y$ is harmonic,  $\lap Y\perp \g Y$  and the last term vanishes. Using \eqref{eq:Ybz_norm}, we obtain
	\begin{equation}
		\label{YzzYbz}
		\scal{Y_{zz} }{Y_{\bz}}_\eta = \frac{1}{2} \para{ e^{2\rho} }_z = \rho_z\, e^{2\rho}.
	\end{equation}
	We now decompose $Y_{zz}$ in the frame $(Y,Y_z,Y_{\bz}, \nu, \nu^*)$.
	\begin{lemma}\label{lm:Yzz}
		Let $Q =\langle Y_{zz}, Y_{zz}\rangle_\eta$. The following decomposition holds on $\D$:
		\begin{align*}
			Y_{zz} &= 2\rho_z Y_z + \frac{Q}{\vp}\nu + \frac{\vp}{2\scal{\nu}{\nu^*}_\eta } \nu^*.
		\end{align*}
	\end{lemma}
	
	\begin{proof}
		Since $(Y,Y_z, Y_{\bz}, \nu, \nu^*)$ is an orthogonal moving frame in $\R^{4,1}$ defined on $\D$, there exist coefficients $u,a,b,\alpha,\beta\colon \D\to \C$ such that 
		\begin{align*}
			Y_{zz} &= uY+a Y_z + bY_{\bz} + \alpha\nu + \beta\nu^*.
		\end{align*}
		Taking the scalar product with $Y$ yields $u = \langle Y, Y_{zz} \rangle_{\eta} = \partial_z \left( \langle Y, Y_z \rangle_\eta \right) - \langle Y_z, Y_z \rangle_{\eta} = 0$. Taking the  scalar product with $Y_z$ leads to $b\frac{e^{2\rho}}{2} = 0$,  and thus $b=0$. In addition, we deduce from \eqref{YzzYbz} that the scalar product with $Y_{\bz}$ leads to 
		\begin{align*}
			\rho_z\, e^{2\rho} &= a\, \frac{e^{2\rho}}{2}.
		\end{align*}
		Hence, it holds
		\begin{align}\label{eq:decompo_Yzz}
			Y_{zz} &= 2\, \rho_z\, Y_z + \alpha\, \nu + \beta\, \nu^*.
		\end{align}
		Taking the scalar product with $\nu$, we obtain:
		\begin{align*}
			\scal{Y_{zz}}{\nu}_\eta = - \scal{Y_z}{\nu_z}_\eta = \frac{\vp}{2}.
		\end{align*}
		By definition of $\beta$, we have
		\begin{align*}
			\beta &= \frac{\vp}{2\scal{\nu}{\nu^*}_\eta }.
		\end{align*}
		Hence, we write \eqref{eq:decompo_Yzz} as:
		\begin{align*}
			Y_{zz} &= 2\rho_z Y_z + \alpha\nu + \frac{\vp}{2\scal{\nu}{\nu^*}_\eta } \nu^*.
		\end{align*}
		Taking the scalar product with $Y_{zz}$ yields
		\begin{align*}
			Q=\scal{ Y_{zz} }{ Y_{zz} }_\eta &= \scal{ 2\rho_z Y_z + \alpha\nu + \frac{\vp}{2\scal{\nu}{\nu^*}_\eta } \nu^* }{ 2\rho_z Y_z + \alpha\nu + \frac{\vp}{2\scal{\nu}{\nu^*}_\eta } \nu^* }_\eta \\
			&= \frac{ \alpha \vp }{ \scal{\nu}{\nu^*}_\eta } \scal{\nu}{\nu^*}_\eta .
		\end{align*}
		We deduce that $\alpha = \vp^{-1}\, Q$.
	\end{proof}
	To express the curvature of $Y$, equivalently the normal components  of $Y_{zz}$, we will split $\vp$ into its module and its argument: 
	\begin{align*}
		\vp = |\vp|\, e^{i\psi}.
	\end{align*}
	We now consider the vector fields on $\Sigma\setminus \Ur$ 
	\begin{align}\label{def:AB}
		V \coloneq \frac{1}{|\vp|} \nu,\qquad \text{and }\qquad V^* \coloneq -\frac{|\vp|}{2\scal{\nu}{\nu^*}_\eta }\nu^*.
	\end{align}
	These vector fields satisfy
	\begin{align*}
		|V|^2_\eta = |V^*|^2_\eta = 0,\qquad \scal{V}{V^*}_\eta = -\frac{1}{2}.
	\end{align*}
	\begin{proposition}\label{pr:structure_equations}
		The structure equations are given by the following relations in complex coordinates around any point of $\Sigma\setminus \Ur$:
		\begin{align*}
			\left\{ \begin{array}{l}
				\displaystyle Y_{zz} = 2\rho_z Y_z + Q\, e^{-i\psi}\, V - e^{i\psi}\, V^*,\\[2mm]
				\displaystyle  Y_{z\bz} = -\frac{e^{2\rho}}{2} Y,\\[3mm]
				\displaystyle  V_z = - e^{-2\rho + i\psi}\, Y_{\bz} -i\, \psi_z\, V,\\[2mm]
				\displaystyle  V^*_z = Q\, e^{-2\rho - i\psi}\, Y_{\bz}  +i\, \psi_z\, V^*.
			\end{array}
			\right.
		\end{align*}
	\end{proposition}

	We write the decomposition of \Cref{lm:Yzz} in terms of $V$ and $V^*$:
	\begin{align}\label{eq:Yzz_AB}
		Y_{zz} = 2\rho_z Y_z + Q e^{-i\psi} V - e^{i\psi} V^*.
	\end{align}
	Together with \eqref{eq:Y_harmonic} yields the first two equalities of \Cref{pr:structure_equations}. There remains only to study the evolution of the normal frame.
	
	\begin{lemma}
		There exists $\gamma\colon \D\to \C$ such that
		\begin{align}
			V_z &= - e^{-2\rho + i\psi}\, Y_{\bz} + \gamma\, V, \label{eq:Az}\\
			V^*_z &= Q\, e^{-2\rho - i\psi}\, Y_{\bz}  -\gamma\, V^*.\label{eq:Bz}
		\end{align}
	\end{lemma}
	
	\begin{proof}
		\underline{Proof of \eqref{eq:Az}:}\\
		We decompose $V_z$ along the frame $(Y,Y_z,Y_{\bz},V,V^*)$:
		\begin{align*}
			V_z &= u\, Y +  a\, Y_z + b\, Y_{\bz} + c\, V + d\, V^*.
		\end{align*}
		The coefficient $u$ is zero:
		$$
		u = \langle V_z, Y \rangle_\eta = \partial_z \left( \langle V, Y \rangle_\eta \right) - \langle V, Y_z \rangle_\eta = 0.
		$$ 
		Since $V$ is unitary, it holds $\partial_z \left(|V|^2_\eta \right) = 2 \langle V_z , V\rangle_\eta = 0$. We obtain $d=0$. We consider now the product of \eqref{eq:Yzz_AB} with $V$:
		\begin{align*}
			\scal{Y_{zz}}{V}_\eta &= -e^{i\psi} \scal{V}{V^*}_\eta = \frac{e^{i\psi}}{2}.
		\end{align*}
		On the other hand, it holds:
		\begin{align*}
			\scal{Y_{zz}}{V}_\eta &= \para{ \scal{Y_z}{V}_\eta }_z - \scal{Y_z}{V_z}_\eta = - \scal{Y_z}{V_z}_\eta.
		\end{align*}
		Thus, we obtain:
		\begin{align*}
			\scal{Y_z}{V_z}_\eta = -\frac{e^{i\psi}}{2}.
		\end{align*}
		Combined with \eqref{eq:Yz_norm}-\eqref{eq:Ybz_norm}, this yields
		\begin{align*}
			V_z &= a\, Y_z - e^{-2\rho + i\psi}\, Y_{\bz} + c\, V.
		\end{align*}
		Taking the scalar product with $Y_{\bz}$, and using that $Y$ is harmonic, we obtain
		\begin{align*}
			a\, \frac{e^{2\rho}}{2} &= \scal{V_z}{Y_{\bz}}_\eta = \partial_z\left( \langle V, Y_{\bz} \rangle_\eta \right) - \langle V, Y_{z \bz} \rangle_\eta = 0.
		\end{align*}
		We end up with
		\begin{align*}
			V_z &= - e^{-2\rho + i\psi}\, Y_{\bz} + c\, V.
		\end{align*}
		
		\underline{Proof of \eqref{eq:Bz}:}\\
		We decompose $V^*_z$ along the frame $(Y,Y_z,Y_{\bz},V,V^*)$:
		\begin{align*}
			V^*_z &= \mu\, Y + \alpha\, Y_z + \beta\, Y_{\bz} + \gamma\, V + \delta\, V^*.
		\end{align*}
		Working as for the proof of \eqref{eq:Az}, we obtain $ \mu = 0$, $\gamma = 0$ and $\alpha = 0$. The scalar product of \eqref{eq:Yzz_AB} with $V^*$ yields
		\begin{align*}
			-\scal{Y_z}{V^*_z}_\eta = \scal{ Y_{zz} }{V^*}_\eta &= -\frac{Q}{2} e^{-i\psi}.
		\end{align*}
		Thanks to \eqref{eq:Ybz_norm}, we deduce that 
		\begin{align*}
			\beta = Q\, e^{-2\rho - i\psi}.
		\end{align*}
		Thus, it holds:
		\begin{align*}
			V^*_z = Q\, e^{-2\rho - i\psi}\, Y_{\bz}  + \delta\, V^*.
		\end{align*}
		Finally, taking the scalar product with $V$ yields on one hand
		\begin{align*}
			\scal{V^*_z}{V}_\eta = -\frac{\delta}{2}.
		\end{align*}
		On the other hand, we have
		\begin{align*}
			\scal{V^*_z}{V}_\eta = - \scal{V^*}{V_z}_\eta = \frac{c}{2}.
		\end{align*}
		Hence $\delta = -c$, which concludes the proof of \eqref{eq:Az} and \eqref{eq:Bz}.
	\end{proof}
	
	We compute $\gamma$ thanks to the third derivatives of $Y$.
	\begin{lemma}
		It holds
		\begin{align}
			V_z &= - e^{-2\rho + i\psi}\, Y_{\bz} -i\, \psi_z\, V, \label{eq:Az_full}\\
			V^*_z &= Q\, e^{-2\rho - i\psi}\, Y_{\bz}  +i\, \psi_z\, V^*.\label{eq:Bz_full}
		\end{align}
	\end{lemma}
	
	\begin{proof}
		We differentiate \eqref{eq:Y_harmonic} with respect to $z$:
		\begin{align*}
			Y_{z\bz z} &= -\frac{e^{2\rho}}{2}\, Y_z - \rho_z\, e^{2\rho}\, Y.
		\end{align*}
		We differentiate \eqref{eq:Yzz_AB} with respect to $\bz$, keeping in mind that $Q$ is holomorphic:
		\begin{align*}
			Y_{zz\bz} = 2\rho_{z\bz}\, Y_z + 2\rho_z\, Y_{z\bz} -i\psi_{\bz}\, Q\, e^{-i\psi}\, V + Q\, e^{-i\psi}\, V_{\bz} -i\psi_{\bz}\, e^{i\psi}\, V^* - e^{i\psi}\, V^*_{\bz}.
		\end{align*}
		Thanks to \eqref{eq:Y_harmonic} together with \eqref{eq:Az} and \eqref{eq:Bz}, we obtain:
		\begin{equation}
			\label{equationYzzb}
			\begin{aligned}
				-\frac{e^{2\rho}}{2}\, Y_z - \rho_z\, e^{2\rho}\, Y & = 2\, \rho_{z\bz}\, Y_z -\rho_z\, e^{2\rho}\, Y -i\, \psi_{\bz}\, Q\, e^{-i\psi}\, V \\[2mm]
				& \qquad + Q\, e^{-i\psi}\, \para{ - e^{-2\rho - i\psi}\, Y_{z} + \bar{\gamma}\, V } -i\, \psi_{\bz}\, e^{i\psi}\, V^* \\[2mm]
				& \qquad - e^{i\psi} \para{ \bar{Q}\, e^{-2\rho + i\psi}\, Y_{z}  -\bar{\gamma}\, V^* }.
			\end{aligned}
		\end{equation}
		Since the components in $V^*$ vanish, we have:
		\begin{align*}
			0 &= -i\, \psi_{\bz}\, e^{i\psi} +  e^{i\psi}\, \bar{\gamma}.
		\end{align*}
		This ensures that $\gamma =-i\psi_z$.
	\end{proof}
	
	We now write the Gauss equation for the conformal Gauss map $Y$.
	
	\begin{lemma}\label{lm:Gauss_equation}
		The Gauss equation is given by
		\begin{align*}
			2\rho_{z\bz} = 2\Re\left( Q \left[\vp\, e^{-\lambda} \right]^{-2} \right) - \frac{e^{2\rho}}{2}.
		\end{align*}
	\end{lemma}
	
	\begin{proof}
		Looking at the $Y_z$ terms in \eqref{equationYzzb} yields		
		\begin{align*}
			2\rho_{z\bz} & = Q\, e^{-2\rho - 2i\psi} + \bar{Q}\, e^{-2\rho + 2i\psi} - \frac{e^{2\rho}}{2} \\[1mm]
			&=  2\, \Re\left( Q\, e^{-2\rho - 2i\psi} \right) - \frac{e^{2\rho}}{2} \\[1mm]
			&=   2\Re\left( Q\left[ \vp e^{-\lambda} \right]^{-2} \right) - \frac{e^{2\rho}}{2}.
		\end{align*}
		Where we have used that, since $e^{2\rho} = |\varphi|^2 e^{-2\lambda}$, we have 
		$$
		e^{2\rho + 2i\psi} = \left( |\varphi|\, e^{i \psi}\, e^{-\lambda} \right)^2 = \left(\varphi\, e^{-\lambda} \right)^2.
		$$
	\end{proof}
	
	We now write the Ricci equation for the conformal Gauss map $Y$.
	
	\begin{lemma}\label{lm:Ricci_equation}
		The Ricci equation is given by
		\begin{align*}
			\psi_{z\bz} = \Im\left( Q \left[\vp\, e^{-\lambda} \right]^{-2} \right).
		\end{align*}
	\end{lemma}
	
	\begin{proof}
		Using \Cref{pr:structure_equations}, it holds
		\begin{align*}
			V_{z\bz} &= (2\rho_{\bz} - i\psi_{\bz} )\, e^{-2\rho + i\psi}\, Y_{\bz} -e^{-2\rho + i\psi}\, Y_{\bz \bz} -i\psi_{z\bz}\, V -i\psi_z\, V_{\bz} \\[2mm]
			&= (2\rho_{\bz} - i\psi_{\bz} )\, e^{-2\rho + i\psi}\, Y_{\bz} -i\psi_{z\bz}\, V \\[1mm]
			&\qquad -e^{-2\rho + i\psi}\left( 2\rho_{\bz}\, Y_{\bz} + \bar{Q}\, e^{i\psi}\, V - e^{-i\psi}\, V^* \right) \\[2mm]
			&\qquad - i\psi_z\left( -e^{-2\rho -i\psi}\, Y_z + i\psi_{\bz}\, V \right) \\[2mm]
			&= i\left( \psi_z\, e^{-2\rho - i\psi}\, Y_z -\psi_{\bz}\, e^{-2\rho + i\psi}\, Y_{\bz} \right) -\left( i\psi_{z\bz} + \bar{Q}\, e^{-2\rho+2i\psi} - \psi_z\, \psi_{\bz} \right)V + e^{-2\rho}\, V^* \\[2mm]
			&= -2\Im\left( \psi_z\, e^{-2\rho - i\psi}\, Y_z \right) -\left( i\psi_{z\bz} + \bar{Q}\, e^{-2\rho+2i\psi} - \psi_z\, \psi_{\bz} \right)V + e^{-2\rho}\, V^*.
		\end{align*}
		Since $V$ is a real-valued vector field, $\Im(V_{z\bz}) = 0$, which yields
		\begin{align}\label{eq:Im_Azbz}
			0 = \Im\left( i\psi_{z\bz} + \bar{Q}\, e^{-2\rho+2i\psi} - \psi_z\, \psi_{\bz} \right).
		\end{align}
		Since $\psi$ is real-valued, the last term vanishes and $\Im(i \psi_{z\bz})= \psi_{z\bz}$. Thus, we obtain:
		\begin{align*}
			\psi_{z\bz} + \Im\left( \bar{Q}\, e^{-2\rho+2i\psi} \right) = 0.
		\end{align*}
		It means that
		\begin{align*}
			\psi_{z\bz} = - \Im\left( \bar{Q}\, e^{-2\rho+2i\psi} \right) = \Im\left( Q\, e^{-2\rho-2i\psi} \right) =  \Im\left( Q \left[\vp\, e^{-\lambda} \right]^{-2} \right).
		\end{align*}
	\end{proof}
	
	\subsection{Gauss Curvature of a Conformal Gauss Map}
	We can now link the curvatures of the conformal Gauss map to the Gauss Curvatures of the surface, and its normal bundle.
	For the definition of the curvature tensors in semi-Riemannian geometry, see \cite{oneill1983}. In this section, we assume that $\Psi\colon \Sigma\to \s^3$ is Willmore, thus its conformal Gauss map $Y\colon \Sigma\to \s^{3,1}$ is harmonic. Let $\Ur\subset \Sigma$ be the umbilic set of $\Psi$. In \Cref{lm:Kperp}, we compute the curvature of the normal bundle of $Y\colon \Sigma\setminus \Ur\to \s^{3,1}$.
	
	\begin{lemma}\label{lm:Kperp}
		Let $K^\perp_Y$ be the curvature of the normal bundle $Y(\Sigma\setminus \Ur)^\perp$. Consider the vector fields $V$ and $V^*$ defined in \eqref{def:AB} in complex coordinates. Let $e_3\coloneqq V-V^*$ and $e_4\coloneqq V+V^*$. Then, $(e_3,e_4)$ is a frame of of the normal bundle. It holds
		\begin{align*}
			K_Y^\perp\, dx\wedge dy = \frac{i}{2}\, K_Y^\perp\, dz\wedge d\bz = d\left( \scal{de_3}{e_4}_\eta \right).
		\end{align*}
	\end{lemma}
	
	\begin{proof}
		Consider the connections $\g^\perp = \proj_{NY(\Sigma)}\circ d$ on the normal bundle $NY(\Sigma)$, and $\nabla^{\s^{3,1}}$ the Levi-Civita connection on the de Sitter space. We denote $\tilde A$ the normal second fundamental form of $Y$ defined as follows. For any $V \in TY(\Sigma)$ and $X \in NY(\Sigma)$, we consider the following orthogonal decomposition:
		$$
		dX(V) = \langle dX(V), Y \rangle_\eta Y + \nabla^{\s^{3,1}}_V X = \langle dX(V), Y \rangle_\eta Y + \left[ \tilde A (V,X) + \g^\perp_V X \right].
		$$
		The only component of the associated Riemann tensor to $\g^\perp$ (see \cite[Chapter 4, Exercise 11, p.125]{oneill1983} for the definition) which is not vanishing is given by 
		\begin{align*}
			R^\perp(\dr_z,\dr_{\bz},e_3,e_4) = \scal{e_4}{\g^\perp_{\dr_z} \g^\perp_{\dr_{\bz}} e_3 - \g^\perp_{\dr_{\bz}}\g^\perp_{\dr_z} e_3 - \g^\perp_{[\dr_z,\dr_{\bz}]}e_3 }_\eta.
		\end{align*}
		The Ricci equation (see for instance \cite[Chapter 4, Exercise 11, p.125]{oneill1983}) ensures that
		\begin{equation} \label{eq:Ricci}
			\begin{aligned}
				R^\perp(\dr_z,\dr_{\bz},e_3,e_4) & =  \Riem^{\s^{3,1}}(\dr_z,\dr_{\bz},e_3,e_4) \\[2mm]
				& \quad +\langle \tilde A (\partial_z, e_3), \tilde A(\partial_{\bz} , e_4) \rangle_\eta - \langle \tilde A (\partial_{\bz}, e_3), \tilde A(\partial_{z} , e_4) \rangle_\eta.
			\end{aligned}
		\end{equation}
		We compute the last coefficients. Differentiating $\scal{de_3}{e_4}_{\eta}$, we have:
		\begin{align*}
			d\left( \scal{de_3}{e_4}_\eta\right)(\dr_z,\dr_{\bz}) & = d\left(\eta_{\alpha\beta}\, (de_3^{\alpha})\, e_4^{\beta} \right)(\dr_z,\dr_{\bz}) \\[2mm]
			& = \eta_{\alpha\beta}\, (de_3^{\alpha})\wedge (de_4^{\beta}) (\dr_z,\dr_{\bz})\\[2mm]
			&= \scal{ de_3(\dr_z)}{de_4(\dr_{\bz})}_\eta -\scal{de_3(\dr_{\bz})}{de_4(\dr_z)}_\eta.
		\end{align*}
		From \eqref{eq:Az_full}-\eqref{eq:Bz_full}, the vector fields $de_3(\partial_z)$ and $de_4(\partial_z)$ have no $Y$ components and  their normal ones are entirely on $e_4$ and $e_3$ respectively. Thus, we obtain 
		\begin{align*}
			\begin{cases}
				\scal{ de_3(\dr_{\bz})}{de_4(\dr_z)}_\eta = \langle \tilde A (\partial_{\bz}, e_3), \tilde A(\partial_z , e_4) \rangle_\eta, \\[2mm]
				\scal{de_3(\dr_z)}{de_4(\dr_{\bz})}_\eta = \langle \tilde A (\partial_{z}, e_3), \tilde A(\partial_{\bz} , e_4) \rangle_\eta.
			\end{cases}
		\end{align*}
		Hence, it holds
		\begin{align*}
			d\left( \scal{de_3}{e_4}_\eta\right)(\dr_z,\dr_{\bz}) =  \langle \tilde A (\partial_z, e_3), \tilde A(\partial_{\bz} , e_4) \rangle_\eta - \langle \tilde A (\partial_{\bz}, e_3), \tilde A(\partial_{z} , e_4) \rangle_\eta.
		\end{align*}
		Coming back to \eqref{eq:Ricci}, we obtain
		\begin{align*}
			d\left( \scal{de_3}{e_4}_\eta\right)(\dr_z,\dr_{\bz}) =  R^\perp(\dr_z,\dr_{\bz},e_3,e_4) - \Riem^{\s^{3,1}}(\dr_z,\dr_{\bz},e_3,e_4).
		\end{align*}
		Since $\Riem^{\s^{3,1}} = \frac{1}{2} \eta \KN \eta$, where $\KN$ is the Kulkarni--Nomizu product, we have
		$$ 
		\Riem^{\s^{3,1}}(\dr_z,\dr_{\bz},e_3,e_4)=0.
		$$
		Hence, we obtain the announced result:
		\begin{align*}
			d\left( \scal{de_3}{e_4}_\eta\right)(\dr_z,\dr_{\bz}) =  R^\perp(\dr_z,\dr_{\bz},e_3,e_4) = \frac{i}{2} R^\perp(\dr_x,\dr_{y},e_3,e_4)= \frac{i}{2} K_Y^\perp.
		\end{align*}
	\end{proof}
	
	Thanks to the structure equations, we now show that the full intrinsic curvature of $Y$ is contained in the quartic $\Qr$ outside of the umbilic set of $\Psi$.
	
	\begin{lemma}{\cite[Equation (2.19)]{palmer1991}}
		\label{Gaussbonnetdansunecarte}
		Let $K_Y$ by the curvature of $Y(\Sigma\setminus \Ur)$ and $K^\perp_Y$ be the curvature of the normal bundle $Y(\Sigma\setminus \Ur)^\perp$. In local complex coordinates, it holds
		\begin{align*}
			4Q \left[\vp\, e^{-\lambda} \right]^{-2} = \left( 1-K_Y+ iK_Y^\perp\right) e^{2\rho}.
		\end{align*}
		In the above relation, the quantity $\vp$ is defined in \eqref{def:varphi}, $\lambda$ and $\rho$ are defined in \eqref{def:conf_factor}.
	\end{lemma}
	
	\begin{proof}
		The Liouville equation links the Gauss curvature $K_Y$ to the Laplacian of its conformal factor in a conformal chart:
		\begin{align*}
			-4\, \rho_{z\bz} = K_Y\, e^{2\rho}.
		\end{align*}
		From \Cref{lm:Gauss_equation}, we obtain
		\begin{align}\label{eq:Gauss_curvature_Y}
			K_Y\, e^{2\rho} = -4\Re\left( Q  \left[\vp\, e^{-\lambda} \right]^{-2} \right) + e^{2\rho} .
		\end{align}
		As in \Cref{lm:Kperp} and its proof, we use the notations $e_3\coloneqq V-V^*$ and $e_4 \coloneqq V+V^*$ and the normal connection $\nabla^\perp = \proj_{NY(\Sigma)}\circ d$. From \Cref{pr:structure_equations} one gets 
		\begin{align*}
			\proj_{NY(\Sigma)}(\dr_z e_3) = -i\, \psi_z\, (V+V^*) = -i\, \psi_z\, e_4.
		\end{align*}
		In terms of differential forms, this is written as
		\begin{align*}
			\proj_{NY(\Sigma)}(de_3) = *(d\psi)\, e_4.
		\end{align*}
		By \Cref{lm:Ricci_equation}, the Gaussian curvature $K_Y^\perp$ of $NY(\Sigma)$ is given by
		\begin{align*}
			e^{2\rho}\, K_Y^\perp = d^*d\psi = 4\, \psi_{z\bz} = 4\, \Im\left(Q \left[\vp\, e^{-\lambda} \right]^{-2}\right).
		\end{align*}
		Hence, we obtain: 
		\begin{align*}
			4Q \left[\vp\, e^{-\lambda} \right]^{-2} & = 4 \Re \left(Q \left[\vp\, e^{-\lambda} \right]^{-2} \right) + 4\, i\, \Im \left( Q \left[\vp\, e^{-\lambda} \right]^{-2}\right) \\[2mm]
			& = e^{2\rho}\left( 1- K_Y \right) +   e^{2\rho}\, i\, K_Y^\perp.
		\end{align*}
	\end{proof}
	
	Since $Q$ is the $(4,0)$ component of Bryant's quartic $\mathcal{Q} = \langle \partial \partial Y, \partial \partial Y\rangle_{\eta}$  where $d = \partial + \overline{\partial} = \partial_z dz + \partial_{\bz} d\overline{z}$, and $\varphi$ is the $(2,0)$ component of the Weingarten tensor $h_0 = \langle \partial \partial \Psi, N \rangle = \varphi^2 dz^2$, then $\mathcal{Q}$ and $h_0 \otimes h_0$ are two $(4,0)$ tensors defined on $\Sigma$, meaning that $\mathcal{Q} \otimes h_0^{-1} \otimes h_0^{-1}$ is a function defined on $\Sigma \backslash \{ h_0=0 \}$, which is expressed in a chart as $Q \varphi^{-2} e^{2\lambda}$. \Cref{Gaussbonnetdansunecarte} is then  the expression in a local chart of \Cref{Gaussbonnetpasdansunecarte}, which concludes the work in this section. 
	
	\section{Gauss--Bonnet formula}\label{sec:GB_appendix}
	
	In this section, we prove the formula \eqref{formulegaussbonnetintegraleaubord}. Let us then consider a curve $\Gamma \in \Sigma$, parametrized by arc-length (for the metric $h$) by $p(s)$, and of $h$-geodesic curvature $k_h\coloneqq \langle \nabla_{\frac{d}{ds}} \tau(s) , \nu(s) \rangle_{h}$, where $\tau(s) \coloneqq \frac{dp(s)}{ds}$ and $\nu(s)$ is a $h$-unit normal along $\Gamma$. Since $g_Y = e^{2\rho}\, h$ by \eqref{eq:conformal_relations}, if $f$ is a diffeomorphism satisfying $f'(t) = e^{-\rho\circ p\circ f(t)}$, then $\tilde p  \coloneqq p\circ f$ is a $g_Y$ arc-length parametrization of $\Gamma$, with unit tangent vector $\tilde{\tau}$  and unit normal vector defined by
	\begin{align*}
		\begin{cases}
			\tilde \tau (t) = e^{-\rho \circ \tilde{p}(t)}\, \tau (f(t)),\\[2mm]
			\tilde \nu (t) = e^{-\rho \circ\tilde{p}(t)}\, \nu (f(t)).
		\end{cases}
	\end{align*}
	Then, computing the Christoffel symbols of $g_Y$ as a function of those of $h$ and of $\rho$,  the geodesic curvature of $\Gamma$ in $(\Sigma, g_Y)$ is:
	$$
	\begin{aligned}
		k_{g_Y}(t)&=  \langle \nabla_{\frac{d}{dt}} \tilde \tau(t) , \tilde \nu(t) \rangle_{g_Y}  \\[3mm]
		&= e^{2\rho(p(f(t)))} \left\langle  e^{-2\rho (p(f(t)))}  \left[ \frac{d\tau}{ds} (f(t)) +{^{g_Y} \Gamma}^{\, .}_{\, ab}(p(f(t)))\, \tau^a(f(t))\, \tau^b(f(t)) \right] ,  e^{-\rho (p(f(t)))}\, \nu (f(t))  \right\rangle_{h} \\[3mm]
		&=  e^{-\rho (p(f(t)))}\left\langle \nabla_{\frac{d}{ds}} \tau(f(t)) + \big( (\partial_a \rho)\, h^._b + (\partial_b \rho)\, h^._a - (\nabla_h^. \rho)\, h_{ab} \big)(p(f(t)))\, \tau^a(f(t))\, \tau^b(f(t)) , \nu(f(t)) \right\rangle_{h}  \\[3mm]
		&=e^{-\rho(p(f(t)))} \left(k_h(f(t)) - \partial_{\nu(f(t))} \rho(p(f(t))) \right).
	\end{aligned}
	$$
	Thus, integrated on each $\Ur_{\ve}$, \Cref{formuleGaussbonnet} yields 
	\begin{align*}
		\int_{\Ur_{\ve}} K_Y\,  d\mathrm{vol}_{g_Y} &= 2\pi \chi(\Ur_{\ve}) - \int_{\partial \Ur_{\ve}} k_{g_Y}\, d\mathrm{vol}_{g_Y} \\[2mm]
		& = 2\pi \chi(\Ur_{\ve}) - \int_{\partial \Ur_{\ve}} e^{-\rho(p(f(t)))} \left(k_h(f(t)) - \partial_{\nu(f(t))} \rho(p(f(t))) \right)d\mathrm{vol}_{g_Y} \\[2mm]
		&= 2\pi \chi(\Ur_{\ve}) - \int_{\partial \Ur_{\ve}} k_h\, d \mathrm{vol}_h  + \int_{\dr \Ur_{\ve}} \partial_{\nu} \rho\, d\mathrm{vol}_{h} \\[2mm]
		&= \int_{\Ur_{\ve}} K_h\, d\mathrm{vol}_{h}  +\int_{\partial \Ur_{\ve}} \partial_{\nu} \rho\, d\mathrm{vol}_{h}.
	\end{align*}
	Since $h$ is regular on $\Sigma$, we obtain 
	\begin{align*}
		\int_{\Ur_{\ve}} K_h\, d\mathrm{vol}_{h} & \ust{\ve\to 0}{=} \int_{\Sigma} K_h\, d\vol_h + O(\varepsilon) \\[2mm]
		& \ust{\ve\to 0}{=} 2\pi\, \chi(\Sigma) + O(\varepsilon).
	\end{align*}
	Therefore, we obtain \eqref{formulegaussbonnetintegraleaubord}.
	
	\bibliography{biblio.bib}
	\bibliographystyle{plain}
	
\end{document}